\documentclass[UKenglish]{article}

\RequirePackage{amsthm,amsmath,amsfonts,amssymb}
\RequirePackage[numbers]{natbib}
\usepackage{bbm}
\RequirePackage[colorlinks,citecolor=blue,urlcolor=blue]{hyperref}
\RequirePackage{graphicx}
\usepackage{cleveref}
\usepackage{subfigure}
\usepackage{romannum}
\usepackage{romanbar}
\usepackage{todonotes}

\usepackage{geometry}
\numberwithin{equation}{section}

\theoremstyle{plain}
\newtheorem{theorem}{Theorem}

\newtheorem{lemma}[theorem]{Lemma}

\theoremstyle{remark}

\newtheorem{assumption}{Assumption}
\newtheorem{example}[theorem]{Example}
\newtheorem{remark}[theorem]{Remark}

\newcommand{\keywords}[1]{\textbf{\textit{Key words:}} #1}
\newcommand{\mathsubjclass}[1]{\textbf{\textit{Mathematics Subject Classification:}} #1}

\def\rmd{\mathrm{d}}
\def\Id{\mathrm{Id}}
\def\rmV{\mathrm{V}}
\def\rmW{\mathrm{W}}

\def\rc{\mathrm{rc}}
\def\sc{\mathrm{sc}}
\def\Law{\mathrm{Law}}
\def\msh{\mathsf{H}}

\newcommand{\Var}{\mathrm{Var}}

\newcommand{\1}{\mathbbm{1}}

\title{Global contractivity for Langevin dynamics with distribution-dependent forces and uniform in time propagation of chaos}
\author{Katharina Schuh \thanks{Universit\"at Bonn, Institut f\"ur Angewandte Mathematik, Endenicher Allee 60, 53115 Bonn, Germany.\linebreak \href{mailto:katharina.schuh@uni-bonn.de}{katharina.schuh@uni-bonn.de}}} 
\date{}

\begin{document}
\maketitle


\begin{abstract}

We study the long-time behaviour of both the classical second-order Langevin dynamics and the nonlinear second-order Langevin dynamics of McKean-Vlasov type. By a coupling approach, we establish global contraction in an $L^1$ Wasserstein distance with an explicit dimension-free rate for pairwise weak interactions. 
For external forces corresponding to a $\kappa$-strongly convex potential, a contraction rate of order $\mathcal{O}(\sqrt{\kappa})$ is obtained in certain cases.
But the contraction result is not restricted to these forces. 
It rather includes multi-well potentials and non-gradient-type external forces as well as non-gradient-type repulsive and attractive interaction forces.
The proof is based on a novel distance function which combines two contraction results for large and small distances and uses a coupling approach adjusted to the distance. 
By applying a componentwise adaptation of the coupling we provide uniform in time propagation of chaos bounds for the corresponding mean-field particle system. \medskip \\
 \\
\mathsubjclass{60H10, 60J60, 82C31}

\end{abstract}
%

\pagenumbering{arabic}

\section{Introduction}

In this paper, we are interested in the long-time behaviour of the Langevin diffusion $(X_t,Y_t)_{t\geq 0}$ of McKean-Vlasov type on $\mathbb{R}^{2d}$ given by the stochastic differential equation
\begin{equation} \label{eq:KFP}
\begin{cases}
&\rmd \bar{X}_t=\bar{Y}_t \rmd t
\\  &\rmd \bar{Y}_t=( u b(\bar{X}_t)+ u \int_{\mathbb{R}^d}\tilde{b}(\bar{X}_t,z)\bar{\mu}_t^x(\rmd z)-\gamma \bar{Y}_t)\rmd t +\sqrt{2\gamma u }\rmd B_t, \hspace{1cm}  \bar{\mu}_t^x=\Law(\bar{X}_t),
\end{cases}
\end{equation}
where $b:\mathbb{R}^d\to\mathbb{R}^d$ and $\tilde{b}:\mathbb{R}^{2d}\to\mathbb{R}^d$ are two Lipschitz continuous functions, $u,\gamma>0$ are two positive constants and
$(B_t)_{t\geq 0}$ is a $d$-dimensional standard Brownian motion. 
The functions $b$ and $\tilde{b}$ denote the \textit{external force} and the \textit{interaction force}, respectively. 
If $\tilde{b}\equiv0$, \eqref{eq:KFP} corresponds to the classical Langevin dynamics, which is also of particular interest and whose long-time behaviour will separately be studied in detail.
Existence of a solution and uniqueness in law hold provided the initial conditions have bounded second moments and $b$ and $\tilde{b}$ are Lipschitz continuous \cite[Theorem 2.2]{Me96}.

Equation \eqref{eq:KFP} is the probabilistic description of the \textit{Vlasov-Fokker-Planck equation} given by
\begin{equation} \label{eq:KFP_analytic}
\partial_t f_t(x,y)=\nabla_y\cdot\Big[\gamma \nabla_y f_t(x,y)+\gamma y f_t(x,y)+u\Big(b(x) +\int_{\mathbb{R}^d} \tilde{b}(x,z)\bar{\mu}_t^x(\rmd z)\Big)\bar{\mu}_t(x,y)\Big]-u\nabla_x \cdot [y f_t(x,y)],
\end{equation}
where $f_t$ is the time dependent density function on $\mathbb{R}^{2d}$ and $\bar{\mu}_t^x$ is the marginal distribution in the first component of $\bar{\mu}_t(\rmd x \rmd y)=f_t(x,y)\rmd x \rmd y$.
The solution $(f_t)_{t\geq 0}$ of \eqref{eq:KFP_analytic} describes the density function of the process $(\bar{X}_t,\bar{Y}_t)_{t\geq 0}$ which moves according to \eqref{eq:KFP}.
Often, $b$ and $\tilde{b}$ are of the form $b(x)=-\nabla \rmV(x)$ and $\tilde{b}(x,x')=-\nabla_x \rmW(x,x')$ for all $x,x'\in\mathbb{R}^d$ and for some functions $\rmV\in\mathcal{C}^1(\mathbb{R}^d)$ and $\rmW\in\mathcal{C}^1(\mathbb{R}^{2d})$, which are called \textit{confinement potential} and \textit{interaction potential}, respectively.

Besides the long-time behaviour of \eqref{eq:KFP}, we study the  mean-field particle system corresponding to \eqref{eq:KFP} with $N\in\mathbb{N}$ particles which is given by 
\begin{equation} \label{eq:KFP_meanfield_nongradient}
\begin{cases}
&\rmd X_t^{i,N}=Y_t^{i,N} \rmd t
\\ &\rmd Y_t^{i,N}=(u b(X_t^{i,N})+N^{-1}\sum_{j=1}^N u  \tilde{b}(X_t^{i,N},X_t^{j,N})-\gamma Y_t^{i,N})\rmd t +\sqrt{2\gamma u} \rmd B_t^{i}, \hspace{5mm} i=1,...,N.
\end{cases}
\end{equation}

We are interested in establish conditions on $b$ and $\tilde{b}$ such that for all $t\ge 0$ for $N\to\infty$ the law of the particles converges to the law of $(\bar{X}_t,\bar{Y}_t)$. This phenomenon was stated under the name \textit{propagation of chaos} and was first introduced by Kac for the Boltzmann equation in \cite{Ka56}. For finite time horizon, bounds on the difference between the law of the particle system and the law of $N$ independent solutions to \eqref{eq:KFP} are established by McKean \cite{Mc66} provided $b$ and $\tilde{b}$ are Lipschitz continuous and bounded. This result is further developed in e.g. \cite{Sz91,Me96}, see \cite{ChDi21,ChDi22} for a overview and the references therein.

The equations \eqref{eq:KFP}, \eqref{eq:KFP_analytic}, \eqref{eq:KFP_meanfield_nongradient} and its variants have various applications in physics. 
If $\tilde{b}\equiv0$, the solution of \eqref{eq:KFP} can be interpreted as a particle having a position $\bar{X}_t$ and a velocity $\bar{Y}_t$ and which moves according to the external force. 
The constant $\gamma>0$ corresponds to the friction parameter and $u>0$ denotes the inverse of the mass per particle.
Equation \eqref{eq:KFP_meanfield_nongradient} describes many particles whose moves are additionally determined by pairwise interactions given by the interaction force. 
Equation \eqref{eq:KFP_analytic} describes the limit distribution as the number of particles tends to infinity.

In the deep learning community, Langevin dynamics with a mean-field interaction provide a tool to prove trainability of neural networks \cite{MeMoNg18,RoVa18}.
Algorithms using Langevin dynamics have a better long-time behaviour compared to the overdamped Langevin dynamics \cite{ChChBaJo17,DaRi20}, which forms a degenerated special case of the Langevin dynamics, where the limit for $\gamma$ to infinity is taken \cite[Section 6.5.1]{Pa14}.
Therefore, nonlinear Langevin dynamics became recently popular for training networks as the Generative Adversarial Network (GAN) \cite{KaReTaYa20}.

If $\tilde{b}\equiv 0$ and $b=\nabla \rmV$, then under some mild conditions on $\rmV$ the unique invariant measure is given by the Boltzmann-Gibbs distribution 
\begin{align*}
\mu_{\infty}(\rmd x \ \rmd y)\propto \exp(-\rmV(x)-|y|^2/(2u)),
\end{align*}
see e.g. \cite[Proposition 6.1]{Pa14}.
Otherwise, i.e., if $b$ is not of gradient-type or $\tilde{b}\neq 0$, it is often not clear if uniqueness of an invariant probability measure holds (see \cite{DuTu16}) and how fast the marginal law of a solution of \eqref{eq:KFP} converges towards it.

Getting a clear picture of the long-time behaviour of processes given by stochastic differential equations with and without nonlinear forces of McKean-Vlasov type is of wide interest and the objective of many works.
For the overdamped Langevin dynamics forming a first-order equation, the long-time behaviour is studied using both analytic approaches as functional inequalities (e.g. \cite{BaGeLe14,BoGeGu12}) and probabilistic approaches as coupling techniques. 
Via a reflection coupling, Eberle \cite{Eb16} established contraction in $L^1$ Wasserstein distance with respect to a carefully aligned distance function with explicit rates for locally non-convex potentials.
For the dynamics with an additional nonlinear drift term, which appears to model for example granular media (see \cite{BeCaCaPu98}), exponential convergence rates have been investigated for uniformly convex potentials in \cite{CaMcVi03} using gradient flow structure, Logarithmic Sobolev inequalities and transportation cost inequalities 
(see \cite{CaMcVi06,Ma01,CaGuMa08} for relaxations to certain non-uniformly convex potentials).
 Further, \cite{Ma01,CaGuMa08} provide uniform in time propagation of chaos estimates for the corresponding particle system. Based on a coupling approach consisting of a mixture of a synchronous and a reflection coupling, uniform in time propagation of chaos is shown in \cite{DuEbGuZi20} for possibly non strongly convex confinement potentials and possibly non-convex interaction potentials. For the unconfined dynamics (i.e., $b=0$) exponential convergence is studied in \cite{CaGuMa08,BoGeGu13} for convex interaction potentials applying analytic tools. If the convexity assumption on the interaction potential is removed, exponential convergence and propagation of chaos can still be established for unconfined overdamped Langevin dynamics via a sticky coupling approach (see \cite{DuEbGuSc21})  for a class of interaction forces that split in a linear term and a perturbation part.

Proving contraction rates for second-order SDEs given by \eqref{eq:KFP} is more delicate as additionally one has to deal with the hypoellipticity of the diffusion. In the case of the classical Langevin dynamics with a gradient-type force, i.e., when $b=\nabla \rmV$ and $\tilde{b}\equiv0$ hold, exponential convergence is studied in e.g.  \cite{AcArSt15,DoMoSc09,DoMoSc15,HeNi05,HeNi04,He07,Vi09} using analytic methods including the Witten Laplacian, semigroups, functional inequalities and hypocoercivity. 
To our knowledge, the best-known contraction rate is obtained for $\kappa$-strongly convex potentials $\rmV$ in \cite{CaLuWa19}, where contraction in $L^2$ distance is shown with a rate of order $\mathcal{O}(\sqrt{\kappa})$ via a Poincaré type inequality.
Harris type theorems, involving a Lyapunov drift condition, provide a probabilistic technique to analyse the long-time behaviour of Langevin dynamics, see \cite{BaCaGu08,Wu01, MaStHi02,Ta02}. 
An alternative powerful probabilistic approach, which provides quantitative rates, is based on couplings. Via a synchronous coupling approach, Dalalyan and Riou-Durand \cite{DaRi20} showed contraction in Wasserstein distance with rate of order $\mathcal{O}(\kappa/\sqrt{L})$ for $\kappa$-strongly convex potentials with $L$-Lipschitz continuous gradients if $L\gamma^{-2}u\le 1$ holds.
In \cite{EbGuZi19}, Eberle, Guillin and Zimmer introduced a coupling for the Langevin dynamics including non-convex confinement potentials and showed exponential convergence with explicit rates. 
There, contraction is shown in a specific $L^1$ Wasserstein distance with respect to a semimetric involving a Lyapunov function.
More precisely, for large distances, a synchronous coupling is considered and the Lyapunov function in the semimetric yields contraction. For small distances, the noise is synchronized on a line, where contraction for the position is observed, and reflected otherwise to force the dynamics to return to that line. Combining the results of the different areas, contraction in average is obtained for a carefully aligned semimetric. Due to the Lyapunov function, the contraction rate depends on the dimension and the semimetric is not applicable for nonlinear Langevin dynamics, which suggests getting rid of the Lyapunov function and treating the area of large distances differently.

To get results on the long-time behaviour for nonlinear  Langevin diffusions given by \eqref{eq:KFP}, we have to handle both the difficulties coming from the nonlinearity and the hypoellipticity of the equation.
Beginning with the analytic approaches, let us mention the work by Villani \cite{Vi09}, where the hypocoercivity is extended to the framework on the torus with small interactions, see also the work by Bouchut and Dolbeault \cite{BoDo95}.
Using a free energy approach, convergence to equilibrium is studied in \cite{DuTu18} for specific non-convex confining potentials and convex polynomial interaction potentials. 
Applying functional inequalities for mean-field models, established in \cite{GuliWuZh21} to prove convergence to equilibrium in weighted Sobolev norm, Monmarch\'e and Guillin proved propagation of chaos for \eqref{eq:KFP_meanfield_nongradient} in \cite{Mo17,GuMo21}. There, they considered both strongly convex confinement potentials and more general confinement potentials and attractive interaction potentials with at most quadratic growth. 

Coupling techniques are also employed in the study of the nonlinear dynamics \eqref{eq:KFP}. 
In \cite{BoGuMa10}, convergence to equilibrium is shown via a synchronous coupling for small Lipschitz interactions and a quadratic-like friction term. 
The combination of the coupling approach of \cite{EbGuZi19} and a Lyapunov function is used in \cite{KaReTaYa20} to prove exponential contraction in the case of certain small mean-field potentials of non-convolution-type. There, the results are applied to the numerical discretized version of the dynamics corresponding to the Hamiltonian Stochastic Gradient Descent, and the connection to the analysis of deep neural networks is drawn, see \cite{HuReSiSz19} for further references on the connection to deep learning.
Very closely related to this work is the recent preprint \cite{GuLeMo21} by Guillin, Le Bris and Monmarch\'e, which has been prepared independently in parallel.
They considered non-globally convex confinement potentials and Lipschitz continuous even interaction potentials and extended the approach by \cite{EbGuZi19}. More precisely, they modified the semimetric by a sophisticated Lyapunov function to treat the nonlinear Langevin dynamics and to obtain propagation of chaos bounds. 
The main differences between this work and \cite{GuLeMo21} are 
that here we include forces that are not necessarily of gradient type and that we establish global contractivity with dimension-free rates by constructing a novel distance function and modifying the coupling approach of \cite{EbGuZi19} appropriately. 
In particular, we consider two separate metrics $r_l$ and $r_s$ for large and small distances instead of a semimetric involving a Lyapunov function and establish contraction for both metrics separately. For small distances we make use of the results by \cite{EbGuZi19}, whereas for large distances we consider a twisted $2$-norm structure for the metric $r_l$ of the form $(x\cdot(Ax)+ x\cdot(B y)+y\cdot(Cy)$ with positive definite matrices $A,B,C\in\mathbb{R}^{d\times d}$. This structure is similar to the structure appearing in the Lyapunov function in \cite{MaStHi02,Ta02} and to the norm used in e.g. \cite{AcArSt15} to prove contraction for certain strongly convex potentials.

Then, our first main contribution is a global contraction result in Wasserstein distance with respect to a distance $\rho$ that is carefully glued of $r_s$ and $r_l$ and that is equivalent to the Euclidean distance. 
More precisely, we impose $b$ to be a sum of a linear function $-K x$, where $K\in\mathbb{R}^{d\times d}$ is a positive definite matrix with smallest eigenvalue $\kappa$, and a certain Lipschitz continuous function $g(x)$ with Lipschitz constant $L_g$  which is such that $b$ includes gradients of asymptotically strongly convex potentials. 
If the friction parameter $\gamma$ is sufficiently large, i.e., $\gamma^2>2L_g^2u/\kappa$, and if the Lipschitz constant $\tilde{L}$ of the interaction force $\tilde{b}$ is sufficiently small, we prove for two probability measures $\bar{\mu}_0$ and $\bar{\nu}_0$ on $\mathbb{R}^{2d}$ with finite second moment, 
\begin{align} \label{eq:contraction_intro}
\mathcal{W}_\rho(\bar{\mu}_t,\bar{\nu}_t)\leq e^{-ct}\mathcal{W}_\rho(\bar{\mu}_0,\bar{\nu}_0),\qquad \text{and } \qquad \mathcal{W}_1(\bar{\mu}_t,\bar{\nu}_t)\leq M_1e^{-ct}\mathcal{W}_1(\bar{\mu}_0,\bar{\nu}_0),
\end{align}
where $\bar{\mu}_t$ and $\bar{\nu}_t$ are the laws of the solutions $(\bar{X}_t,\bar{Y}_t)$ and $(\bar{X}_t',\bar{Y}_t')$ to \eqref{eq:KFP} with initial distribution $\bar{\mu}_0$ and $\bar{\nu}_0$, respectively.
The dimension-free constants $c$ and $M_1$ depend on $\kappa$, $\gamma$, $u$, on the largest eigenvalue of $K$ and on properties of $g$. Note that the additional constant $M_1$ in the second bound measures the difference between the distance $\rho$ and the Euclidean distance. 

These bounds are established using a modification of the coupling introduced in \cite{EbGuZi19},
which is a synchronous coupling for large distances and mainly a reflection coupling for small distances except on one line the noise is synchronized. In this work, we adjust the transition from synchronous coupling for large distances to reflection coupling for small distances to suit the underlying distance function. Namely, the synchronous coupling is applied when $r_l$ is considered and the coupling approach of \cite{EbGuZi19} when $r_s$ is considered.

This approach which does not rely on a Lyapunov function has the advantage that the upper bound in \eqref{eq:contraction_intro} depends only on the Wasserstein distance between the two initial distributions and is independent of the two distributions themselves (cf. \cite{EbGuZi19,KaReTaYa20,GuLeMo21}).
Further, the metric $r_l$ is chosen such that the rate of the contraction result for large distances is optimized up to a constant. 
We emphasize that these bounds give also global contractivity for the classical Langevin dynamics and improve the result obtained in \cite{EbGuZi19}.
 
Moreover, using the ansatz for large distances, we contribute to the analysis of the optimal contraction rate for strongly convex potentials and improve the results of \cite{DaRi20}. If the drift corresponds to a $\kappa$-strongly convex potential, we can split $\rmV$ in a linear part $x\cdot (Kx)$, where $K$ is a positive definite matrix with smallest eigenvalue $\kappa$, and a convex function $G$ with $L_G$ Lipschitz continuous gradients. We prove contraction in Wasserstein distance with respect to a distance function of the same form as $r_l$ with rate $c=\gamma/2\min(1/4,\kappa u \gamma^{-2})$ provided $L_G u \gamma^{-2} \le 3/4$ holds. If the perturbation $G$ is sufficiently small, i.e., $L_G\le 3 \kappa$, we obtain for optimized $\gamma$ a rate of order $\mathcal{O}(\sqrt{\kappa})$, that coincides with the order given in the $L^2$ contraction result in \cite{CaLuWa19}, and otherwise we obtain a rate of the same order as in \cite{DaRi20}. 

Finally, applying a componentwise version of the preceding coupling we establish a uniform in time propagation of chaos bound for the corresponding particle system \eqref{eq:KFP_meanfield_nongradient}, i.e., we show for a probability measure $\mu_0$ on $\mathbb{R}^{2d}$ with finite second moment,
\begin{align*}
\mathcal{W}_{1,\ell^1_N}(\bar{\mu}_t^{\otimes N}, \mu_t^N)\leq C_1 c^{-1} N^{-1/2},
\end{align*}
where $\mu_t^N$ is the law of the particles driven by \eqref{eq:KFP_meanfield_nongradient} with initial distribution $\mu_0^N=\mu_0^{\otimes N}$ and $\bar{\mu}_t^{\otimes N}$ is the product law of $N$ independent solutions to \eqref{eq:KFP} with initial distribution $\mu_0$.
Here, $C_1$ is a constant depending on $\kappa$, $\gamma$, $u$, $d$, on properties of $g$, and on the second moment of $\mu_0$. The normalized $\ell^1$-distance $\ell_N^1$ is given by 
\begin{align} \label{eq:l_N^1}
\ell_N^1((x,y),(\bar{x},\bar{y}))=N^{-1}\sum_{i=1}^N (|x^i-\bar{x}^i|+|y^i-\bar{y}^i|), \qquad \text{for all } x,y,\bar{x},\bar{y}\in\mathbb{R}^{Nd},
\end{align}
where $|\cdot|$ denotes the Euclidean metric.

Eventually, we note that the construction of the metric for large distance can be applied to prove contraction to specific unconfined cases, where $b\equiv 0$ and $\tilde{b}$ is a small perturbation of a linear force. 

\paragraph{Notation:}
For some space $\mathbb{X}$, which is here either $\mathbb{R}^{2d}$ or $\mathbb{R}^{2Nd}$, we denote its Borel $\sigma$-algebra by $\mathcal{B}(\mathbb{X})$. The space of all probability measures on $(\mathbb{X},\mathcal{B}(\mathbb{X}))$ is denoted by $\mathcal{P}(\mathbb{X})$. 
Let $\mu,\nu\in\mathcal{P}(\mathbb{X})$. A coupling $\omega$ of $\mu$ and $\nu$ is a probability measure on $(\mathbb{X}\times \mathbb{X},\mathcal{B}(\mathbb{X})\otimes\mathcal{B}(\mathbb{X}))$ with marginals $\mu$ and $\nu$. The  $L^p$ Wasserstein distance with respect to a distance function $d:\mathbb{X}\times\mathbb{X}\to\mathbb{R}$ is defined by
\begin{align*}
\mathcal{W}_{p,d}(\mu,\nu)=\inf_{\omega \in\Pi(\mu,\nu)}\Big(\int_{\mathbb{X}\times\mathbb{X}}d(x,y)^p\omega(\rmd x \rmd y)\Big)^{1/p},
\end{align*}
where $\Pi(\mu,\nu)$ denotes the set of all couplings of $\mu$ and $\nu$.
We write $\mathcal{W}_p$ if the underlying distance function is the Euclidean distance. 

\paragraph{Outline of the paper:}
In \Cref{sec:mainresults}, we state the contraction results for the classical Langevin dynamics and give an informal construction of the coupling and the metric.
In \Cref{sec:contr_nonlinear}, we state the framework and the contraction results for Langevin dynamics of McKean-Vlasov type before defining rigorously the metric and the coupling approach in \Cref{sec:couplingandmetric}. 
Uniform in time propagation of chaos is established in \Cref{sec:unif_prop}.
The proofs are postponed to \Cref{sec:proof}.

\section{Contraction for classical Langevin dynamics} \label{sec:mainresults}

\subsection{Contraction for Langevin dynamics with strongly convex confinement potential}\label{sec:contr_strongconv}

First, we consider the Langevin dynamics without a non-linear drift and with confinement potential $\rmV$ given by the stochastic differential equation
\begin{equation} \label{eq:underdampLang_pot}
\begin{cases}
&\rmd X_t = Y_t \rmd t,
\\ &\rmd Y_t = (-\gamma Y_t-u \nabla \rmV(X_t))\rmd t +\sqrt{2\gamma u}\rmd B_t,
\end{cases} 
\end{equation}
with initial condition $(X_0,Y_0)=(x,y)\in\mathbb{R}^{2d}$ and with $d$-dimensional standard Brownian motion $(B_t)_{t\ge 0}$.
We impose for $\rmV\in\mathcal{C}^2(\mathbb{R}^d)$:

\begin{assumption}\label{ass:conf_pot}  
There exist a positive definite matrix $K\in\mathbb{R}^{d\times d}$ with smallest eigenvalue $\kappa>0$ and a convex function $G:\mathbb{R}^d\to\mathbb{R}$ with $L_G$-Lipschitz continuous gradients, i.e., 
\begin{align} 
&\langle \nabla G(x)- \nabla G(\bar{x}),x-\bar{x}\rangle \ge 0 && \text{ and }  \label{eq:strongconv_pot}
\\ & |\nabla G(x )-\nabla G(\bar{x})| \le L_G|x-\bar{x}| && \text{ for all } x, \bar{x}\in\mathbb{R}^d, \nonumber
\end{align} such that
\begin{align*}
\rmV(x) = x \cdot (Kx)/2+G(x) \qquad \text{ for any } x \in\mathbb{R}^d.
\end{align*}
\end{assumption}

We note that \Cref{ass:conf_pot} is satisfied for all $\kappa$-strongly convex functions $\rmV$ with  $L_V$-Lipschitz continuous gradients, i.e., 
\begin{align*}
&\langle \nabla \rmV(x)-\nabla \rmV(y), x-y \rangle \ge \kappa|x-y|^2 && \text{and}
\\ & |\nabla \rmV(x )-\nabla \rmV(y)| \le L_V|x-y| && \text{for all } x, y\in\mathbb{R}^d.
\end{align*}
Note that the splitting of $\rm V$ in $K$ and $G$ is in general not unique. A natural choice is given by $K=\kappa \mathrm{Id}$ and $G(x)=\rmV(x)-(\kappa/2)|x|^2$, where $\mathrm{Id}$ is the $d\times d$ identity matrix. As we see later, we often want a splitting of $\rmV$ such that the Lipschitz constant $L_G$ is minimized.

We establish a global contraction result for \eqref{eq:underdampLang_pot} in $L^p$ Wasserstein distance with respect to the distance function $r:\mathbb{R}^{2d}\times\mathbb{R}^{2d}\to[0,\infty)$ given by
\begin{align}\label{eq:dist_r}
r((x,y),(\bar{x},\bar{y}))=\gamma^{-2}u(x-\bar{x})\cdot(K(x-\bar{x}))+\frac{1}{2}|(1-2\lambda)(x-\bar{x})+\gamma^{-1}(y-\bar{y})|^2+\frac{1}{2}\gamma^{-2}|y-\bar{y}|^2
\end{align}
for $(x,y),(\bar{x},\bar{y})\in\mathbb{R}^{2d}$ with
\begin{align} \label{eq:lambda}
\lambda=\min(1/8,\kappa u \gamma^{-2}).
\end{align}

\begin{theorem}[Contractivity for strongly convex potentials]\label{thm:contr_strongconv}
For $t\ge 0$, let $\mu_t$ and $\nu_t$ be the law at time $t$ of the processes $(X_t, Y_t)$ and $(X_t',Y_t')$, respectively, where $(X_s,Y_s)_{s\ge 0}$ and $(X_s',Y_s')_{s\ge 0}$ are solutions to \eqref{eq:underdampLang_pot} with initial distributions $\mu_0$ and $\nu_0$ on $\mathbb{R}^{2d}$, respectively. Suppose \Cref{ass:conf_pot} holds and 
\begin{align} \label{eq:cond_LG}
L_G u \gamma^{-2} \le 3/4.
\end{align}
Then, for any $1\le p <\infty$
\begin{align*}
\mathcal{W}_{p,r}(\mu_t, \nu_t)\le e^{-ct} \mathcal{W}_{p,r}(\mu_0,\nu_0) \qquad \text{ and } \qquad \mathcal{W}_{p}(\mu_t, \nu_t)\le M e^{-ct} \mathcal{W}_{p}(\mu_0,\nu_0),
\end{align*}
where the contraction rate $c$ is given by
\begin{align} \label{eq:rate_strongconv}
c=\gamma \lambda=\min(\gamma/8, \kappa u \gamma^{-1}/2).
\end{align}
The constant $M$ is given by 
\begin{align} \label{eq:M_strongconv}
M=\sqrt{\max(uL_K+\gamma^2, 3/2)\max(1/(u\kappa),2)},
\end{align}
where $L_K$ denotes the largest eigenvalue of $K$.
\end{theorem}
\begin{proof}
The proof is based on a synchronous coupling and is postponed to \Cref{sec:proof_strongconv}.
\end{proof}

\begin{remark} 
If $\rmV$ is a quadratic function, then $L_G=0$ and the restriction on $\gamma$ vanishes. In this case, the $L^2$ spectral gap of the corresponding generator is given by
\begin{align*}
c_{\mathrm{gap}}=(1-\sqrt{(1-4\kappa u\gamma^{-2})^+})(\gamma/2),
\end{align*}
cf., \cite[Section 6.3]{Pa14}. More precisely, $c_{\mathrm{gap}}=\gamma/2$ if $4\kappa u\gamma^{-2}\geq 1$, and $\kappa u\gamma^{-1}\leq c_{\mathrm{gap}}\leq 2 \kappa u\gamma^{-1}$ if $4\kappa u\gamma^{-2}< 1$. Hence, the contraction rate is of the same order as the spectral gap. In particular for $\gamma=2\sqrt{\kappa u}$ the optimal contraction rate $c=\sqrt{\kappa u}/8$ is obtained.
If $L_G\le 3 \kappa$, $\gamma=2\sqrt{\kappa u }$ satisfies condition \eqref{eq:cond_LG} and yields the optimal contraction rate of order $\mathcal{O}(\sqrt{\kappa})$. Otherwise, for $\gamma=\sqrt{(4/3)L_Gu}$ the contraction rate is optimized and of order $\mathcal{O}(\kappa/\sqrt{L_G})$.
\end{remark}

\subsection{Framework for classical Langevin dynamics with general external forces}

Next, we consider the classical Langevin dynamics  $(X_t,Y_t)_{t\geq 0}$ with a general external drift given by the stochastic differential equation
\begin{equation} \label{eq:underdampLang}
\begin{cases}
&\rmd X_t = Y_t \rmd t,
\\ &\rmd Y_t = (-\gamma Y_t+u b(X_t))\rmd t +\sqrt{2\gamma u}\rmd B_t,
\end{cases} 
\end{equation}
with initial condition $(X_0,Y_0)=(x,y)\in\mathbb{R}^{2d}$.

We impose the following assumption on the force $b$:
\begin{assumption}\label{ass:b} The function $b:\mathbb{R}^d\to \mathbb{R}^d$ is Lipschitz continuous and there exist a positive definite matrix $K\in\mathbb{R}^{d\times d}$ with smallest eigenvalue $\kappa\in(0,\infty)$ and largest eigenvalue $L_K\in(0,\infty)$, a constant $R\in[0,\infty)$ and a function $g:\mathbb{R}^d\to \mathbb{R}^d$ with Lipschitz constant $L_g\in(0,\infty)$ such that 
\begin{align}
b(x)=-K x+ g(x) \qquad \text{ for all }x \in\mathbb{R}^d,
\end{align}
and 
\begin{align} \label{eq:strongconv}
\langle g(x)-g(\bar{x}),x-\bar{x}\rangle \leq 0 \qquad \text{ for all } x, \bar{x}\in\mathbb{R}^d \text{ such that } |x-\bar{x}|\ge R.
\end{align}
\end{assumption}

\begin{remark} Suppose that $b=-\nabla V$ where $V$ is a potential function 
with a $L_V$-Lipschitz continuous gradient and that is $k$-strongly convex outside a Euclidean ball of radius $\tilde{R}$, i.e.,
\begin{align*}
\langle \nabla \rmV(x)-\nabla \rmV (\bar{x}),x-\bar{x}\rangle \geq k|x-\bar{x}|^2 \qquad \text{ for all } x, \bar{x}\in\mathbb{R}^d \text{ such that } |x|,|\bar{x}|\ge \tilde{R}.
\end{align*}
Note that $\nabla \rmV$ can be split in $\nabla \rmV(x)=kx+h(x)$ where $h:\mathbb{R}^d\to\mathbb{R}^d$ is an $L_h$-Lipschitz continuous function with $L_h\leq L_V+k$ and $\langle h(x)-h(\bar{x}),x-\bar{x}\rangle \geq 0$ for all $ x, \bar{x}\in\mathbb{R}^d$ such that $|x|,|\bar{x}|\ge\tilde{R}$.
Then for $ l \le \frac{1}{2}\min(1, \frac{L_h}{k})$, $b=-\nabla \rmV$ satisfies \Cref{ass:b} with $L_g\leq L_V+(1-l)k$, $\kappa=(1-l)k\ge \max(\frac{1}{2}k, k-\frac{L_h}{2})$ and $R=2\tilde{R}\frac{L_h}{l k}$.

\end{remark}

\begin{example}[Double-well potential]\label{example1} For $\beta>0$, we consider the double-well potential $\rmV\in\mathcal{C}^1(\mathbb{R})$ defined by 
\begin{align} \label{eq:doublewell}
\rmV (x)=\begin{cases} \beta \Big(\frac{|x|^4}{4}-\frac{|x|^2}{2}\Big) & \text{for }|x|\leq 2,
\\ \beta \Big(\frac{3|x|^2}{2}-4\Big) & \text{for } |x|> 2.
 \end{cases}
\end{align}
This potential has a Lipschitz continuous gradient
and is strongly convex with convexity constant $k=3\beta$ outside a Euclidean ball with radius $\tilde{R}=2$. 
We consider the splitting $-\nabla \rmV(x)=-\kappa x+g(x)$ with $\kappa=(2/3)k=2\beta$ and
\begin{align*}
g(x)=\begin{cases} -\beta (x^3-3x) & \text{for }|x|\leq 2,
\\ -\beta x & \text{for } |x|> 2.
 \end{cases}
\end{align*}
Then, the function $g$ is Lipschitz continuous with Lipschitz constant $L_g=9\beta$ and  \eqref{eq:strongconv} is satisfied for sufficiently large $R$. 
\end{example}

\subsection{Construction of the metric and the coupling}

We provide an informal construction of the coupling and the complementary metric. Given two Brownian motions $(B_t)_{t\geq 0}$, $(B_t')_{t\geq 0}$ and $(x,y),(x',y')\in\mathbb{R}^{2d}$, let $((X_t, Y_t),(X_t',Y_t'))_{t\geq 0}$ be an arbitrary coupling of two solutions to \eqref{eq:underdampLang}. 
It holds for the difference process $(Z_t, W_t)_{t\geq 0}=(X_t-X_t',Y_t-Y_t')_{t\ge 0}$,
\begin{align*}
\begin{cases}
&\rmd Z_t=W_t 
\\ & \rmd W_t=(-\gamma W_t +u b(X_t)-u b(X_t'))\rmd t +\sqrt{2\gamma u} \rmd (B_t-B_t').
\end{cases}
\end{align*}
Adapting the idea of the coupling construction from \cite{EbGuZi19}, the process $(Z_t, Q_t)_{t\ge 0}=(Z_t, Z_t+\gamma^{-1} W_t)_{t\ge 0}$ satisfies the stochastic differential equation
\begin{align} \label{eq:coupl_diffproz}
\begin{cases}
& \rmd Z_t=-\gamma Z_t \rmd t +\gamma Q_t \rmd t
\\& \rmd Q_t=\gamma^{-1} u (b(X_t)-b(X_t'))\rmd t +\sqrt{2\gamma^{-1}u }\rmd (B_t-B_t').
\end{cases}
\end{align}
As in \cite{EbGuZi19}, we apply a synchronous coupling for $Q_t=0$, since in this case the first equation of \eqref{eq:coupl_diffproz} is contractive and the absence of the noise ensures that the dynamics is not driven away from this area by random fluctuations.
Apart from $Q_t=0$, we want to apply a reflection coupling, which guarantees that the dynamics returns to the line $Q_t=0$. 
Note that this construction leads to a coupling that is sticky on the hyperplane $\{((x,y),(x',y'))\in\mathbb{R}^{4d}:x-x'+\gamma^{-1}(y-y')=0\}$. However, since it is technically hard to construct this sticky coupling, we consider 
approximations of the coupling, which are rigorously stated in \Cref{sec:coupling_confined} and which suffice for our purpose.
Similarly as in \cite{EbGuZi19}, we show for $r_s(t)=\alpha|Z_t|+|Q_t|<R_1$ with appropriately chosen constants $\alpha$, $R_1$ that there exists a concave increasing function $f$ depending on $\alpha$ and $R_1$ such that $f(r_s(t))$ is contractive on average. 
Note that the application of a concave function has the effect that a decrease in $r_s$ has a larger impact than an increase in $r_s$.

On the other hand, if the difference process $(Z_t, W_t)_{t\geq 0}$ is sufficiently far away from the origin, we obtain under \Cref{ass:b} for the force $b$ contractivity for the process $r_l(t)=(\gamma^{-2}u Z_t\cdot(KZ_t)+(1/2)|(1-2\tau)Z_t+\gamma^{-1}W_t|^2+(1/2)|\gamma^{-1}W_t|^2)^{1/2}$, where 
$\tau>0$ is a constant depending on $\kappa$, $\gamma$, $u$ and $L_g$. More precisely, we obtain local contractivity with contraction rate $\gamma\tau$ for $r_l(t)^2>\mathcal{R}$ for some $\mathcal{R}>0$ depending on $R$, $\kappa$, $\gamma$, $u$ and $L_g$.  
The process $r_l(t)$ is designed such that the local contraction rate is optimized up to some constant, see \Cref{lem:contr_outside_nonl}.

We construct a metric which is globally contractive on average using the previously established coupling. The key idea lies in combining $r_s$ and $r_l$ in such a way, that the two local contraction results imply global contractivity in the new metric. Note that for simplicity, we write $r_l$ and $r_s$ both for the norm $r_l(z,w)$ (respectively $r_s(z,w)$) of $(z,w)\in\mathbb{R}^{2d}$ and for the distance $r_l((x,y),(x',y'))$ (respectively $r_s((x,y),(x',y'))$) of $(x,y),(x',y')\in\mathbb{R}^{2d}$.


As we see in \Cref{sec:proof_conf_contr}, the lower bound $\mathcal{R}$ in the contraction result for large distances is fixed due to the dependence on the drift assumptions, whereas the upper bound $R_1$ in the result for small distances is flexible with the drawback that the contraction rate gets smaller for larger $R_1$. 
To benefit from the local contraction results, we want for all $(z,w)\in\mathbb{R}^{2d}$ that $r_s(z,w)\leq R_1$ or $r_l(z,w)^2\geq \mathcal{R}$ holds, which we achieve by choosing $R_1$ sufficiently large. 
We construct a continuous transition between $r_s$ and $r_l$ by considering $r_s\wedge(D_{\mathcal{K}}+\epsilon r_l)$, where the constant $\epsilon$ satisfies $2\epsilon r_l\leq r_s$ and the constant $D_{\mathcal{K}}$ is given such that $r_s(z,w)\wedge(D_{\mathcal{K}}+\epsilon r_l(z,w))=r_s(z,w)$ for $(z,w)$ with $r_l(z,w)^2\leq \mathcal{R}$. Then, we set $R_1$ such that $r_s(z,w)\wedge(D_{\mathcal{K}}+\epsilon r_l(z,w))=D_{\mathcal{K}}+\epsilon r_l(z,w)$ for $(z,w)$ with $r_s(z,w)\leq R_1$ is guaranteed.

In particular, in this construction the level set $r_s(z,w)-\epsilon r_l(z,w)=D_{\mathcal{K}}$ is optimally encompassed by the level set $r_s(z,w)=R_1$ and $r_l(z,w)^2=\mathcal{R}$, as illustrated in \Cref{figure1}, and $r_s(z,w)\leq R_1$ or $r_l(z,w)^2\geq \mathcal{R}$ is ensured. We define the metric $\rho((x,y),(x',y'))=f(r_s((x,y),(x',y'))\wedge \{D_{\mathcal{K}}+\epsilon r_l((x,y),(x',y'))\})$. 
As illustrated in \Cref{figure2}, we obtain $f(r_s)$ for small distances and $f(D_{\mathcal{K}}+\epsilon r_l((x,y),(x',y')))$ for large distances.
A detailed rigorous construction and a proof showing that $\rho$ defines a metric are given in \Cref{sec:couplingandmetric}.

\begin{figure}

\includegraphics[scale=0.6]{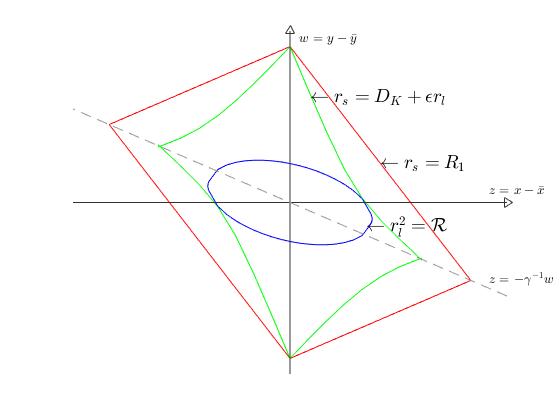}
\centering
\caption{Level sets of the metrics $r_l$ and $r_s$.}\label{figure1}

\end{figure}

\begin{figure}
\centering
\includegraphics[scale=0.5]{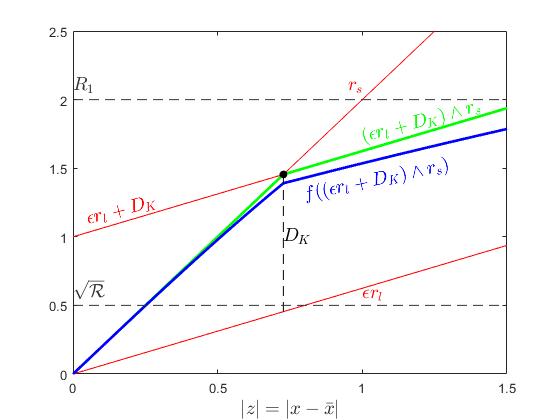}
\caption{Sketch of the metric construction 
 $f((\epsilon r_l+D_{\mathcal{K}})\wedge r_s)$. Here the metric is evaluated for $z=-\gamma^{-1}w$ (i.e., along the dashed line in \Cref{figure1}).
}\label{figure2}
\end{figure}

\subsection{A global contraction result for the classical Langevin dynamics with general external force} \label{sec:contr_underdamped}
We establish the main contraction result for the classical Langevin dynamics given by \eqref{eq:underdampLang}.

\begin{theorem} \label{thm:contraction_underdamped}
For $t\ge 0$, let $\mu_t$ and $\nu_t$ be the law at time $t$ of the processes $(X_t,Y_t)$ and $(X_t',Y_t')$, respectively, where $(X_s,Y_s)_{s\geq 0}$ and $(X_s',Y_s')_{s\geq 0}$ are solutions to \eqref{eq:underdampLang} with initial distributions $\mu_0$ and $\nu_0$ on $\mathbb{R}^{2d}$, respectively.
Suppose \Cref{ass:b} holds and 
\begin{align} \label{eq:cond_gamma}
L_g u\gamma^{-2}< \frac{\kappa}{2 L_g}.
\end{align} 
Then, 
\begin{align*}
\mathcal{W}_{1,\rho}(\mu_t, \nu_t)\leq e^{-ct} \mathcal{W}_{1,\rho}(\mu_0,\nu_0) \quad \text{and} \quad \mathcal{W}_{1}(\mu_t, \nu_t)\leq M_1 e^{-ct} \mathcal{W}_{1}(\mu_0,\nu_0),
\end{align*}
where the distance $\rho$ is defined precisely in \eqref{eq:rho} below and the contraction rate $c$ is given by 
\begin{align}
&c=\gamma \exp(-\Lambda)\min\Big(\frac{(L_K+L_g)u \gamma^{-2}}{4}\Lambda^{1/2},\frac{1}{8}\Lambda^{1/2},\frac{\tau E}{2}\Big) \quad \text{with} \label{eq:c_thm}
\\ & \Lambda=\frac{L_K+L_g}{4}R_1^2, && \text{and}, \label{eq:Lambda}
\\ & \tau:=\min(1/8,\gamma^{-2}u \kappa/2-\gamma^{-4}L_g^2u^2), && \text{and}, \label{eq:tau1}
\\ & E:=\frac{1}{6}\min\Big(1, \frac{\sqrt{\kappa}\gamma}{\sqrt{8u }(L_K+L_g)},\sqrt{\frac{\kappa u}{2}}\gamma^{-1}, 2(L_K+L_g)u\gamma^{-2} \Big). \label{eq:E}
\end{align}
The constants $R_1$ satisfies
\begin{align} \label{eq:relation_R1}
\frac{2}{3}\min(1, 2(L_K+L_g)u\gamma^{-2}) \sqrt{\frac{8 u \1_{\{R>0\}}+L_g u R^2}{\tau \gamma^2}} \le R_1 \le 4\max\Big(\frac{\sqrt{8}(L_K+L_g)u}{\gamma \sqrt{\kappa}}, 1\Big) \sqrt{\frac{8 u \1_{\{R>0\}}+L_g u R^2}{\tau \gamma^2}},
\end{align} and is explicitly stated in \eqref{eq:R_1}.
The constant $M_1$ is given by 
\begin{align} \label{eq:M1}
M_1 =\max(2(L_K+L_g)u \gamma^{-1}+\gamma,1)\frac{1}{2}\exp(\Lambda)\max\Big(3,\frac{3\gamma^2}{2(L_K+L_g)u}\Big)\max(\sqrt{2/(\kappa u)},2).
\end{align}
\end{theorem}

\begin{proof}
The proof is postponed to \Cref{sec:proof_conf_contr}.
\end{proof}

\begin{remark} 
Compared to the contraction result obtained in \cite[Theorem 2.3]{EbGuZi19}, global contractivity in Wasserstein distance is obtained with rate $c$ given in \eqref{eq:c_thm} which is independent of the dimension $d$. 
\end{remark}

\begin{remark}[Kinetic behaviour]
If $\gamma$ is chosen such that $\kappa u\gamma^{-2}$, $L_g u\gamma^{-2}$ and $L_K u\gamma^{-2}$ are fixed and further $L_KR^2$ and $L_g R^2$ are fixed,  we obtain similarly to \cite[Corollary 2.9]{EbGuZi19} that the contraction rate is of order $\Omega(R^{-1})$.
\end{remark}

\begin{remark} \label{rem:spectralgap} If $R=0$, the metric $\rho$ defined in \eqref{eq:rho} reduces to  $\rho((x,y),(\bar{x},\bar{y}))=(\gamma^{-2}(x-\bar{x})\cdot (K(x-\bar{x}))+(1/2)|(1-2\tau)(x-\bar{x})+\gamma^{-1}(y-\bar{y})|^2+(1/2)|\gamma^{-1}(y-\bar{y})|^2)^{1/2}$ 
and the coupling given in \Cref{sec:coupling_confined} becomes the synchronous coupling. This metric differs from $r$ defined in \eqref{eq:dist_r} by the constant $\tau$, since here the drift $b$ is not necessarily of gradient-type and we can not make use of the co-coercivity property as in the proof of \Cref{thm:contr_strongconv}.
 Following the proof given in \Cref{sec:proof_conf_contr}, we obtain contraction in $L^1$ Wasserstein distance, 
with contraction rate $c=\min(\gamma/16,\kappa\gamma^{-1}/4-8\gamma^{-3}L_g^2 u^2)$. 
We remark that the constant $E$ vanishes in the contraction rate, which measures the difference between the two metrics that are considered in general for $\rho$.
If $L_g\le \sqrt{2}\kappa$, the contraction rate is maximized for $\gamma=u^{1/2}(2\kappa+(4\kappa^2-8L_g^2)^{1/2})^{1/2}$ and satisfies $c=u^{1/2}(2\kappa+(4\kappa^2-8L_g^2)^{1/2})^{1/2}/16$, i.e., in this case the rate is of order $\mathcal{O}(\sqrt{\kappa})$. 
\end{remark}


\begin{example}[Double-well potential] For the model given in \Cref{example1}, we obtain contraction with respect to the designed Wasserstein distance if $\gamma>9\sqrt{\beta}$ is satisfied. 
\end{example}

\section{Contraction for nonlinear Langevin dynamics of McKean-Vlasov type} \label{sec:contr_nonlinear}


Consider the Langevin dynamics of McKean-Vlasov type given in \eqref{eq:KFP}.
We require \Cref{ass:b} for the function $b:\mathbb{R}^d\to\mathbb{R}^d$. For the function $\tilde{b}:\mathbb{R}^{2d}\to\mathbb{R}^d$ we impose:
\begin{assumption} \label{ass:tildeb}
The function $\tilde{b}:\mathbb{R}^{2d}\to\mathbb{R}^d$ is $\tilde{L}$-Lipschitz continuous.
\end{assumption}
\begin{example}[Quadratic interaction potential] \label{example2}
Consider $\tilde{b}(x ,y)= ky$ with ${k}\in\mathbb{R}$. Then $\tilde{L}= |k|$ and $\tilde{b}$ corresponds to the interaction potential $\rmW(x ,y )=-kx\cdot y$. This potential is attractive for $k>0$ and repulsive for $k<0$.  
\end{example}
\begin{example}[Mollified Coulomb, Newtonian and logarithmic potentials] \label{example3}
The gradients of the Coulomb potential and of the Newtonian potential, which describe charged and self-gravitating particles \cite{BoDo95}, are not  Lipschitz continuous. However, the gradient of a mollified version (see \cite{GuLeMo21}) given by 
\begin{align*}
\rmW(x,y)=\frac{\tilde{k}}{(|x-y|^p+q^p)^{1/p}} \qquad \text{ for } p\ge 2, q\in\mathbb{R}_+ \text{ and } \tilde{k}\in\mathbb{R}
\end{align*}
satisfies \Cref{ass:tildeb}, since  $\|\mathrm{Hess} \ W\|<\infty$, and therefore $\nabla_x \rmW$ is Lipschitz continuous.
%
In the same line, the gradient of 
the mollified version of the logarithmic potential given by 
\begin{align*}
\rmW(x,y)=-2\log((|x-y|^p+q^p)^{1/p}) \qquad \text{for } p\ge 2,q\in\mathbb{R}_+
\end{align*}
satisfies \Cref{ass:tildeb}.
\end{example}

Under the above conditions, we establish contraction in an $L^1$ Wasserstein distance.

\begin{theorem}[Contraction for nonlinear Langevin dynamics] \label{thm:contraction_nonlinear_conf}
Let $\bar{\mu}_0$ and $\bar{\nu}_0$ be two probability distributions on $\mathbb{R}^{2d}$ with finite second moment.
For $t\geq 0$, let $\bar{\mu}_t$ and $\bar{\nu}_t$ be the law at time $t$ of the processes $(\bar{X}_t,\bar{Y}_t)$ and $(\bar{X}_t',\bar{Y}_t')$, respectively, where $(\bar{X}_s,\bar{Y}_s)_{s\geq 0}$ and $(\bar{X}_s',\bar{Y}_s')_{s\geq 0}$ are solutions to \eqref{eq:KFP} with initial distribution $\bar{\mu}_0$ and $\bar{\nu}_0$, respectively. 
Suppose \Cref{ass:b}, \Cref{ass:tildeb} and \eqref{eq:cond_gamma} hold.  
Let $\tilde{L}$ satisfy
\begin{align} \label{eq:condition_tildeL}
\tilde{L}\le \exp(-\Lambda)\min\Big\{
\frac{\gamma \tau }{12}\sqrt{\frac{\kappa}{u}}\min(1,2(L_K+L_g)u\gamma^{-2}),\frac{L_K+L_g}{4}\Big\},
\end{align} 
where $\Lambda$ and $\tau$ are given in \eqref{eq:Lambda} and \eqref{eq:tau1}, respectively.
Then
\begin{align*}
\mathcal{W}_{1,\rho}(\bar{\mu}_t,\bar{\nu}_t)\leq e^{-\bar{c}t}\mathcal{W}_{1,\rho}(\bar{\mu}_0,\bar{\nu}_0)
\quad \text{and}\quad
\mathcal{W}_{1}(\bar{\mu}_t,\bar{\nu}_t)\leq M_1 e^{-\bar{c}t}\mathcal{W}_{1}(\bar{\mu}_0,\bar{\nu}_0),
\end{align*}
where the distance $\rho$ is given in \eqref{eq:rho} and $\bar{c}=c/2$ with $c$ given in \eqref{eq:c_thm}.
The constant $M_1$ is given in \eqref{eq:M1}.
%
Moreover, there exists a unique invariant probability measure $\bar{\mu}_\infty$ for \eqref{eq:KFP} and convergence in $L^1$ Wasserstein distance to $\bar{\mu}_\infty$ holds.

\end{theorem} 

\begin{proof}
The proof is based on the coupling approach and the metric construction given in \Cref{sec:metric_conf} and \Cref{sec:coupling_confined}, respectively, and is postponed to \Cref{sec:proof_conf_contr}.
\end{proof}

\begin{remark} In comparison to \cite[Theorem 3.1]{GuLeMo21}, global contractivity is established with a contraction rate and a restriction on the Lipschitz constant $\tilde{L}$ that are independent of the dimension $d$. 
\end{remark}

\begin{remark}
Compared to the contraction result in \Cref{thm:contraction_underdamped} for classical Langevin dynamics, the contraction rate deteriorates by a factor of $2$ to compensate for the nonlinear interaction terms.

If $R=0$, \eqref{eq:condition_tildeL} reduces to $\tilde{L}\le \tau\gamma\sqrt{\kappa/u}/8$ and contraction holds with rate $\bar{c}=\min(\gamma/32,\kappa u\gamma^{-1}/8-L_g^2 u^2\gamma^{-3}/2)$ by \Cref{lem:contr_outside_nonl} and \eqref{eq:proof1_3}. If $L_g\le \sqrt{2}\kappa$, the contraction rate is maximized for $\gamma=\sqrt{u}(2\kappa+(4\kappa^2-8L_g^2)^{1/2})^{1/2}$ yielding $\bar{c}=\sqrt{u}(2\kappa+(4\kappa^2-8L_g^2)^{1/2})^{1/2}/32$. If the drift is additionally of gradient-type, we can adapt the proof of \Cref{thm:contr_strongconv} and use the co-coercivity property to obtain a contraction rate of order $\mathcal{O}(\sqrt{\kappa})$ for $L_g\le 3\kappa$ and a rate of order $\mathcal{O}(\kappa/\sqrt{L_g})$ for $L_g> 3\kappa$.
\end{remark}

\begin{remark} \label{rem:unconfined}
The contraction results can be extended to unconfined Langevin dynamics. Consider $b\equiv 0$ and $\tilde{b}:\mathbb{R}^{2d}\to\mathbb{R}^d$ given by $\tilde{b}(x,y)=-\tilde{K}(x-y)+\tilde{g}(x-y)$ where $\tilde{K}\in\mathbb{R}^{d\times d}$ is a positive definite matrix  with smallest eigenvalue $\tilde{\kappa}$ and where $\tilde{g}:\mathbb{R}^d\to\mathbb{R}^d$ is an anti-symmetric, $L_{\tilde{g}}$-Lipschitz continuous function $\tilde{g}:\mathbb{R}^d\to\mathbb{R}^d$. If $L_{\tilde{g}}\le (\gamma/2) \sqrt{\tilde{\kappa}/u}\min(1/8, \tilde{\kappa}u\gamma^{-2}/2)$, contraction in an $L^1$ Wasserstein distance can be shown via a synchronous coupling approach. The underlying distance function in the Wasserstein distance is based on a similar twisted $2$-norm structure as the distance $r_l$ given in \eqref{eq:r_l}.
We note that the conditions on $L_g$ and $\tilde{L}$ are combined in the restrictive condition on $L_{\tilde{g}}$, which implies $L_{\tilde{g}}\le \tilde{\kappa}/8$ and which gives only contraction for small perturbations of linear interaction forces. A detailed analysis of the unconfined dynamics is given in \Cref{app:unconf}.
\end{remark}

\section{Metric and coupling} \label{sec:couplingandmetric}

\subsection{Metric construction} \label{sec:metric_conf}
For both the classical Langevin dynamics and the nonlinear Langevin dynamics, i.e., when \Cref{ass:b} holds, 
we consider the metrics $r_l,r_s:\mathbb{R}^{2d}\times \mathbb{R}^{2d}\to[0,\infty)$ 
given by
\begin{equation}  \label{eq:r_l}
\begin{aligned}
r_l((x,y),(\bar{x},\bar{y}))^2&:=\frac{u}{\gamma^2}(x-\bar{x})\cdot (K(x-\bar{x}))+\frac{(1-2\tau)^2}{2}|x-\bar{x}|^2+\gamma^{-1}(1-2\tau)(x-\bar{x})(y-\bar{y})+\gamma^{-2}|y-\bar{y}|^2 
\\ & =\gamma^{-2} u(x-\bar{x})\cdot (K(x-\bar{x}))+\frac{1}{2}|(1-2\tau)(x-\bar{x})+\gamma^{-1}(y-\bar{y})|^2+\frac{1}{2}\gamma^{-2}|y-\bar{y}|^2,
\end{aligned}
\end{equation}
and 
\begin{equation}\label{eq:r_s}
\begin{aligned}
r_s((x,y),(\bar{x},\bar{y})):=\alpha|x-\bar{x}|+|x-\bar{x}+\gamma^{-1}(y-\bar{y})|, 
\end{aligned}
\end{equation}
for $(x,y),(\bar{x},\bar{y})\in\mathbb{R}^{2d}$, where the constants $\tau$ and $\alpha$ are
given by \eqref{eq:tau1} and
\begin{align} 
\alpha:=2(L_K+L_g)u\gamma^{-2},\label{eq:alpha1}
\end{align} 
respectively.
Next, we state the rigorous construction of the metric $\rho:\mathbb{R}^{2d}\times\mathbb{R}^{2d}\to[0,\infty)$, that is applied in \Cref{thm:contraction_underdamped} and \Cref{thm:contraction_nonlinear_conf}, and that is glued together of $r_l$ and $r_s$ in an appropriate way. Note that $r_l$ and $r_s$ are equivalent metrics. 
More precisely, for all $(x,y),(\bar{x},\bar{y})\in\mathbb{R}^{2d}$ it holds
$2\epsilon r_l((x,y),(\bar{x},\bar{y}))\leq r_s((x,y),(\bar{x},\bar{y}))$ with
\begin{align} \label{eq:epsilon}
\epsilon=(1/2)\min(1,(2/3)\alpha/(\sqrt{L_K u}\gamma^{-1}),\alpha).
\end{align}
Indeed, for $(z,w)=(x-\bar{x},y-\bar{y})$
\begin{align*}
r_l^2((x,y),(\bar{x},\bar{y}))&
\leq L_K \gamma^{-2} u|z|^2+\frac{1}{2}|z+\gamma^{-1}w|^2+2\tau|z||z+\gamma w|+2\tau^2|z|^2+\frac{1}{2}|\gamma^{-1}w|^2 \qquad \text{and}
\\  r_s^2((x,y),(\bar{x},\bar{y}))&\geq \frac{1}{2}(\alpha|z|+|z+\gamma^{-1}w|)^2+\frac{1}{2}\min(\alpha^2,1)\gamma^{-2}|w|^2 
\\ & \geq  \frac{\alpha^2}{2}|z|^2+\alpha|z||z+\gamma^{-1}w|+\frac{1}{2}|z+\gamma^{-1}w|^2+\frac{1}{2}\min(1,\alpha^2)\gamma^{-2}|w|^2,
\end{align*}
and 
\begin{align*}
4\epsilon^2\leq \min\Big(\frac{\alpha^2}{2(L_K u \gamma^{-2}+2\tau L_K u \gamma^{-2}/2)},1,\alpha^2\Big)\leq \min\Big(\frac{\alpha^2}{2(L_K u \gamma^{-2}+2\tau^2)},1,\frac{\alpha}{2\tau},\alpha^2\Big),
\end{align*}
since $\alpha>\kappa\gamma^{-2}$ and $\tau\leq \min(1/8,L_K\gamma^{-2} u/2)$ by \eqref{eq:alpha1} and \eqref{eq:tau1}.
Further, for all $(x,y),(\bar{x},\bar{y})\in\mathbb{R}^{2d}$ it holds
$\mathcal{E} r_s((x,y),(\bar{x},\bar{y})) \le r_l((x,y),(\bar{x},\bar{y}))  $ with
\begin{align} \label{eq:mathcalE}
\mathcal{E}=\min(\sqrt{\kappa u}\gamma^{-1}/(\sqrt{8}\alpha),1/2),
\end{align}
since
\begin{align*}
\frac{r_l(t)}{r_s(t)} 
 \ge 
\Big(\frac{\kappa u\gamma^{-2}|\bar{Z}_t|^2+(1/2)|(1-2\tau)\bar{Z}_t+\gamma^{-1}\bar{W}_t|^2}{2(a+2\tau)^2|\bar{Z}_t|^2+2|(1-2\tau)\bar{Z}_t+\gamma^{-1}\bar{W}_t|^2}\Big)^{1/2}\ge \min\Big(\frac{\sqrt{\kappa u} \gamma^{-1}}{\sqrt{8}\alpha},\frac{1}{2}\Big).
\end{align*}

Define
\begin{align} \label{eq:Delta}
\Delta((x,y),(\bar{x},\bar{y})):=r_s((x,y),(\bar{x},\bar{y}))-\epsilon r_l((x,y),(\bar{x},\bar{y}))
\end{align}
for $(x,y),(\bar{x},\bar{y})\in\mathbb{R}^{2d}$ and 
\begin{align} \label{eq:D_K}
D_{\mathcal{K}}:=\sup_{((x,y),(\bar{x},\bar{y}))\in\mathbb{R}^{4d}:(x-\bar{x},y-\bar{y})\in \mathcal{K}}\Delta((x,y),(\bar{x},\bar{y})),
\end{align}
where the compact set $\mathcal{K}\subset\mathbb{R}^{2d}$ is given by 
\begin{align} \label{eq:K}
\mathcal{K}:=\{(z,w)\in\mathbb{R}^{2d}: \gamma^{-2} uz\cdot (Kz)+(1/2)|(1-2\tau)z+\gamma^{-1}w|^2+(1/2)|\gamma^{-1}w|^2 
\leq\mathcal{R}\}.
\end{align}
with
\begin{align} \label{eq:mathcalR}
\mathcal{R}=(1/\tau)(8 u \1_{\{R>0\}}+L_g uR^2)\gamma^{-2}.
\end{align}

We define the metric $\rho:\mathbb{R}^{2d}\times\mathbb{R}^{2d}\to[0,\infty)$ by 
\begin{align} \label{eq:rho}
\rho((x,y),(\bar{x},\bar{y})):= f((\Delta((x,y),(\bar{x},\bar{y}))\wedge D_{\mathcal{K}})+\epsilon r_l((x,y),(\bar{x},\bar{y})))
\end{align} 
for $(x,y),(\bar{x},\bar{y})\in\mathbb{R}^{2d}$,
where $\Delta$ and $D_{\mathcal{K}}$ are given in \eqref{eq:Delta} and \eqref{eq:D_K}.
The function $f$ is an increasing concave function defined by 
\begin{align} \label{eq:f}
f(r)&:=\int_0^r \phi(s)\psi(s) \rmd s, 
\end{align}
where
\begin{equation} \label{eq:f_definitions}
\begin{aligned}
 \phi(s)&:= \exp\Big(-\frac{\alpha\gamma^{2}}{4 u} \frac{(s\wedge R_1)^2}{2}\Big), && \qquad\Phi(s)=\int_0^s \phi(x) \rmd x,
\\ \psi(s)&:= 1-\frac{\hat{c}}{2} \gamma u^{-1} \int_0^{s\wedge R_1} \Phi(x) \phi(x)^{-1} \rmd x, && \qquad\hat{c}=\frac{1}{u^{-1} \gamma \int_0^{R_1}\Phi(s)\phi(s)^{-1}\rmd s},
\end{aligned}
\end{equation}
and where $R_1$ is given by 
\begin{align} \label{eq:R_1}
R_1:=\sup_{((x,y),(\bar{x},\bar{y})):\Delta((x,y),(\bar{x},\bar{y}))\leq D_{\mathcal{K}}} r_s(((x,y),(\bar{x},\bar{y}))).
\end{align}
The construction of the function $f$ is adapted from \cite{Eb16}. Since $\psi(s)\in[1/2,1]$, it holds for $r\geq 0$
\begin{align} \label{eq:f_property}
f'(R_1)r=(\phi(R_1)/2) r\leq \Phi(r)/2\leq f(r)\leq \Phi(r)\leq r.
\end{align}
Note that the constant $R_1$ is finite and $R_1\leq \sup_{\Delta((x,y),(\bar{x},\bar{y})) \leq D_{\mathcal{K}}} 2\Delta((x,y),(\bar{x},\bar{y}))\leq 2 D_{\mathcal{K}}$ holds, since $\Delta((x,y),(\bar{x},\bar{y}))=r_s((x,y),(\bar{x},\bar{y}))-\epsilon r_l((x,y),(\bar{x},\bar{y}))\geq (1/2)r_s((x,y),(\bar{x},\bar{y}))$ for any $(x,y),(\bar{x},\bar{y})\in\mathbb{R}^{2d}$ by \eqref{eq:epsilon}. 
Hence, $\hat{c}$ given in \eqref{eq:f_definitions} and $f$ are well-defined. Further,
\begin{align*}
R_1 \le 2 D_{\mathcal{K}} &  
 \le 2 \sup_{((x,y),(\bar{x},\bar{y}))\in\mathbb{R}^{4d}:(x-\bar{x},y-\bar{y})\in \mathcal{K}} (\mathcal{E}^{-1}-2\epsilon) r_l((x,y),(\bar{x},\bar{y}))\le 2(\mathcal{E}^{-1}-2\epsilon)\sqrt{\mathcal{R}}.
\end{align*}
The constant $R_1$ is also bounded from below by 
\begin{align*}
R_1\ge \sup_{((x,y),(\bar{x},\bar{y})):\Delta((x,y),(\bar{x},\bar{y}))\leq D_{\mathcal{K}}} 2\epsilon r_l(((x,y),(\bar{x},\bar{y}))) \ge 2\epsilon \sqrt{\mathcal{R}},
\end{align*}
since $\Delta((x,y),(\bar{x},\bar{y}))\leq D_{\mathcal{K}}$ for all $(x,y),(\bar{x},\bar{y})\in\mathbb{R}^{2d}$ such that $r_l((x,y),(\bar{x},\bar{y}))^2=\mathcal{R}$. By \eqref{eq:mathcalR}, \eqref{eq:epsilon}, \eqref{eq:mathcalE}, the two bounds on $R_1$ imply the relation \eqref{eq:relation_R1} of $R$ and $R_1$ given in \Cref{thm:contraction_underdamped}.

 By this construction for the metric $\rho$, it holds $(\Delta((x,y),(\bar{x},\bar{y}))\wedge D_{\mathcal{K}})+\epsilon r_l((x,y),(\bar{x},\bar{y}))=r_s((x,y),(\bar{x},\bar{y}))$ for $\Delta((x,y),(\bar{x},\bar{y}))\leq D_{\mathcal{K}}$, and in particular for $r_l((x,y),(\bar{x},\bar{y}))^2\leq\mathcal{R}$. Further,
$(\Delta((x,y),(\bar{x},\bar{y}))\wedge D_{\mathcal{K}})+\epsilon r_l((x,y),(\bar{x},\bar{y}))=D_{\mathcal{K}}+\epsilon r_l((x,y),(\bar{x},\bar{y}))$ for $\Delta((x,y),(\bar{x},\bar{y}))>D_{\mathcal{K}}$ and in particular for $r_s((x,y),(\bar{x},\bar{y}))>R_1$.


If $R=0$, then $\mathcal{K}=\{(0,0)\}$ and hence 
$D_{\mathcal{K}}=R_1=0$ and $f(r)=r$. In this case, we can omit the factor $\epsilon$ in \eqref{eq:rho} and \eqref{eq:rho_N} and set $\rho((x,y),(\bar{x},\bar{y}))=r_l((x,y),(\bar{x},\bar{y}))$ for simplicity. 


\begin{lemma} \label{lem:rho_metric}
The function $\rho$ given in \eqref{eq:rho} defines a metric on $\mathbb{R}^{2d}$ and is equivalent to the Euclidean distance on $\mathbb{R}^{2d}$. 
\end{lemma}
\begin{proof}
Symmetry and positive definiteness holds directly. Hence, $\rho$ is a semimetric. 
To prove the triangle inequality, we note that for $(x, y),(\bar{x},\bar{y}),(\hat{x},\hat{y})\in\mathbb{R}^{2d}$,
\begin{align*}
(\Delta & ((x, y),(\bar{x},\bar{y}))\wedge D_{\mathcal{K}})+\epsilon r_l((x, y),(\bar{x},\bar{y}))
\\ &  = r_s((x, y),(\bar{x},\bar{y}))\wedge (D_{\mathcal{K}}+\epsilon r_l((x, y),(\bar{x},\bar{y})))
\\ & \leq   (r_s((x, y),(\hat{x},\hat{y})) +  r_s((\hat{x}, \hat{y}),(\bar{x},\bar{y})) ) 
 \wedge (D_{\mathcal{K}}+ \epsilon r_l((x, y),(\hat{x},\hat{y})) + \epsilon r_l((\hat{x}, \hat{y}),(\bar{x},\bar{y}))) 
\\ & \leq   (r_s((x, y),(\hat{x},\hat{y})) +  r_s((\hat{x}, \hat{y}),(\bar{x},\bar{y})) ) 
 \wedge (D_{\mathcal{K}}+ \epsilon r_l((x, y),(\hat{x},\hat{y}))+ D_{\mathcal{K}} +  \epsilon r_l((\hat{x}, \hat{y}),(\bar{x},\bar{y}))) 
\\ & \quad \wedge (D_{\mathcal{K}}+ \epsilon r_l((x, y),(\hat{x},\hat{y})) +  (1/2) r_s((\hat{x}, \hat{y}),(\bar{x},\bar{y}))) 
 \wedge (D_{\mathcal{K}}+ (1/2) r_s((x, y),(\hat{x},\hat{y})) + \epsilon r_l((\hat{x}, \hat{y}),(\bar{x},\bar{y}))) 
\\& \leq (\Delta((x, y),(\bar{x},\bar{y}))\wedge D_{\mathcal{K}})+\epsilon r_l((x, y),(\bar{x},\bar{y})) + (\Delta((x, y),(\hat{x},\hat{y}))\wedge D_{\mathcal{K}})+\epsilon r_l((\hat{x}, \hat{y}),(\bar{x},\bar{y})),
\end{align*}
since $r_l$ and $r_s$ are metrics on $\mathbb{R}^{2d}$ and $\epsilon r_l((x, y),(\bar{x},\bar{y}))\leq (1/2) r_s((x, y),(\bar{x},\bar{y}))$. Since $f$ given in \eqref{eq:f} is a concave function, $\rho((x, y),(\bar{x},\bar{y}))\leq\rho((x, y),(\hat{x},\hat{y}))+\rho((\hat{x}, \hat{y}),(\bar{x},\bar{y}))$ for  $(x, y),(\bar{x},\bar{y}),(\hat{x},\hat{y})\in\mathbb{R}^{2d}$. 
Hence, $\rho$ defines a metric. 

Further, it holds for all $(x,y),(\bar{x},\bar{y})\in\mathbb{R}^{2d}$,
\begin{equation} \label{eq:rho_upperbound}
\begin{aligned}
\Delta((x,y),(\bar{x},\bar{y}))\wedge D_{\mathcal{K}}+\epsilon r_l((x,y),(\bar{x},\bar{y}))&\leq  r_s((x,y),(\bar{x},\bar{y})) \leq  \max(\alpha+1,\gamma^{-1})(|x-\bar{x}|+|y-\bar{y}|) 
\\ & \leq \max(\alpha+1,\gamma^{-1})\sqrt{2}|(x,y)-(\bar{x},\bar{y})|.
\end{aligned}
\end{equation}
and
\begin{equation} \label{eq:rho_lowerbound}
\begin{aligned}
\Delta((x,y),(\bar{x},\bar{y}))\wedge D_{\mathcal{K}}+\epsilon r_l((x,y),(\bar{x},\bar{y})) &\geq \epsilon r_l((x,y),(\bar{x},\bar{y}))\geq \epsilon(\kappa u \gamma^{-2}|x-\bar{x}|^2+\frac{1}{2}\gamma^{-2}|y-\bar{y}|^2)^{1/2} 
\\ & \geq \epsilon\gamma^{-1}\min(\sqrt{\kappa u },1/\sqrt{2})|(x,y)-(\bar{x},\bar{y})|
\\ & \geq \epsilon\gamma^{-1}\min(\sqrt{\kappa u /2},1/2)(|x-\bar{x}|+|y-\bar{y}|).
\end{aligned}
\end{equation}
Then, by \eqref{eq:f_property}, 
\begin{align} \label{eq:rho_equivalence}
\mathbf{C}_1 |(x,y)-(\bar{x},\bar{y})|\leq \rho((x,y),(\bar{x},\bar{y})) \leq \mathbf{C}_2 |(x,y)-(\bar{x},\bar{y})|
\end{align}
with 
$\mathbf{C}_1=f'(R_1)\epsilon\gamma^{-1}\min(\sqrt{\kappa u},1/\sqrt{2}) $ and $\mathbf{C}_2= \sqrt{2} \max(\alpha+1,\gamma^{-1})$.

\end{proof}

%

\subsection{Coupling for Langevin dynamics} \label{sec:coupling_confined}
To prove \Cref{thm:contraction_underdamped} and \Cref{thm:contraction_nonlinear_conf} we construct a coupling of two solutions to \eqref{eq:KFP}. The construction is partially adapted from the coupling approach introduced in \cite{EbGuZi19}. Recall that $\tilde{b}\equiv 0$ in \Cref{thm:contraction_underdamped}.

Let $\xi$ be a positive constant, which we take finally to the limit $\xi\to0$. Let $(B^{\rc}_t)_{t\geq 0}$ and $(B^{\sc}_t)_{t\geq 0}$ be two independent $d$-dimensional Brownian motions and let $\bar{\mu}_0,\bar{\nu}_0$ be two probability measures on $\mathbb{R}^{2d}$. The coupling $((\bar{X}_t,\bar{Y}_t),(\bar{X}_t',\bar{Y}_t'))_{t\geq 0}$ of two copies of solutions to  \eqref{eq:KFP} is a solution to the SDE on $\mathbb{R}^{2d}\times\mathbb{R}^{2d}$ given by
\begin{equation} \label{eq:KFP_coupling_nonl}
\begin{aligned}
&\begin{cases} 
\rmd \bar{X}_t&=\bar{Y}_t \rmd t
\\  \rmd \bar{Y}_t&=(-\gamma\bar{Y}_t+ ub(\bar{X}_t)+u \int_{\mathbb{R}^d}\tilde{b}(\bar{X}_t,z)\bar{\mu}_t^x (\rmd z))\rmd t+\sqrt{2\gamma u}\sc(Z_t,W_t)\rmd B_t^{\sc} +\sqrt{2\gamma u}\rc(Z_t,W_t)\rmd B_t^{\rc}
\end{cases}
\\
&\begin{cases}
\rmd \bar{X}_t'&=\bar{Y}_t' \rmd t
\\  \rmd \bar{Y}_t'&=(-\gamma \bar{Y}_t' +u b(\bar{X}_t') +u \int_{\mathbb{R}^d}\tilde{b}(\bar{X}_t',z)\bar{\nu}_t^x (\rmd z))\rmd t+\sqrt{2\gamma u}\sc(Z_t,W_t)\rmd B_t^{\sc}
\\ & +\sqrt{2\gamma u}\rc(Z_t,W_t)(\Id-2e_t e_t^T)\rmd B_t^{\rc}, 
\end{cases}
\\ & (\bar{X}_0,\bar{Y}_0)\sim \bar{\mu}_0, \quad (\bar{X}_0',\bar{Y}_0')\sim \bar{\nu}_0, 
\end{aligned}
\end{equation}
where $\bar{\mu}_t^x=\Law(\bar{X}_t)$ and $\bar{\nu}_t^x=\Law(\bar{X}_t')$.
Further, $Z_t=\bar{X}_t-\bar{X}_t'$, $W_t=\bar{Y}_t-\bar{Y}_t'$, $Q_t=Z_t+\gamma^{-1}W_t$ and $e_t=Q_t/|Q_t|$ if $Q_t\neq 0$ and $e_t=0$ otherwise.
The functions $\rc, \sc:\mathbb{R}^{2d}\to[0,1)$ are Lipschitz continuous and satisfy $\rc^2+\sc^2\equiv 1$ and 
\begin{equation}\label{eq:rc_sc}
\begin{aligned}
&\mathrm{rc}(z,w)=0 &&\qquad \text{ if } |z+\gamma^{-1}w|=0 
 \text{ or } (r_s(z,w))-\epsilon(r_l(z,w))\geq D_{\mathcal{K}}+\xi\cdot \1_{\{D_{\mathcal{K}}>0\}},
\\ &\mathrm{rc}(z,w)=1 && \qquad \text{ if } |z+\gamma^{-1}w|\geq \xi 
\text{ and } (r_s(z,w))-\epsilon(r_l(z,w))\leq D_{\mathcal{K}} \text{ and } D_{\mathcal{K}}>0
\end{aligned} 
\end{equation}
for $(z,w)\in\mathbb{R}^{2d}$,
where $\epsilon$ is given in \eqref{eq:epsilon}. Analogously to \eqref{eq:r_l} and \eqref{eq:r_s},
$r_l(z,w)^2=\gamma^{-2} uz\cdot(Kz)+(1/2)|(1-2\tau)z+\gamma^{-1}w|^2+(1/2)\gamma^{-2}|w|^2 $ and $r_s(z,w)=\alpha|z|+|z+\gamma^{-1} w|$.

We note that by Levy's characterization, for any solution to \eqref{eq:KFP_coupling} the processes 
\begin{align*}
B_t&:=\int_0^t \mathrm{sc}(Z_s,W_s) \rmd B_s^{\mathrm{sc}} +\int_0^t\mathrm{rc}(Z_s,W_s) \rmd B_s^{\mathrm{rc}} \qquad \text{and}
\\ \tilde{B}_t&:=\int_0^t \mathrm{sc}(Z_s,W_s) \rmd B_s^{\mathrm{sc}} +\int_0^t\mathrm{rc}(Z_s,W_s)(\Id-e_s{e_s}^T) \rmd B_s^{\mathrm{rc}}
\end{align*}
are $d$-dimensional Brownian motions. Therefore, \eqref{eq:KFP_coupling} defines a coupling between two solutions to \eqref{eq:KFP}. The constructed coupling denotes a \textit{reflection coupling} for $\rc\equiv 1$ and $\sc \equiv 0$ and a \textit{synchronous coupling} for $\sc\equiv 1$ and $\rc \equiv 0$.
Note that we obtain a synchronous coupling if $D_{\mathcal{K}}=0$. 

The processes $(Z_t)_{t\geq 0}$, $(W_t)_{t\geq 0}$ and $(Q_t)_{t\geq 0}$ satisfy the following SDEs:
\begin{equation} \label{eq:differenceproc_nonl}
\begin{aligned}
\rmd Z_t&=W_t \rmd t=(Q_t-\gamma Z_t)\rmd t,
\\ \rmd W_t&=-\gamma W_t \rmd t+ u\Big( b(\bar{X}_t)- b(\bar{X}_t')+ \int_{\mathbb{R}^d}\tilde{b}(\bar{X}_t,z)\bar{\mu}_t^x (\rmd z)-\int_{\mathbb{R}^d}\tilde{b}(\bar{X}_t',\tilde{z})\bar{\nu}_t^x(\rmd \tilde{z})\Big)\rmd t
\\ &  +\sqrt{8\gamma u}\rc(Z_t,W_t)e_t{e_t}^T\rmd B_t^{\rc},
\\  \rmd Q_t&= \gamma^{-1} u\Big(b(\bar{X}_t)-b(\bar{X}_t')+\int_{\mathbb{R}^d}\tilde{b}(\bar{X}_t,z)\bar{\mu}_t^x (\rmd z)-\int_{\mathbb{R}^d}\tilde{b}(\bar{X}_t',\tilde{z})\bar{\nu}_t^x (\rmd \tilde{z})\Big)\rmd t
 +\sqrt{8\gamma^{-1} u}\rc(Z_t,W_t)e_t{e_t}^T\rmd B_t^{\rc}.
\end{aligned}
\end{equation} 
If $Q_t=0$, we note that $Z_t$ is contractive, which we exploit in the proof of \Cref{lem:contr_inside_nonl}.

\section{Uniform in time propagation of chaos} \label{sec:unif_prop}
We provide uniform in time propagation of chaos bounds for the mean-field particle system corresponding to the nonlinear  Langevin dynamics of McKean-Vlasov type.



Fix $N\in\mathbb{N}$.
We consider the metric $\rho_N:\mathbb{R}^{2Nd}\times\mathbb{R}^{2Nd}\to[0,\infty)$ given by 
\begin{align} \label{eq:rho_N}
\rho_N((x,y),(\bar{x},\bar{y})):=N^{-1}\sum_{i=1}^N\rho((x^i,y^i),(\bar{x}^i,\bar{y}^i)) \qquad \text{for }((x,y),(\bar{x},\bar{y}))\in\mathbb{R}^{2Nd}\times\mathbb{R}^{2Nd},
\end{align}  
where $\rho$ is given in \eqref{eq:rho}.
Since $\rho$ is a metric on $\mathbb{R}^{2d}\times\mathbb{R}^{2d}$ by \Cref{lem:rho_metric}, $\rho_N$ defines a metric on $\mathbb{R}^{2Nd}\times\mathbb{R}^{2Nd}$. By \eqref{eq:rho_upperbound} and \eqref{eq:rho_lowerbound}, $\rho_N$ is equivalent to $l_N^1$ given in \eqref{eq:l_N^1}, i.e., 
\begin{align}\label{eq:rhoN_equivalence}
\mathbf{C}_1/\sqrt{2} \ell_N^1((x,y),(\bar{x},\bar{y}))\leq \rho_N((x,y),(\bar{x},\bar{y})) \leq \mathbf{C}_2/\sqrt{2} \ell_N^1((x,y),(\bar{x},\bar{y}))
\end{align}
with 
$\mathbf{C}_1=\exp(-\Lambda)\min(1,2(L_K+L_g)u\gamma^{-2})/3\gamma^{-1}\min(\sqrt{\kappa u },1/\sqrt{2}) $ and $\mathbf{C}_2= \sqrt{2} \max(2(L_K+L_g)u\gamma^{-2}+1,\gamma^{-1})$.

 For $t\geq 0$, we denote by $\bar{\mu}_t$ the law of the process $(\bar{X}_t,\bar{Y}_t)$, where $(\bar{X}_s,\bar{Y}_s)_{s\geq 0}$ is a solution to \eqref{eq:KFP} with initial distribution $\bar{\mu}_0$. We denote by $\mu_t^N$ the law of $\{X_t^{i,N},Y_t^{i,N}\}_{i=1}^N$, where $(\{X_s^{i,N},Y_s^{i,N}\}_{i=1}^N)_{s\geq 0}$ is a solution to \eqref{eq:KFP_meanfield_nongradient} with initial distribution $\mu^N_0=\mu_0^{\otimes N}$.

\begin{theorem}[Propagation of chaos for Langevin dynamics]\label{thm:propofchaos}
Suppose \Cref{ass:b} and \Cref{ass:tildeb} hold. Let $\bar{\mu}_0$ and $\mu_0$ be two probability distributions on $\mathbb{R}^{2d}$ with finite second moment. Suppose that \eqref{eq:cond_gamma} holds. 
If $\tilde{L}$ satisfies \eqref{eq:condition_tildeL}, 
then 
\begin{align*}
&\mathcal{W}_{1,\rho_N}(\bar{\mu}_t^{\otimes N},\mu_t^N)\leq e^{-\tilde{c}t}\mathcal{W}_{1,\rho_N}(\bar{\mu}_0^{\otimes N},\mu_0^N)+\mathcal{C}_1\tilde{c}^{-1}N^{-1/2} && \text{and}
\\ & \mathcal{W}_{1,\ell_N^1}(\bar{\mu}_t^{\otimes N},\mu_t^N)\leq M_1 e^{-\tilde{c}t}\mathcal{W}_{1,\ell_N^1}(\bar{\mu}_0^{\otimes N},\mu_0^N)+M_2\mathcal{C}_1\tilde{c}^{-1}N^{-1/2},
\end{align*}
where the distance $\rho_N$ is defined in \eqref{eq:rho_N} and $\tilde{c}=c/2$ with $c$ given in \eqref{eq:c_thm}. The constant $\mathcal{C}_1$ depends on $\gamma$, 
$d$, $u$, $R$, $\kappa$, $L_g$, $\tilde{L}$ and on the second moment of $\bar{\mu}_0$. The constants $M_1$ and is given in \eqref{eq:M1} and \eqref{eq:M2} and $M_2$ is given by
\begin{align} \label{eq:M2}
M_2=3\exp(\Lambda)\max\Big(1,\frac{\gamma^2}{2(L_K+L_g)u} \Big)\gamma\max(\sqrt{2/(\kappa u)},2).
\end{align}
\end{theorem}

\begin{proof}
The proof is postponed to \Cref{sec:proof_confined}.
\end{proof}

\begin{remark}
For $t\ge 0$, let $\mu_t^N$ and $\nu_t^N$ be the law of $\{X_t^{i,N},Y_t^{i,N}\}_{i=1}^N$ and $\{{X_t'}^{i,N},{Y_t'}^{i,N}\}_{i=1}^N$ where the processes $(\{X_s^{i,N},Y_s^{i,N}\}_{i=1}^N)_{s\ge 0}$ and  $(\{{X_s'}^{i,N},{Y_s'}^{i,N}\}_{i=1}^N)_{s\ge 0}$ are solutions to \eqref{eq:KFP_meanfield_nongradient} with initial distributions $\mu_0^N$ and $\nu_0^N$, respectively. An easy adaptation of the proof of \Cref{thm:propofchaos} shows that if \Cref{ass:b}, \Cref{ass:tildeb}, \eqref{eq:cond_gamma} and \eqref{eq:condition_tildeL} hold, then 
\begin{align*}\mathcal{W}_{1,\rho_N}(\mu_t^N,\nu_t^N)\leq e^{-\tilde{c}t}\mathcal{W}_{1,\rho_N}(\mu_0^N, \nu_0^N)\qquad \text{and} \qquad
\mathcal{W}_{1,\ell_N^1}(\mu_t^N,\nu_t^N)\leq M_1 e^{-\tilde{c}t}\mathcal{W}_{1,\ell_N^1}(\mu_0^N,\nu_0^N),
\end{align*}
where $\rho_N$ and  $M_1$ are given in \eqref{eq:rho_N}, and \eqref{eq:M1}, respectively, and $\tilde{c}=c/2$ with $c$ given in \eqref{eq:c_thm}.
To adapt the proof, a coupling between two copies of $N$ particle systems is applied which is constructed in the same line as \eqref{eq:KFP_coupling}.
\end{remark}

\section{Proofs} \label{sec:proof}

\subsection{Proof of Section~\ref{sec:contr_strongconv} }\label{sec:proof_strongconv}

\begin{proof}[Proof of \Cref{thm:contr_strongconv}]
Given a $d$-dimensional standard Brownian motion on $(B_t)_{t\ge 0}$ and $(x,y),(x',y')\in\mathbb{R}^{2d}$, we consider the synchronous coupling $((X_t,Y_t),(X_t',Y_t'))_{t\ge 0}$ of two copies of solutions to \eqref{eq:underdampLang_pot} on $\mathbb{R}^{2d}\times\mathbb{R}^{2d}$ given by 
\begin{equation}
\begin{aligned}
&\begin{cases}
\rmd X_t&=Y_t \rmd t
\\ \rmd Y_t&=(-\gamma Y_t-u \nabla\rmV(X_t))\rmd t+\sqrt{2\gamma u }\rmd B_t, \qquad (X_0,Y_0)=(x,y)
\end{cases}
\\ &\begin{cases}
\rmd X_t'&=Y_t' \rmd t
\\ \rmd Y_t'&=(-\gamma Y_t'-u \nabla \rmV(X_t'))\rmd t+\sqrt{2\gamma u }\rmd B_t, \qquad (X_0',Y_0')=(x',y').
\end{cases}
\end{aligned}
\end{equation}
Then, the difference process $(Z_t,W_t)_{t\ge 0}=(X_t-X_t',Y_t-Y_t')_{t\ge 0}$ satisfies
\begin{align*}
\begin{cases}
\rmd Z_t=W_t \rmd t
\\ \rmd W_t= (-\gamma W_t- u K Z_t-u (\nabla G(X_t)-\nabla G(X_t')))\rmd t. 
\end{cases}
\end{align*}
We note that since by \Cref{ass:conf_pot}, $G$ is continuously differentiable, convex and has $L_G$-Lipschitz continuous gradients, $G$ is co-coercive (see e.g. \cite[Theorem 2.1.5]{Ne18}), i.e., it holds
\begin{align} \label{eq:cocoercivity}
|\nabla G(x)-\nabla G(x')|^2 \le L_G(\nabla G(x)-\nabla G(x'))\cdot (x-x') \qquad \text{ for all } x,x'\in\mathbb{R}^d.
\end{align}
Let $A,B,C\in\mathbb{R}^{d\times d}$ be positive definite matrices given by
\begin{align*}
A=\gamma^{-2} u K+(1/2)(1-2\lambda)^2 \mathrm{Id}, \qquad B=(1-2\lambda)\gamma^{-1} \mathrm{Id}, \qquad C=\gamma^{-2}\mathrm{Id},
\end{align*}
where $\lambda$ is given in \eqref{eq:lambda} and $\mathrm{Id}$ is the $d\times d$ identity matrix. Then by Ito's formula and Young's inequality, we obtain
\begin{equation} \label{eq:calc_stronconv}
\begin{aligned}
\frac{\rmd}{\rmd t} &(Z_t\cdot(A Z_t)+Z_t\cdot(B W_t)+W_t\cdot (C W_t))
\\ &  = 2W_t\cdot(A Z_t) + W_t \cdot (B W_t)
+ Z_t \cdot (B (-\gamma W_t-uK Z_t-u(\nabla G(X_t)-\nabla G(X_t'))))
\\ & + 2 W_t\cdot (C (-\gamma W_t-uK Z_t-u(\nabla G(X_t)-\nabla G(X_t'))))
\\ & \le -u \gamma^{-1}(1-2\lambda)Z_t\cdot (K Z_t)- (1-2\lambda)\gamma^{-1}u Z_t(\nabla G(X_t)-\nabla G(X_t'))+\gamma^{-3}u^2|\nabla G(X_t)-\nabla G(X_t')|^2
\\ & + Z_t\cdot((2A-\gamma B -2uKC) W_t)+((1-2\lambda)\gamma^{-1}-\gamma^{-1}) |W_t|^2.
\end{aligned}
\end{equation}
By \eqref{eq:cocoercivity}, \eqref{eq:lambda} and \eqref{eq:cond_LG}, it holds
\begin{equation}\label{eq:conseq_cocoerc}
\begin{aligned}
- (1-2\lambda)\gamma^{-1}u Z_t&\cdot(\nabla G(X_t)-\nabla G(X_t'))+\gamma^{-3}u^2|\nabla G(X_t)-\nabla G(X_t')|^2
\\ & \le - ((1-2\lambda)\gamma^{-1}u -\gamma^{-3} L_G u^2) Z_t(\nabla G(X_t)-\nabla G(X_t'))\le 0.
\end{aligned}
\end{equation}
Further by \eqref{eq:lambda}, it holds
\begin{align*}
-u \gamma^{-1} (1-4 \lambda) Z_t\cdot(K Z_t) \le -u (\gamma^{-1}/2)  Z_t\cdot(K Z_t) \le -u (\gamma^{-1}/2)  \kappa |Z_t|^2 \le - \lambda \gamma |Z_t|^2 \le - \lambda \gamma (1-2\lambda)^2|Z_t|^2
\end{align*}
and hence, $-u \gamma^{-1} (1-2 \lambda) Z_t\cdot(K Z_t) \le -2\gamma \lambda Z_t\cdot ( A Z_t) $.
Set $r(t)=r((X_t,Y_t),(X_t',Y_t'))$ with $r$ defined in \eqref{eq:dist_r}.
Then by \eqref{eq:calc_stronconv} and \eqref{eq:conseq_cocoerc}, we obtain
\begin{align*}
\frac{\rmd}{\rmd t} r(t)^2 &= \frac{\rmd}{\rmd t} (Z_t\cdot(A Z_t)+Z_t\cdot(B W_t)+W_t\cdot (C W_t)) 
\\ & \le - 2\lambda \gamma (Z_t \cdot(A Z_t)+ Z_t \cdot (B W_t) + W_t \cdot (C W_t)) = - 2\lambda \gamma r(t)^2.
\end{align*}
Taking the square root and applying Gr\"onwall's inequality yields
\begin{align*}
r(t)\le e^{-ct} r(0)
\end{align*}
with $c$ given in \eqref{eq:rate_strongconv}. Then for all $p\ge 1$ it holds
\begin{align*}
\mathcal{W}_{p,r}(\mu_t,\nu_t)\le \mathbb{E}[r(t)^p]^{1/p}\le e^{-ct}\mathbb{E}[r(0)^p]^{1/p}.
\end{align*}
We take the infimum over all couplings $\gamma\in \Pi(\mu_0,\nu_0)$ and obtain the first bound. For the second bound we note that for any $(x, y),(x',y')\in\mathbb{R}^{2d}$ 
\begin{align*}
\sqrt{\min(u\gamma^{-2}\kappa,\gamma^{-2}/2)}(|x-x'|^2+|y-y'|^2)^{1/2}&\le r((x,y),(x',y'))
\\ & \le \sqrt{\max(u\gamma^{-2}L_K+1, 3/2\gamma^{-2})}(|x-x'|^2+|y-y'|^2)^{1/2}.
\end{align*}
Hence, the second bound in \Cref{thm:contr_strongconv} holds with $M$ given in \eqref{eq:M_strongconv}.
\end{proof}

\subsection{Proofs of Section~\ref{sec:contr_underdamped} and Section~\ref{sec:contr_nonlinear}}\label{sec:proof_conf_contr}

To show \Cref{thm:contraction_nonlinear_conf}, we prove two local contraction results using the coupling defined in \eqref{eq:KFP_coupling_nonl}.
We write $r_l(t)=r_l((\bar{X}_t,\bar{Y}_t),(\bar{X}_t',\bar{Y}_t'))$, $r_s(t)=r_s((\bar{X}_t,\bar{Y}_t),(\bar{X}_t',\bar{Y}_t'))$ and $\Delta(t)=\Delta((\bar{X}_t,\bar{Y}_t),(\bar{X}_t',\bar{Y}_t'))$. 
\begin{lemma} \label{lem:contr_outside_nonl}
Suppose \Cref{ass:b}, \Cref{ass:tildeb} and  \eqref{eq:cond_gamma} hold. 
Let $((\bar{X}_s,\bar{Y}_s),(\bar{X}_s',\bar{Y}_s'))_{s\geq 0}$ be a solution to \eqref{eq:KFP_coupling_nonl}. 
Then for $t\geq 0$ with $\Delta(t)\geq D_{\mathcal{K}} $, it holds 
\begin{equation}\label{eq:contr_outside_nonl}
\begin{aligned} 
\rmd r_l(t)&\leq -c_1 r_l(t) \rmd t  
  + \frac{|(1-2\tau)Z_t+2\gamma^{-1}W_t|}{2\gamma r_l(t)}\tilde{L} u(\mathbb{E}[|Z_t|]+|Z_t|)\rmd t 
\\ & + \sqrt{8\gamma^{-1} u}\rc(Z_t,W_t)\frac{(1-2\tau)Z_t+2\gamma^{-1}W_t}{2r_l(t)}\cdot e_t{e_t}^T\rmd B_t,
\end{aligned} 
\end{equation}
where $c_1=\tau\gamma/2$ with $\tau$ given in \eqref{eq:tau1}.
\end{lemma}

\begin{proof} 
Let $A,B,C\in\mathbb{R}^{d\times d}$ be positive definite matrices given by
\begin{align} \label{eq:matrix_dist}
A=\gamma^{-2} u K +(1/2)(1-2\tau)^2\mathrm{Id}, \qquad B=(1-2\tau)\gamma^{-1}\mathrm{Id}, \qquad \text{and } \qquad C=\gamma^{-2}\mathrm{Id},
\end{align} 
where $\tau$ is given by \eqref{eq:tau1} and $\mathrm{Id}$ is the $d\times d$ identity matrix. 
By \eqref{eq:differenceproc_nonl} and  Ito's formula, it holds 
%
\begin{align*}
\rmd & (Z_t\cdot (A Z_t)+ Z_t\cdot (B W_t)+W_t\cdot (C W_t)) 
\\&\le 2(AZ_t)\cdot W_t \rmd t + \Big(W_t\cdot (B W_t)-\gamma (BZ_t)\cdot W_t- u (BZ_t)\cdot(KZ_t)+L_{g}u (1-2\tau)\gamma^{-1}|Z_t|^2\cdot \1_{\{|Z_t|<R\}}\Big)\rmd t
 \\ & +\Big(-2\gamma W_t\cdot (CW_t)-2u(C W_t)\cdot (K Z_t)+2\gamma^{-2} L_{g}u|W_t||Z_t|\Big)\rmd t
 +|B Z_t+2C W_t|\tilde{L}u(\mathbb{E}[|Z_t|]+|Z_t|)\rmd t
 \\ & +\gamma^{-2} 8\gamma u \rc(Z_t,W_t)^2\rmd t
 +\sqrt{8\gamma u} \rc(Z_t,W_t)(BZ_t +2C{W_t})\cdot e_t{e_t}^T\rmd B_t
\\ &\leq Z_t\cdot((-u BK+\gamma^{-1}u L_g^2 C)Z_t)\rmd t +Z_t\cdot(2A-\gamma B-2 u K C)W_t \rmd t
 +((1-2\tau)\gamma^{-1}-\gamma^{-1}) |W_t|^2\rmd t
 \\ &  +(1-2\tau)\gamma^{-1} u L_{g}|Z_t|^2 \1_{\{|Z_t|< R\}}\rmd t 
 + |(1-2\tau)\gamma^{-1} Z_t+2\gamma^{-2} W_t|\tilde{L} u(\mathbb{E}[|Z_t|]+|Z_t|)\rmd t 
\\ &  + 8\gamma^{-1} u (\rc(Z_t,W_t))^2\rmd t+ \sqrt{8\gamma u}\rc(Z_t,W_t)((1-2\tau)\gamma^{-1} Z_t+2\gamma^{-1} W_t)\cdot e_t{e_t}^T\rmd B_t
\\ & \le -2\tau \gamma (Z_t\cdot (AZ_t) +Z_t\cdot(B W_t)+W_t\cdot(CW_t))\rmd t
\\ &  +(1-2\tau)\gamma^{-1} u L_{g}|Z_t|^2 \1_{\{|Z_t|< R\}}\rmd t 
 + |(1-2\tau)\gamma^{-1} Z_t+2\gamma^{-2} W_t|\tilde{L} u (\mathbb{E}[|Z_t|]+|Z_t|)\rmd t 
\\ &  + 8\gamma^{-1}u (\rc(Z_t,W_t))^2\rmd t+ \sqrt{8\gamma u}\rc(Z_t,W_t)((1-2\tau)\gamma^{-1} Z_t+2\gamma^{-1} W_t)\cdot e_t{e_t}^T\rmd B_t,
\end{align*}
where we used \eqref{eq:tau1} in the last step. More precisely, the definition of $\tau$ implies for all $z\in\mathbb{R}^d$,
\begin{equation} \label{eq:estimatemu}
\begin{aligned}
 z\cdot((-(1-4\tau)\gamma^{-1} u K+\gamma^{-3}L_g^2 u^2 \mathrm{Id})z) & \leq  (-(1/2)\kappa u \gamma^{-1}+\gamma^{-3}L_g^2 u^2)|z|^2
\\ & \leq (-\tau\gamma)|z|^2\leq (-\tau\gamma(1-2\tau)^2)|z|^2.
\end{aligned}
\end{equation} 
Note that $r_l(t)^2=Z_t\cdot (A Z_t)+ Z_t\cdot (B W_t)+W_t\cdot (C W_t)$. 
Then,
\begin{align*}
\rmd r_l(t)^2&\leq -2\tau\gamma r_l(t)^2\rmd t  +\gamma^{-1}(1-2\tau) L_{g}u |Z_t|^2 \1_{\{|Z_t|< R\}}\rmd t
+ \gamma^{-1}|(1-2\tau)Z_t+2\gamma^{-1}W_t|\tilde{L} u(\mathbb{E}[|Z_t|]+|Z_t|)\rmd t
\\ & +8\gamma^{-1}u \rc(Z_t,W_t)^2 \rmd t
 + \sqrt{8\gamma^{-1}u}\rc(Z_t,W_t)((1-2\tau)Z_t+2\gamma^{-1}W_t)\cdot e_t{e_t}^T\rmd B_t.
\end{align*}
Since $\Delta(t)\geq D_{\mathcal{K}} $, it holds 
$r_l(t)^2\geq \mathcal{R}$ by \eqref{eq:D_K} and \eqref{eq:K}. By \eqref{eq:rc_sc}, $\rc(Z_t,W_t)^2\le \1_{\{R>0\}}$, and hence, by \eqref{eq:mathcalR}
\begin{align*}
-\tau\gamma r_l(t)^2  +\gamma^{-1}(1-2\tau) L_{g} u|Z_t|^2 \1_{\{|Z_t|< R\}}&+8\gamma^{-1}u\rc(Z_t,W_t)^2 
\leq -\tau\gamma \mathcal{R}+L_g u R^2\gamma^{-1}+8\gamma^{-1}u \1_{\{R>0\}}\leq 0.
\end{align*}
We obtain by Ito's formula and since the second derivative of the square root is negative,
\begin{align*}
\rmd r_l(t)\leq (2r_l(t))^{-1}\rmd r_l(t)^2
&\leq -c_1 r_l(t) \rmd t  
  + \gamma^{-1}|(1-2\tau)Z_t+2\gamma^{-1}W_t|(2r_l(t))^{-1}\tilde{L} u(\mathbb{E}[|Z_t|]+|Z_t|)\rmd t 
\\ & + \sqrt{8\gamma^{-1}}\rc(Z_t,W_t)(2r_l(t))^{-1}((1-2\tau)Z_t+2\gamma^{-1}W_t)\cdot e_t{e_t}^T\rmd B_t,
\end{align*}
which concludes the proof.
\end{proof}

\begin{lemma} \label{lem:contr_inside_nonl}
Suppose \Cref{ass:b} and \Cref{ass:tildeb} hold. Fix $\xi>0$.
Let $((\bar{X}_s,\bar{Y}_s),(\bar{X}_s',\bar{Y}_s'))_{s\geq 0}$ be a solution to \eqref{eq:KFP_coupling_nonl}. 
Let $r_s$ be given by \eqref{eq:r_s} with $\alpha$ given in \eqref{eq:alpha1}. Then for $t\geq 0$ with $\Delta(t)< D_{\mathcal{K}}$, it holds
\begin{align*}
\rmd f(r_s(t))&\leq -c_2 f(r_s(t))\rmd t + \gamma^{-1}\tilde{L}u (\mathbb{E}[|Z_t|]+|Z_t|) \rmd t -\frac{\gamma\alpha}{4}f'(R_1)|Z_t|\rmd t
 +(1+\alpha)\xi\gamma\rmd t +\rmd M_t,
\end{align*}
where $f$ is given in \eqref{eq:f}, $(M_t)_{t\geq 0}$ is a martingale and $c_2$ is given by 
\begin{equation} \label{eq:c_2} 
c_2:= \min\Big(\frac{2}{\gamma \int_0^{R_1}\Phi(s)\phi(s)^{-1}ds}, \frac{\gamma}{8}\frac{R_1\phi(R_1)}{\Phi(R_1)}\Big). 
\end{equation}
\end{lemma} 

\begin{proof}
The proof is an adaptation of the proof of \cite[Lemma 3.1]{EbGuZi19}. First, we note that, $(Z_t)_{t\geq 0}$ given in \eqref{eq:differenceproc_nonl} is  almost surely continuously differentiable with derivative $\rmd Z_t/\rmd t=-\gamma Z_t+\gamma Q_t$ and hence $t\to|Z_t|$ is almost surely absolutely continuous with 
\begin{align*}
&\frac{\rmd}{\rmd t}|Z_t|=\frac{Z_t}{|Z_t|}\cdot(-\gamma Z_t +\gamma Q_t) &&\text{ for a.e. }t \text{ such that } Z_t\neq 0 \text{ and }
\\ & \frac{\rmd }{\rmd t} |Z_t| \leq \gamma |Q_t| &&\text{ for a.e. }t\text{ such that } Z_t = 0.
\end{align*}
and therefore
\begin{align} \label{eq:Z_t}
\frac{\rmd}{\rmd t}|Z_t|\leq -\gamma |Z_t|+\gamma |Q_t| \text{ for a.e. } t\geq 0.
\end{align}
By Ito's formula and by \Cref{ass:b} and \Cref{ass:tildeb}, we obtain for $|Q_t|$,
\begin{align*}
\rmd & |Q_t|
\\ &= \gamma^{-1} u{e_t}\cdot\Big(b(\bar{X}_t)-b(\bar{X}_t')+\int_{\mathbb{R}^d}\tilde{b}(\bar{X}_t,z)\bar{\mu}_t^x (\rmd z)-\int_{\mathbb{R}^d}\tilde{b}(\bar{X}_t',\tilde{z})\bar{\nu}_t^x (\rmd \tilde{z})\Big)\rmd t
 +\sqrt{8\gamma^{-1} u}\rc(Z_t,W_t){e_t}^T\rmd B_t
\\ & \leq \gamma^{-1} u(L_K+L_g+\tilde{L})|Z_t|\rmd t+\gamma^{-1}\tilde{L}u \mathbb{E}[|Z_t|]\rmd t +\sqrt{8\gamma^{-1}u}\rc(Z_t,W_t){e_t}^T\rmd B_t^{\rc}.
\end{align*} 
Note that there is no Ito correction term, since $\partial^2_{q/|q|}|q|=0$ for $q\neq 0$ and $\rc=0$ for $Q_t=0$. Combining this bound with \eqref{eq:Z_t} yields for $r_s(t)$,
\begin{align*}
\rmd r_s(t)&\leq \Big(((L_K+L_g)u\gamma^{-2}-\alpha)\gamma|Z_t|+\alpha\gamma|Q_t|+\gamma^{-1}\tilde{L}u(\mathbb{E}[|Z_t|]+|Z_t|)\Big) \rmd t
 +\sqrt{8\gamma^{-1}u}\mathrm{rc}(Z_t,W_t){e_t}^T\rmd B_t^{\rc}.
\end{align*}
By Ito's formula,
\begin{align*}
\rmd f(r_s(t))&\le f'(r_s(t))\Big(((L_K+L_g)u\gamma^{-2}-\alpha)\gamma|Z_t|+\alpha\gamma|Q_t|+\gamma^{-1}\tilde{L}u(\mathbb{E}[|Z_t|]+|Z_t|)\Big) \rmd t
\\ & +f'(r_s(t))\sqrt{8\gamma^{-1}u}\mathrm{rc}(Z_t,W_t){e_t}^T \rmd B_t^{\mathrm{rc}}+f''(r_s(t))4\gamma^{-1}u\mathrm{rc}(Z_t,W_t)^2\rmd t.
\end{align*}
\textit{Case 1:} Consider $\Delta(t)<D_{\mathcal{K}}$ and $|Q_t|>\xi$, then $\rc(Z_t,W_t)=1$ and $r_s(t)< R_1$. Hence, we obtain
\begin{align*}
 \rmd f(r_s(t))&\leq f'(r_s(t))\alpha\gamma r_s(t) \rmd t+f''(r_s(t))4\gamma^{-1}u \rmd t +\gamma^{-1}\tilde{L} u(\mathbb{E}[|Z_t|]+|Z_t|)\rmd t-\frac{\alpha\gamma}{2}|Z_t|f'(r_s(t))\rmd t+\rmd M_t
 \\ &\leq -2\hat{c}f(r_s(t))\rmd t+\gamma^{-1}\tilde{L} u(\mathbb{E}[|Z_t|]+|Z_t|)\rmd t-\frac{\alpha\gamma}{2}|Z_t|f'(R_1)\rmd t+\rmd M_t
 \\ &\leq -c_2f(r_s(t))\rmd t+\gamma^{-1}\tilde{L} u(\mathbb{E}[|Z_t|]+|Z_t|)\rmd t-\frac{\alpha\gamma}{2}|Z_t|f'(R_1)\rmd t+\rmd M_t,
\end{align*} 
where $(M_t)_{t\ge 0}$ is a martingale and $\hat{c}$ is given in \eqref{eq:f_definitions}. Note that the second step holds since by \eqref{eq:f} and \eqref{eq:f_property},
\begin{align} \label{eq:f_property2}
f'(r)\alpha\gamma r + f''(r)4 \gamma^{-1} u\leq -2\hat{c} f(r) \qquad \text{for all } r\in [0,R_1).
\end{align} 
\textit{Case 2:} Consider $\Delta(t)<D_{\mathcal{K}}$ and $|Q_t|\leq \xi$, then $\alpha|Z_t|= r_s(t)-|Q_t|\geq r_s(t)-\xi$.
We note that
\begin{align*}
((L_K+L_g) u\gamma^{-2}-\alpha)|Z_t|+\alpha|Q_t|\leq -\frac{1}{2}r_s(t)+(1+\alpha)\xi.
\end{align*}
Since the second derivative of $f$ is negative and $\psi(s)\in[1/2,1]$, it holds
\begin{equation} \label{eq:nonl_inside_2case}
\begin{aligned}
 \rmd f(r_s(t))&\leq -\frac{ \gamma}{2}r_s(t)f'(r_s(t))\rmd t+ (1+\alpha)\gamma\xi \rmd t+\gamma^{-1} u\tilde{L}(\mathbb{E}[|Z_t|]+|Z_t|)\rmd t+\rmd M_t
 \\ & \leq  -\frac{\gamma}{8}\inf_{r\leq R_1}\frac{r\phi(r)}{\Phi(r)}f(r_s(t))\rmd t-\frac{ \gamma}{4}f'(R_1)\alpha|Z_t|\rmd t+ (1+\alpha)\gamma\xi \rmd t
 +\gamma^{-1}\tilde{L} u(\mathbb{E}[|Z_t|]+|Z_t|)\rmd t+\rmd M_t
  \\ & \leq  -\frac{\gamma}{8}\frac{R_1\phi(R_1)}{\Phi(R_1)}f(r_s(t))\rmd t -\frac{\gamma\alpha}{4}f'(R_1)|Z_t|\rmd t+ (1+\alpha)\gamma\xi \rmd t
+\gamma^{-1}\tilde{L} u(\mathbb{E}[|Z_t|]+|Z_t|)\rmd t+\rmd M_t. 
\end{aligned}
\end{equation}
Combining the two cases, we obtain the result with $c_2$ given in \eqref{eq:c_2}.
\end{proof}

\begin{proof}[Proof of \Cref{thm:propofchaos}]
To prove contraction, we consider the coupling $((\bar{X}_t,\bar{Y}_t),(\bar{X}_t',\bar{Y}_t'))_{t\geq 0}$ given in \eqref{eq:KFP_coupling_nonl} and combine the results of \Cref{lem:contr_outside_nonl} and \Cref{lem:contr_inside_nonl}. We abbreviate $\rho(t)=f((\Delta(t)\wedge D_{\mathcal{K}})+\epsilon r_l(t))$.
We distinguish two cases: \\
\textit{Case 1:} Consider $\Delta(t)< D_{\mathcal{K}}$. Then
$r_s(t)\leq R_1$ and $\rho(t)=f(r_s(t))$.
By \Cref{lem:contr_inside_nonl}, it holds for $\xi>0$
\begin{align}\label{eq:proof1_7}
\rmd\rho(t)=\rmd f(r_s(t))&\leq -c_2 f(r_s(t))\rmd t + \gamma^{-1}\tilde{L} u(\mathbb{E}[|Z_t|]+|Z_t|) \rmd t- \frac{\alpha\gamma}{4}f'(R_1)|Z_t|\rmd t+(1+\alpha)\gamma\xi \rmd t + \rmd M_t \nonumber
\\ & \leq -c_2 f(r_s(t))\rmd t + \gamma^{-1}\tilde{L} u \mathbb{E}[|Z_t|] \rmd t- \frac{\alpha\gamma}{8}f'(R_1)|Z_t|\rmd t+(1+\alpha)\gamma\xi \rmd t + \rmd M_t,
\end{align}
where $c_2$ is given by \eqref{eq:c_2} and $(M_t)_{t\geq 0}$ is a martingale. The second step holds by \eqref{eq:condition_tildeL}. \\ 
\textit{Case 2:} Consider $\Delta(t)\geq D_{\mathcal{K}}$.
We obtain by \Cref{lem:contr_outside_nonl},
\begin{align*}
\rmd r_l(t)&\leq -c_1 r_l(t) \rmd t
 +  \frac{|(1-2\tau)Z_t+2\gamma^{-1} W_t|}{2\gamma r_l(t)}\tilde{L} u(\mathbb{E}[|Z_t|]+|Z_t|) \rmd t
\\ &    + \sqrt{8\gamma^{-1} u}\rc(Z_t,W_t)\frac{(1-2\tau)Z_t+2\gamma^{-1}W_t}{2r_l(t)}\cdot e_t{e_t}^T\rmd B_t,
\end{align*}
where $c_1$ is given in \Cref{lem:contr_outside_nonl}. 
Note that $\frac{\rmd}{\rmd x} f(D_{\mathcal{K}}+\epsilon x)=\epsilon f'(D_{\mathcal{K}}+\epsilon x)$. Further, since $f(D_{\mathcal{K}}+\epsilon x)$ is a concave function, $\frac{\rmd^2}{\rmd x^2} f(D_{\mathcal{K}}+\epsilon x)$ is negative. 
By Ito's formula, we obtain
\begin{equation}\label{eq:proof_nonl_conf_2nd}
\begin{aligned} 
\rmd \rho(t)&= \rmd f(D_{\mathcal{K}}+\epsilon r_l(t))
\\ & \leq \epsilon f'(D_{\mathcal{K}}+\epsilon r_l(t))\Big(-c_1 r_l(t)
 + \frac{|(1-2\tau)Z_t+2\gamma^{-1}W_t|}{2\gamma r_l(t)}\tilde{L}u (\mathbb{E}[|Z_t|]+|Z_t|) \Big)\rmd t+\rmd \tilde{M}_t,
\end{aligned} 
\end{equation}
where $\tilde{M}_t$ is a martingale given by
\begin{align} \label{eq:proof_nonl_conf_mart}
\tilde{M}_t=\int_0^t \frac{\epsilon f'(D_{\mathcal{K}}+\epsilon r_l(s))}{2r_l(s)}\sqrt{8\gamma^{-1}u}\rc(Z_s,W_s) ((1-2\tau)Z_s+2\gamma^{-1}W_s)\cdot e_s{e_s}^T\rmd B_s.
\end{align}
We split the first term of \eqref{eq:proof_nonl_conf_2nd} and bound each part applying \eqref{eq:f_property},
\begin{align}\label{eq:proof1_1a}
-\frac{\epsilon f'(D_{\mathcal{K}}+\epsilon r_l(t))}{2}c_1 r_l(t)  & \leq -\Big\{\inf_{q\geq 0}\frac{f'(q)q}{f(q)}\Big\}\frac{\epsilon c_1 r_l(t)}{2(D_{\mathcal{K}}+\epsilon r_l(t))}\rho(t)
 \leq -f'(R_1)\frac{\epsilon c_1 r_l(t)}{2(D_{\mathcal{K}}+\epsilon r_l(t))}\rho(t) 
\end{align}
and
\begin{align}
-\frac{\epsilon f'(D_{\mathcal{K}}+\epsilon r_l(t))}{2}c_1 r_l(t)  & \leq -f'(R_1)\frac{\epsilon c_1}{2}r_l(t). \label{eq:proof1_1b}
\end{align}
We note that since $\Delta(t)>D_{\mathcal{K}}$ it holds,
\begin{align}
\frac{ r_l(t)}{D_{\mathcal{K}}+\epsilon r_l(t)}&\geq \frac{r_l(t)}{r_s(t)}\geq  \mathcal{E}, \label{eq:proof1_2}
\end{align} 
where $\mathcal{E}$ is given in \eqref{eq:mathcalE}.
Hence, we obtain for the first term of \eqref{eq:proof_nonl_conf_2nd}, by \eqref{eq:proof1_1a}, \eqref{eq:proof1_1b} and \eqref{eq:proof1_2}
\begin{align}\label{eq:proof1_1}
-\epsilon f'(D_{\mathcal{K}}+\epsilon r_l(t))c_1 r_l(t)^2 \leq -f'(R_1)\frac{c_1\epsilon\mathcal{E}}{2}\rho(t)-f'(R_1)\frac{c_1\epsilon}{2}r_l(t).
\end{align}
For the second term of \eqref{eq:proof_nonl_conf_2nd}, we note
\begin{align}
\epsilon f'(D_{\mathcal{K}}+\epsilon r_l(t))   \frac{|(1-2\tau)Z_t+2\gamma^{-1} W_t|}{2\gamma r_l(t)}
&\le \frac{\epsilon}{2\gamma} \sqrt{\frac{(1-2\tau)^2|Z_t|^2+4(1-2\tau)\gamma^{-1}Z_t\cdot W_t+4\gamma^{-2}|W_t|^2}{(1/2)(1-2\tau)^2|Z_t|^2+(1-2\tau)\gamma^{-1}Z_t\cdot W_t+\gamma^{-2}|W_t|^2}}
\le \frac{\epsilon}{\gamma}. \label{eq:proof1_3}
\end{align}
Combining \eqref{eq:proof1_1} and \eqref{eq:proof1_3} yields,
\begin{align}
\rmd  \rho(t)&\leq -f'(R_1)\frac{c_1\epsilon\mathcal{E}}{2} \rho(t)\rmd t-f'(R_1)\frac{c_1\epsilon}{2}r_l(t)\rmd t  +\epsilon\gamma^{-1}\tilde{L}u(\mathbb{E}[|Z_t|]+|Z_t|)\rmd t+\rmd \tilde{M}_t. \nonumber
\\ & \leq -f'(R_1)\frac{c_1\epsilon\mathcal{E}}{2} \rho(t)\rmd t-f'(R_1)\frac{c_1\epsilon}{2}\sqrt{\kappa u\gamma^{-2}}|Z_t|\rmd t  +\frac{1}{2}\gamma^{-1}\tilde{L}u(\mathbb{E}[|Z_t|]+|Z_t|)\rmd t  +\rmd \tilde{M}_t, \label{eq:proof1_4}
\end{align}
where $r_l(t)\geq \sqrt{\kappa u\gamma^{-2}}|Z_t|$ and $2\epsilon\leq 1$ are applied and where $(\tilde{M}_t)_{t\geq 0}$ is given in \eqref{eq:proof_nonl_conf_mart}.

Combining \eqref{eq:proof1_7} and \eqref{eq:proof1_4}, taking expectation and $\xi\to0$, yields
\begin{align*}
\frac{\rmd }{\rmd t} \mathbb{E}[ \rho(t)]&\leq -\min\Big(c_2,
f'(R_1)\frac{c_1\epsilon \mathcal{E}}{2}\Big)\mathbb{E}[\rho(t)]
 -\min\Big(f'(R_1)\frac{\alpha\gamma}{8},f'(R_1)\frac{c_1\epsilon}{2}\sqrt{\kappa u\gamma^{-2}}\Big)\mathbb{E}[|Z_t|]+\gamma^{-1}\tilde{L} u\mathbb{E}[|Z_t|]
\\ & \leq  -\min\Big(c_2,
f'(R_1)\frac{c_1\epsilon \mathcal{E}}{2}\Big)\mathbb{E}[\rho(t)],
\end{align*}
where we used \eqref{eq:condition_tildeL} and \eqref{eq:alpha1} in the second step. 
By applying Gr\"{o}nwall's inequality, we obtain
\begin{align*}
\mathcal{W}_{1,\rho}(\bar{\mu}_t, \bar{\nu}_t)\leq \mathbb{E}[\rho(t)]&\leq e^{-c_3t}\mathbb{E}[\rho(0)]
\end{align*}
with
\begin{align} \label{eq:c}
c_3=\min\Big(\frac{2}{\gamma \int_0^{R_1}\Phi(s)\phi(s)^{-1}ds}, \frac{\gamma}{8}\frac{R_1\phi(R_1)}{\Phi(R_1)} ,
f'(R_1)\gamma\tau\frac{\epsilon \mathcal{E}}{4}\Big).
\end{align}
The term $\epsilon\mathcal{E}$ is bounded from below by $E$ given in \eqref{eq:E}.
For the first two arguments in the minimum we note that
\begin{align}
\int_0^{R_1}\int_0^s & \exp\Big(-\frac{\alpha\gamma^2}{4 u}\frac{r^2}{2}\Big)\rmd r \exp\Big(\frac{\alpha\gamma^2}{4 u}\frac{s^2}{2}\Big) \rmd s  \le \sqrt{\frac{\pi}{2}}\Big(\frac{\alpha \gamma^2}{4 u}\Big)^{-1/2} \int_0^{R_1}  \exp\Big(\frac{\alpha\gamma^2}{4 u}\frac{s^2}{2}\Big) \rmd s  \nonumber
\\ & \le \sqrt{\frac{\pi}{2}}\Big(\frac{\alpha \gamma^2}{4 u}\Big)^{-1/2} 2 \Big(\frac{\alpha \gamma^2}{4 u}R_1\Big)^{-1} \exp\Big(\frac{\alpha\gamma^2}{4 u}\frac{R_1^2}{2}\Big)\le 4\Big(\frac{\alpha \gamma^2}{4 u}\Big)^{-1} \Big(\frac{\alpha \gamma^2}{4 u}\frac{R_1^2}{2}\Big)^{-1/2} \exp\Big(\frac{\alpha\gamma^2}{4 u}\frac{R_1^2}{2}\Big) \label{eq:c_simplify1}
\end{align}
since $\int_0^x \exp(r^2/2)\rmd r\leq 2x^{-1}\exp(x^2/2)$, and 
\begin{align}
\frac{R_1 \phi(R_1)}{\Phi(R_1)}\ge \frac{R_1\exp(-\frac{\alpha\gamma^2}{4 u}\frac{R_1^2}{2})}{\sqrt{\frac{\pi}{2}}(\frac{\alpha\gamma^2}{4 u})^{-1/2}}=\frac{2}{\sqrt{\pi}}\Big(\frac{\alpha \gamma^2}{4 u}\frac{R_1^2}{2}\Big)^{1/2}\exp\Big(-\frac{\alpha\gamma^2}{4 u}\frac{R_1^2}{2}\Big) \ge \Big(\frac{\alpha \gamma^2}{4 u}\frac{R_1^2}{2}\Big)^{1/2}\exp\Big(-\frac{\alpha\gamma^2}{4 u}\frac{R_1^2}{2}\Big). \label{eq:c_simplify2}
\end{align}
Hence, $\mathcal{W}_{1,\rho}(\bar{\mu}_t, \bar{\nu}_t)\leq \mathbb{E}[\rho(t)]\leq e^{-\bar{c}t}\mathbb{E}[\rho(0)]$
with $c$ given by
\begin{align}
&\bar{c}=\gamma \exp(-\Lambda)\min\Big(\frac{(L_K+L_g) u\gamma^{-2}}{4}\Lambda^{1/2},\frac{1}{8}\Lambda^{1/2},\frac{\tau E}{4}\Big) \label{eq:c_thm_nonl}
\end{align}
with $\Lambda$, $\tau$ and $E$ given in \eqref{eq:Lambda}, \eqref{eq:tau1} and \eqref{eq:E}.
Taking the infimum over all couplings $\omega\in\Pi(\bar{\mu}_0,\bar{\nu}_0)$ concludes the proof of the first result.

By \eqref{eq:rho_equivalence}, the second result holds with $M_1=\mathbf{C}_2/\mathbf{C}_1$ given by \eqref{eq:M1}. 
\end{proof}

\begin{proof}[Proof of \Cref{thm:contraction_underdamped}]
\Cref{thm:contraction_underdamped} forms a special case of \Cref{thm:contraction_nonlinear_conf}. We obtain analogously to \Cref{lem:contr_outside_nonl} for $\Delta(t)\geq D_{\mathcal{K}}$,
\begin{align*}
\rmd r_l(t)\leq -c_1 r_l(t)\rmd t +\sqrt{8\gamma^{-1} u}\mathrm{rc}(Z_t,W_t)(r_l(t)^{-1}/2)((1-2\tau)Z_t+2\gamma^{-1}W_t)\cdot e_t{e_t}^T \rmd B_t,
\end{align*}
where $c_1=\tau\gamma/2$ with $\tau$ given in \eqref{eq:tau1}. 
Similarly as in \Cref{lem:contr_inside_nonl}, we get for $\Delta(t)<D_{\mathcal{K}}$ using $\tilde{L}=0$
\begin{align*}
\rmd f(r_s(t))\leq - c_2 f(r_s(t))\rmd t +(1+\alpha)\xi\gamma\rmd t +\rmd M_t ,
\end{align*}
where $M_t$ is a martingale, $\alpha$ is defined in \eqref{eq:alpha1}, $f$ is defined in \eqref{eq:f} and $c_2$ is given in \eqref{eq:c_2}. 
Combining the two local contraction results as in the proof of \Cref{thm:contraction_nonlinear_conf} gives the desired result with contraction rate 
\begin{align} \label{eq:c_underdamped}
c=\min\Big(\frac{2}{u^{-1}\gamma \int_0^{R_1}\Phi(s)\phi(s)^{-1}ds}, \frac{\gamma}{8}\frac{R_1\phi(R_1)}{\Phi(R_1)} ,
f'(R_1)\gamma\tau\frac{\epsilon\mathcal{E}}{2}\Big).
\end{align}
Note that the last two terms in the minimum differ by a factor of $2$ from the last two terms in \eqref{eq:c}, as the first terms in \eqref{eq:proof_nonl_conf_2nd} and \eqref{eq:nonl_inside_2case} are not split up to compensate for the interaction term as in the nonlinear term.
\end{proof}

\subsection{Proof of Section~\ref{sec:unif_prop}} \label{sec:proof_confined}

Fix $N\in\mathbb{N}$.
To show propagation in chaos in \Cref{thm:propofchaos} we construct in the same line as in \Cref{sec:coupling_confined} a coupling between a solution to \eqref{eq:KFP_meanfield_nongradient} and $N$ copies of solutions to \eqref{eq:KFP}.
We fix a positive constant $\xi$, which we take in the end to the limit $\xi\to0$. Let $\{(B^{i,\rc})_{t\geq 0}:i=1,\ldots,N\}$ and $\{(B^{i,\sc})_{t\geq 0}:i=1,\ldots,N\}$ be $2N$ independent $d$-dimensional Brownian motions and let $\mu_0$ and $\bar{\mu}_0$ be two probability measures on $\mathbb{R}^{2d}$. The coupling $(\{(\bar{X}_t^i,\bar{Y}_t^i),(X_t^{i},Y_t^{i})\}_{i=1}^N)_{t\ge 0}$ 
is a solution to the SDE on $\mathbb{R}^{2Nd}\times\mathbb{R}^{2Nd}$ given by
\begin{equation} \label{eq:KFP_coupling}
\begin{aligned}
&\begin{cases} 
\rmd \bar{X}_t^i&=\bar{Y}_t^i \rmd t
\\  \rmd \bar{Y}_t^i&=(-\gamma\bar{Y}_t^i+ u b(\bar{X}_t^i)+ u \int_{\mathbb{R}^d}\tilde{b}(\bar{X}_t^i,z)\bar{\mu}_t^x (\rmd z))\rmd t+\sqrt{2\gamma u}\sc(Z_t^i,W_t^i)\rmd B_t^{i,\sc} +\sqrt{2\gamma u}\rc(Z_t^i,W_t^i)\rmd B_t^{i,\rc}
\end{cases}
\\
&\begin{cases}
\rmd X_t^{i}&=Y_t^{i} \rmd t
\\  \rmd Y_t^{i}&=(-\gamma Y_t^{i}+u b(X_t^{i})+u N^{-1}\sum_{j=1}^N\tilde{b}(X_t^{i},X_t^{j}))\rmd t+\sqrt{2\gamma u}\sc(Z_t^{i},W_t^i)\rmd B_t^{i,\sc}
\\ & +\sqrt{2\gamma u}\rc(Z_t^i,W_t^i)(\Id-2e_t^i{e_t^i}^T)\rmd B_t^{i,\rc} 
\end{cases}
\\ &(\bar{X}_0^i,\bar{Y}_0^i)\sim \bar{\mu}_0, \quad ({X}_0^i,{Y}_0^i)\sim {\mu}_0
\end{aligned}
\end{equation}
for $i=1,...,N$, where $\bar{\mu}_t^x=\Law(\bar{X}_t^i)$ for all $i$. Further, $Z_t^i=\bar{X}_t^i-X_t^{i}$, $W_t^i=\bar{Y}_t^i-Y_t^{i}$, $Q_t^i=Z_t^i+\gamma^{-1}W_t^i$, and
$e_t^i=Q_t^i/|Q_t^i|$ if $Q_t^i\neq 0$ and $e_t^i=0$ if $Q_t^i=0$.
As in \Cref{sec:coupling_confined}, the functions $\rc, \sc:\mathbb{R}^{2d}\to[0,1)$ are Lipschitz continuous and satisfy $\rc^2+\sc^2\equiv 1$ and \eqref{eq:rc_sc}.
We note that by Levy's characterization, for any solution of \eqref{eq:KFP_coupling} the processes 
\begin{align*}
B_t^i&:=\int_0^t \mathrm{sc}(Z_s^i,W_s^i) \rmd B_s^{i,\mathrm{sc}} +\int_0^t\mathrm{rc}(Z_s^i,W_s^i) \rmd B_s^{i,\mathrm{rc}}
\\ \tilde{B}_t^i&:=\int_0^t \mathrm{sc}(Z_s^i,W_s^i) \rmd B_s^{i,\mathrm{sc}} +\int_0^t\mathrm{rc}(Z_s^i,W_s^i)(\Id-e_s^i{e_s^i}^T) \rmd B_s^{i,\mathrm{rc}}
\end{align*}
are $d$-dimensional Brownian motions. Therefore, \eqref{eq:KFP_coupling} defines a coupling between $N$ copies of solutions to \eqref{eq:KFP} and a solution to \eqref{eq:KFP_meanfield_nongradient}. 
The processes $(\{Z_t^i\}_{i=1}^N)_{t\geq 0}$, $(\{W_t^i\}_{i=1}^N)_{t\geq 0}$ and $(\{Q_t^i\}_{i=1}^N)_{t\geq 0}$ satisfy the stochastic differential equations given by
\begin{equation} \label{eq:differenceproc}
\begin{aligned}
\rmd Z_t^i&=W_t^i \rmd t=(Q_t^i-\gamma Z_t^i)\rmd t
\\ \rmd W_t^i&=\Big(-\gamma W_t^i+u\Big( b(\bar{X}_t^i)-b(X_t^{i})+\int_{\mathbb{R}^d}\tilde{b}(\bar{X}_t^i,z)\bar{\mu}_t^x (\rmd z)-N^{-1}\sum_{j=1}^N\tilde{b}(X_t^{i},X_t^{j})\Big)\Big)\rmd t
\\ & +\sqrt{8\gamma u}\rc(Z_t^i,W_t^i)e_t^i{e_t^i}^T\rmd B_t^{i,\rc}
\\  \rmd Q_t^i&= \gamma^{-1}u\Big(b(\bar{X}_t^i)-b(X_t^{i})+\int_{\mathbb{R}^d}\tilde{b}(\bar{X}_t^i,z)\bar{\mu}_t^x (\rmd z)-N^{-1}\sum_{j=1}^N\tilde{b}(X_t^{i},X_t^{j})\Big)\rmd t
\\ & +\sqrt{8\gamma^{-1}u}\rc(Z_t^i,W_t^i)e_t^i{e_t^i}^T\rmd B_t^{i,\rc},
\end{aligned}
\end{equation}
for all $i=1,...,N$. 

The proof of \Cref{thm:propofchaos} relies on three auxiliary lemmata. 
We abbreviate $r_l^i(t)=r_l((\bar{X}_t^i,\bar{Y}_t^i),(X_t^i,Y_t^i))$, $r_s^i(t)=r_s((\bar{X}_t^i,\bar{Y}_t^i),(X_t^i,Y_t^i))$ and $\Delta^i(t)=\Delta((\bar{X}_t^i,\bar{Y}_t^i),(X_t^{i},Y_t^{i}))$. 

\begin{lemma} \label{lem:contr_outside}
Suppose \Cref{ass:b} and \Cref{ass:tildeb} hold. Suppose that \eqref{eq:cond_gamma} holds. 
Let $\tau>0$ be given by \eqref{eq:tau1}. Let $(\{(\bar{X}_t^i,\bar{Y}_t^i),(X_t^{i},Y_t^{i})\}_{i=1}^N)_{t\geq 0}$ be a solution to \eqref{eq:KFP_coupling}. 
Then for  $i\in\{1,\ldots,N\}$ with $\Delta^i(t)\geq D_{\mathcal{K}} $, it holds 
\begin{equation}\label{eq:contr_outside}
\begin{aligned} 
\rmd r_l^i(t)&\leq -c_1 r_l^i(t) \rmd t  
  + \frac{|(1-2\tau)Z_t^i+2\gamma^{-1}W_t^i|}{2 \gamma r_l^i(t)}u\Big(\tilde{L}N^{-1}\sum_{j=1}^N(|Z_t^j|+|Z_t^i|)+A_t^i\Big)\rmd t 
\\ & + \sqrt{2\gamma^{-1}}\rc(Z_t^i,W_t^i)\frac{(1-2\tau)Z_t^i+2\gamma^{-1}W_t^i}{r_l^i(t)}\cdot e_t^i{e_t^i}^T\rmd B_t^i,
\end{aligned} 
\end{equation}
where $c_1=\tau\gamma/2$ and $\{A_t^i\}_{i=1}^N$ is given by
\begin{align} \label{eq:bound_A_t^i}
A_t^{i}:=\Big|\int_{\mathbb{R}^d}\tilde{b}(\bar{X}_t^i,z)\bar{\mu}_t^x (\rmd z)-N^{-1}\sum_{j=1}^N\tilde{b}(\bar{X}_t^i,\bar{X}_t^j)\Big| \qquad \text{ with } \bar{\mu}_t^x=\mathrm{Law}(\bar{X}_t^i).
\end{align} 
\end{lemma}

\begin{proof} 
By Ito's formula, it holds for $(\{Z_t^i,W_t^i\}_{i=1}^N)_{t\geq 0}=(\{\bar{X}_t^i-X_t^{i},\bar{Y}_t^{i}-Y_t^{i}\}_{i=1}^N)_{t\geq 0}$,
\begin{align*}
\begin{cases}
\rmd Z_t^i&=W_t^i \rmd t  \\
\rmd  W_t^i& = (-\gamma W_t^i+u(b(\bar{X}_t^i)-b(X_t^i)+N^{-1}\sum_{j=1}^N(\tilde{b}(\bar{X}_t^i,\bar{X}_t^j)-\tilde{b}(X_t^i,X_t^j))  +\tilde{A}_t^{i}))\rmd t
\\ & + \sqrt{8\gamma u} \rc(Z_t^i, W_t^i)e_t^i{e_t^i}^T \rmd B_t^i ,
\end{cases}
\end{align*}
where 
\begin{align*} 
\tilde{A}_t^{i}:=\Big(\int_{\mathbb{R}^d}\tilde{b}(\bar{X}_t^i,z)\bar{\mu}_t^x (\rmd z)-N^{-1}\sum_{j=1}^N\tilde{b} (\bar{X}_t^i,\bar{X}_t^j)\Big) \qquad \text{ with } \bar{\mu}_t^x=\mathrm{Law}(\bar{X}_t^i)
\end{align*} 
for all $i=1,...,N$.
Hence, by Ito's formula it holds for the positive matrices $A,B,C$ given in \eqref{eq:matrix_dist}, 
%
\begin{align*}
\rmd  (Z_t^i&\cdot (A Z_t^i)+ Z_t^i\cdot (BW_t^i)+W_t^i\cdot (C W_t^i)) 
\\ & \le  2(AZ_t^i)\cdot W_t^i \rmd t+ \Big(W_t^i\cdot(BW^i_t)-(BZ^i_t)\cdot (\gamma W^i_t+u K Z^i_t)+L_{g}u(1-2\tau)\gamma^{-1}|Z^i_t|^2\cdot \1_{\{|Z_t^i| < R\}}\Big)\rmd t
 \\ & + \Big(-2\gamma W^i_t\cdot (CW_t^i)-2 u(CW_t^i)\cdot ( KZ_t^i)+2 L_{g}u|C^{1/2}W_t^i||C^{1/2}Z_t^i|\Big)\rmd t
\\ &  +|B Z^i_t+2C W_t^i|u\Big( \tilde{L}N^{-1}\sum_{j=1}^N(|Z^j_t|+|Z_t^i|)+A_t^i\Big)
 \\ & +\gamma^{-2} 8\gamma u \rc(Z_t^i,W_t^i)^2\rmd t +  \sqrt{8\gamma u} \rc(Z_t^i,W_t^i)(BZ_t^i+CW_t^i)\cdot e_t^i{e_t^i}^T \rmd B_t^i 
\\ &\leq Z_t^i\cdot((-u KB+\gamma^{-1}L_g^2 u^2 C)Z_t^i)\rmd t +Z_t^i\cdot((2A-\gamma B-2u K C) W_t^i) \rmd t
 + W_t^i\cdot((B-\gamma C) W_t^i)\rmd t 
 \\ & + |B Z_t^i+2C W_t^i|u \Big(\tilde{L}N^{-1}\sum_{j=1}^N (|Z_t^j|+|Z_t^i|)+A_t^i\Big)\rmd t +(1-2\tau)\gamma^{-1} L_{g}|Z^i_t|^2\cdot \1_{\{|Z_t^i| < R\}}\rmd t
 \\ &  + 8\gamma^{-1}u (\rc(Z_t^i,W_t^i))^2\rmd t + \sqrt{8\gamma u}\rc(Z_t^i,W_t^i)(B Z_t^i+2C W_t^i)\cdot e_t^i{e_t^i}^T\rmd B_t^i
\end{align*}
with $\{A_t^i\}_{i=1}^N$ given by \eqref{eq:bound_A_t^i}.
By \eqref{eq:tau1} and \eqref{eq:estimatemu},
\begin{align*}
\rmd r_l^i(t)^2&= \rmd (Z_t^i\cdot(AZ_t^i)+ Z_t^i\cdot (BW_t^i)+ W_t^i\cdot(C W_t^i)) 
\\ &\leq -2\tau\gamma r_l^i(t)^2\rmd t  
+ |(1-2\tau)Z_t^i+2\gamma^{-1}W_t^i|\frac{u}{\gamma}\Big(\tilde{L}N^{-1}\sum_{j=1}^N(|Z_t^j|+|Z_t^i|)+A_t^i\Big)\rmd t+8\gamma^{-1} u \rc(Z_t^i,W_t^i)^2 \rmd t
\\ & +\gamma^{-1}(1-2\tau) L_{g} u|Z^i_t|^2\cdot \1_{\{|Z_t^i| < R\}}\rmd t 
 + \sqrt{8\gamma^{-1}u}\rc(Z_t^i,W_t^i)((1-2\tau)Z_t^i+2\gamma^{-1}W_t^i)\cdot e_t^i{e_t^i}^T\rmd B_t^i.
\end{align*}
Since $\Delta^i(t)\geq D_{\mathcal{K}} $, it holds 
$r_l^i(t)^2>\mathcal{R}$ by \eqref{eq:D_K} and \eqref{eq:K}. By \eqref{eq:mathcalR} and \eqref{eq:rc_sc}, 
\begin{align*}
-\tau\gamma r_l^i(t)^2  +\gamma^{-1}(1-2\tau) L_{g}u|Z^i_t|^2\1_{\{|Z_t^i| < R\} }&+8\gamma^{-1}u\rc(Z_t^i,W_t^i)^2 
\leq -\tau\gamma \mathcal{R}+L_g u R^2\gamma^{-1}+8\gamma^{-1}u \1_{\{R>0\}}\leq 0.
\end{align*}
By Ito's formula and since the second derivative of the square root is negative,
\begin{align*}
\rmd r_l^i(t)\leq (2r_l^i(t))^{-1}\rmd r_l^i(t)^2
&\leq -c_1 r_l^i(t) \rmd t  
  + \frac{|(1-2\tau)Z_t^i+2\gamma^{-1}W_t^i|}{2\gamma r_l^i(t)}u\Big(\tilde{L}N^{-1} \sum_{j=1}^N (|Z_t^j|+|Z_t^i|)+A_t^i\Big)\rmd t 
\\ & + \sqrt{2\gamma^{-1}}\rc(Z_t^i,W_t^i)r_l^i(t)^{-1}((1-2\tau)Z_t^i+2\gamma^{-1}W_t^i)\cdot e_t^i{e_t^i}^T\rmd B_t^i,
\end{align*}
which concludes the proof.
\end{proof}

\begin{lemma} \label{lem:contr_inside}
Suppose \Cref{ass:b} and \Cref{ass:tildeb} hold. 
Let $(\{\bar{X}_t^i,\bar{Y}_t^i,X_t^i,Y_t^i\}_{i=1}^N)_{t\ge 0}$ be a solution to \eqref{eq:KFP_coupling}. 
Let $r_s$ be given in \eqref{eq:r_s} with $\alpha$ defined in \eqref{eq:alpha1}. 
If $\Delta^i(t)< D_{\mathcal{K}}$ with $D_{\mathcal{K}}$ given in \eqref{eq:D_K}, it holds
\begin{align*}
\rmd f(r_s^i(t))&\leq -c_2 f(r_s^i(t))\rmd t + \gamma^{-1}\tilde{L}uN^{-1}\sum_{j=1}^N (|Z_t^j|+|Z_t^i|) \rmd t -\frac{\alpha\gamma}{4}f'(R_1)|Z_t^i|\rmd t
\\ & +\gamma^{-1}u\Big|\int_{\mathbb{R}^d} \tilde{b}(\bar{X}_t^i,z)\bar{\mu}_t(\rmd z)-N^{-1}\sum_{j=1}^N\tilde{b}(\bar{X}_t^i,\bar{X}_t^j)\Big| \rmd t+(1+\alpha)\gamma\xi\rmd t+ \rmd M_t^i,
\end{align*}
where $f$ is given in \eqref{eq:f}, $(M_t^i)_{t\geq 0}$ is a martingale and $c_2$ is given in \eqref{eq:c_2}. 
\end{lemma} 

\begin{proof}
The proof works similarly as the proof of \Cref{lem:contr_inside_nonl}.
First, note that for all $i$, $(Z_t^i)_{t\geq 0}$ is  almost surely continuously differentiable with derivative $\rmd Z^i/\rmd t=-\gamma Z^i+\gamma Q^i$ and hence $t\to|Z_t^i|$ is almost surely absolutely continuous with 
\begin{align*}
&\frac{\rmd}{\rmd t}|Z_t^i|=\frac{Z_t^i}{|Z_t^i|}\cdot(-\gamma Z_t^i +\gamma Q_t^i) &&\text{ for a.e. }t \text{ such that } Z_t^i\neq 0 \text{ and }
\\ & \frac{\rmd }{\rmd t} |Z_t^i| \leq \gamma |Q_t^i| &&\text{ for a.e. }t\text{ such that } Z_t^i = 0.
\end{align*}
and therefore
\begin{align} \label{eq:Z_t^i}
\frac{\rmd}{\rmd t}|Z_t^i|\leq -\gamma |Z_t^i|+\gamma |Q_t^i| \qquad \text{ for a.e. } t\geq 0.
\end{align}
By Ito's formula and by \Cref{ass:b} and \Cref{ass:tildeb}, we obtain for $|Q_t^i|$,
\begin{align*}
\rmd |Q_t^i|&= \gamma^{-1} u{e_t^i}\cdot\Big(b(\bar{X}_t^i)-b(X_t^i)+\int_{\mathbb{R}^d}\tilde{b}(\bar{X}_t^i,z)\bar{\mu}_t^x (\rmd z)-N^{-1}\sum_{j=1}^N\tilde{b}(X_t^i,X_t^j)\Big)\rmd t
 +\sqrt{8\gamma^{-1} u}\rc(Z_t^i,W_t^i){e_t^i}^T\rmd B_t^i
\\ & \leq \gamma^{-1}u(L_K+L_g)|Z_t^i|\rmd t+\gamma^{-1}u(A_t^i+N^{-1}\sum_{j=1}^N\tilde{L}(|Z_t^j|+|Z_t^i|))\rmd t +\sqrt{8\gamma^{-1}u}\rc(Z_t^i,W_t^i){e_t^i}^T\rmd B_t^{i,\rc},
\end{align*}
where $A_t^i$ is given by \eqref{eq:bound_A_t^i}. Note that there is no Ito correction term, since $\partial^2_{q/|q|}|q|=0$ for $q\neq 0$ and $\rc=0$ for $Q_t=0$. Combining this bound and \eqref{eq:Z_t^i} yields for $f(r_s^i(t))$ by Ito's formula,
\begin{align*}
\rmd f(r_s^i(t))&=f'(r_s^i(t))\Big(((L_K+L_g)u\gamma^{-2}-\alpha)\gamma|Z_t^i|+\alpha\gamma|Q_t^i|+\gamma^{-1} u\Big(A_t^i+N^{-1}\sum_{j=1}^N\tilde{L}(|Z_t^j|+|Z_t^i|)\Big)\Big) \rmd t
\\ & +f'(r_s^i(t))\sqrt{8\gamma^{-1} u}\mathrm{rc}(Z_t^i,W_t^i)(e_t^i)^T \rmd B_t^{i,\mathrm{rc}}+f''(r_s^i(t))4\gamma^{-1} u\mathrm{rc}(Z_t^i,W_t^i)^2\rmd t.
\end{align*}
\textit{Case 1:} Consider $\Delta^i(t)< D_{\mathcal{K}}$ and $|Q_t^i|>\xi$, then $\rc(Z_t^i,W_t^i)=1$ and $r_s^i(t)<R_1$. Hence, by \eqref{eq:f_property2} we obtain
\begin{align*}
 \rmd f(r_s^i(t))&\leq f'(r_s^i(t))\alpha\gamma r_s^i(t) \rmd t+f''(r_s^i(t))4\gamma^{-1} u \rmd t +\gamma^{-1} u\Big(A_t^i+N^{-1}\sum_{j=1}^N\tilde{L}(|Z_t^j|+|Z_t^i|)\Big)\rmd t
 \\ & \quad -f'(R_1)\frac{1}{2}\gamma\alpha|Z_t^i|\rmd t+\rmd M_t^i
 \\ &\leq -2\hat{c}f(r_s^i(t))\rmd t+\gamma^{-1} u\Big(A_t^i+N^{-1}\sum_{j=1}^N\tilde{L}(|Z_t^j|+|Z_t^i|)\Big)\rmd t-f'(R_1)\frac{\gamma\alpha}{2}|Z_t^i|\rmd t+\rmd M_t^i
 \\ &\leq -c_2f(r_s^i(t))\rmd t+\gamma^{-1} u\Big(A_t^i+N^{-1}\sum_{j=1}^N\tilde{L}(|Z_t^j|+|Z_t^i|)\Big)\rmd t-f'(R_1)\frac{\gamma\alpha}{2}|Z_t^i|\rmd t+\rmd M_t^i.
\end{align*}
\textit{Case 2:} Consider $\Delta^i(t)< D_{\mathcal{K}}$ and $|Q_t^i|\leq \xi$, then $\alpha|Z_t^i|= r_s^i(t)-|Q_t^i|\geq r_s^i(t)-\xi$.
We note that
\begin{align*}
((L_K+L_g) u\gamma^{-2}-\alpha)|Z_t^i|+\alpha|Q_t^i|\leq -\frac{1}{2}r_s^i(t)+(1+\alpha)\xi.
\end{align*}
Since the second derivative of $f$ is negative and $\psi(s)\in[1/2,1]$, it holds
\begin{align*}
 \rmd f(r_s^i(t))&\leq -\frac{ \gamma}{2}r_s^i(t)f'(r_s(t))\rmd t+ (1+\alpha)\gamma\xi \rmd t+\gamma^{-1} u\Big(A_t^i+N^{-1}\sum_{j=1}^N\tilde{L}(|Z_t^j|+|Z_t^i|)\Big)\rmd t+\rmd M_t^i
 \\ & \leq  -\frac{\gamma}{8}\inf_{r\leq R_1}\frac{r\phi(r)}{\Phi(r)}f(r_s^i(t))\rmd t-\frac{\gamma\alpha}{4}|Z_t^i|f'(R_1)\rmd t+ (1+\alpha)\gamma\xi \rmd t
 \\ & \quad+\gamma^{-1} u\Big(A_t^i+N^{-1}\sum_{j=1}^N\tilde{L}(|Z_t^j|+|Z_t^i|)\Big)\rmd t+\rmd M_t^i
  \\ & \leq  -\frac{\gamma}{8}\frac{R_1\phi(R_1)}{\Phi(R_1)}f(r_s^i(t))\rmd t-\frac{\gamma\alpha}{4}|Z_t^i|f'(R_1)\rmd t+ (1+\alpha)\gamma\xi \rmd t
  \\ & \quad +\gamma^{-1} u\Big(A_t^i+N^{-1}\sum_{j=1}^N\tilde{L}(|Z_t^j|+|Z_t^i|)\Big)\rmd t+\rmd M_t^i. 
\end{align*}
Combining the two cases, we obtain the result by using the definition of $c_2$ given in \eqref{eq:c_2}.
\end{proof}

\begin{lemma} \label{lem:momentbound}(Moment control for Langevin dynamics)
Suppose that \Cref{ass:b} and \Cref{ass:tildeb} hold. 
Suppose that \eqref{eq:cond_gamma} and \eqref{eq:condition_tildeL} hold. Let $(\bar{X}_t,\bar{Y}_t)_{t\geq 0}$ be a solution to \eqref{eq:KFP} with $\mathbb{E}[|\bar{X}_0|^2+|\bar{Y}_0|^2]\leq \infty$. Then there exists a finite constant $\mathcal{C}_2>0$ such that 
\begin{align*}
\sup_{t\geq 0} \mathbb{E}[|\bar{X}_t|^2]\leq \mathcal{C}_2.
\end{align*}
The constant $\mathcal{C}_2$ depends on $\gamma$, $\mathbb{E}[|\bar{X}_0|^2+|\bar{Y}_0|^2]$, $d$, $R$, $\kappa$, $L_g$, $u$ and $\tilde{L}$.
\end{lemma}

\begin{proof}
We adapt the proof idea from \cite[Lemma 8]{DuEbGuZi20}. 
By Ito's formula, by \Cref{ass:b} and by \Cref{ass:tildeb}, it holds
\begin{align*}
& \rmd (\gamma^{-2} u\bar{X}_t\cdot (K\bar{X}_t)+\frac{1}{2}|(1-2\tau)\bar{X}_t+\gamma^{-1}\bar{Y}_t|^2+\frac{1}{2}\gamma^{-2}|\bar{Y}_t|^2)
\\ & \leq \Big(2\gamma^{-2} u\bar{X}_t\cdot(K\bar{Y}_t)+(1-2\tau)^2\bar{X}_t\cdot\bar{Y}_t+\gamma^{-1}(1-2\tau)|\bar{Y}_t|^2\Big)\rmd t
 + \gamma^{-1}(1-2\tau)\Big(- u \bar{X}_t\cdot (K\bar{X}_t)-\gamma\bar{X}_t\cdot\bar{Y}_t\Big)\rmd t
\\ & +2\gamma^{-2}\Big(-u(K \bar{Y}_t)\cdot\bar{X}_t+L_g|\bar{Y}_t||\bar{X}_t|-\gamma|\bar{Y}_t|^2\Big)\rmd t +\frac{u}{\gamma}|(1-2\tau)\bar{X}_t+2\gamma^{-1}\bar{Y}_t|  \Big(\tilde{L} (\mathbb{E}[|\bar{X}_t|]+|X_t|)+|\tilde{b}(0,0)|\Big)\rmd t 
\\ & +(1-2\tau)\gamma^{-1} u(L_g|\bar{X}_t|^2+|g(0)||\bar{X}_t|) \1_{\{|\bar{X}_t|<R\}}\rmd t+2\gamma^{-2} u|\bar{Y}_t||g(0)|\rmd t+2\gamma^{-1} u d \rmd t
\\ & +\sqrt{2\gamma^{-1} u}((1-2\tau)\bar{X}_t+2\gamma^{-1}\bar{Y}_t) \rmd B_t
\\ & \le -\gamma^{-1}u(1-2\tau) \bar{X}_t\cdot(K\bar{X}_t)-2\tau\gamma(\gamma^{-2}|\bar{Y}_t|^2+(1-2\tau)\gamma^{-1}\bar{X}_t\cdot \bar{Y}_t)+\gamma^{-3}u^2L_g^2|\bar{X}_t|^2
\\ & +\frac{u}{\gamma}|(1-2\tau)\bar{X}_t+2\gamma^{-1}\bar{Y}_t|  (\tilde{L} (\mathbb{E}[|\bar{X}_t|]+|X_t|)+|\tilde{b}(0,0)|)\rmd t 
 +(1-2\tau)\gamma^{-1} u(L_g|\bar{X}_t|^2+|g(0)||\bar{X}_t|) \1_{\{|\bar{X}_t|<R\}}\rmd t
 \\ &+2\gamma^{-2} u|\bar{Y}_t||g(0)|\rmd t+2\gamma^{-1} u d \rmd t +\sqrt{2\gamma^{-1} u}((1-2\tau)\bar{X}_t+2\gamma^{-1}\bar{Y}_t) \rmd B_t
\end{align*}
 Taking expectation, 
 we obtain
\begin{align*}
&\frac{\rmd}{\rmd t}\mathbb{E}[\gamma^{-2} u\bar{X}_t\cdot (K\bar{X}_t)+\frac{1}{2}|(1-2\tau)\bar{X}_t+\gamma^{-1}\bar{Y}_t|^2+\frac{1}{2}\gamma^{-2}|\bar{Y}_t|^2]
\\& \leq -\gamma^{-1} u(1-2\tau)\mathbb{E}[\bar{X}_t\cdot (K\bar{X}_t)]+\gamma^{-3}u^2 L_g^2\mathbb{E}[|\bar{X}_t|^2]
-2\tau\gamma\Big(\gamma^{-2} \mathbb{E}[|\bar{Y}_t|^2]+(1-2\tau)\gamma^{-1}\mathbb{E}[\bar{X}_t\cdot\bar{Y}_t]\Big)
\\ & 
+(1-2\tau)\gamma^{-1} u (L_gR^2+R|g(0)|)+2\gamma^{-1} ud+u\gamma^{-1}\mathbb{E}\Big[|(1-2\tau)\bar{X}_t+2\gamma^{-1}\bar{Y}_t| \Big(\tilde{L}(\mathbb{E}[|X_t|]+|X_t|)+|\tilde{b}(0,0)|\Big)\Big]
\\ & +2\gamma^{-2} u\mathbb{E}[|\bar{Y}_t|]|g(0)|.
\end{align*} 
We note that by \eqref{eq:condition_tildeL} and by Young's inequality,
\begin{align*}
\gamma^{-1}&\mathbb{E}[|(1-2\tau)\bar{X}_t+2\gamma^{-1}\bar{Y}_t|u(\tilde{L}(\mathbb{E}[|\bar{X}_t|]+|\bar{X}_t|)+|\tilde{b}(0,0)|)]
\\ & \le \frac{\tau \sqrt{\kappa u}}{8}\mathbb{E}[|(1-2\tau)\bar{X}_t+2\gamma^{-1}\bar{Y}_t|(\mathbb{E}[|\bar{X}_t|]+|\bar{X}_t|)]+\gamma^{-1} u\mathbb{E}[|(1-2\tau)\bar{X}_t+2\gamma^{-1}\bar{Y}_t|]|\tilde{b}(0,0)|
\\ & \le \frac{\tau\gamma}{4}\Big(\kappa u\gamma^{-2}\mathbb{E}[|\bar{X}_t|^2]+\frac{1}{4}\mathbb{E}[|(1-2\tau)\bar{X}_t+2\gamma^{-1}\bar{Y}_t|^2]\Big)
 +\frac{\tau\gamma}{4}\frac{1}{4}\mathbb{E}[|(1-2\tau)\bar{X}_t+2\gamma^{-1}\bar{Y}_t|^2]+\frac{4 u^2}{\tau\gamma^3}|\tilde{b}(0,0)|^2
\\ & \le \frac{\tau\gamma}{2}\Big(\kappa u\gamma^{-2}\mathbb{E}[|\bar{X}_t|^2]+\frac{1}{2}\mathbb{E}[|(1-2\tau)\bar{X}_t+\gamma^{-1}\bar{Y}_t|^2]+\frac{1}{2}\mathbb{E}[|\gamma^{-1}\bar{Y}_t|^2]\Big)+\frac{4 u^2}{\tau\gamma^3}|\tilde{b}(0,0)|^2
\end{align*}
and 
\begin{align*}
2\gamma^{-2} u \mathbb{E}[|\bar{Y}_t|]|g(0)|\le \frac{\tau \gamma}{2}\frac{1}{2}\mathbb{E}[|\gamma^{-1}\bar{Y}_t|^2]+\frac{4 u^2}{\tau \gamma^3}|g(0)|^2.
\end{align*}
%
Then by \eqref{eq:estimatemu},
\begin{align*}
\frac{\rmd}{\rmd t}&\mathbb{E}\Big[\gamma^{-2} u \bar{X}_t\cdot (K\bar{X}_t)+\frac{1}{2}|(1-2\tau)\bar{X}_t+\gamma^{-1}\bar{Y}_t|^2+\frac{1}{2}\gamma^{-2}|\bar{Y}_t|^2\Big]
\\ & \leq -2\tau\gamma\mathbb{E}\Big[\gamma^{-2} u\bar{X}_t\cdot (K\bar{X}_t)+\frac{1}{2}|(1-2\tau)\bar{X}_t+\gamma^{-1}\bar{Y}_t|^2+\frac{1}{2}\gamma^{-2}|\bar{Y}_t|^2\Big]
+(1-2\tau)\gamma^{-1} uL_gR^2+2\gamma^{-1}u d
\\ & \quad +\tau\gamma\Big(\kappa\gamma^{-2}\mathbb{E}[|\bar{X}_t|^2]+\frac{1}{2}\mathbb{E}[|(1-2\tau)\bar{X}_t+\gamma^{-1}\bar{Y}_t|^2+|\gamma^{-1}\bar{Y}_t|^2]\Big)+4\tau^{-1}\gamma^{-3}u^2(|\tilde{b}(0,0)|^2+|g(0)|^2)
\\ & \le -\tau\gamma\mathbb{E}\Big[\gamma^{-2}u\bar{X}_t\cdot (K\bar{X}_t)+\frac{1}{2}|(1-2\tau)\bar{X}_t+\gamma^{-1}\bar{Y}_t|^2+\frac{1}{2}\gamma^{-2}|\bar{Y}_t|^2\Big]
+(1-2\tau)\gamma^{-1}L_g u R^2+2\gamma^{-1} ud
\\ & \quad+4\tau^{-1}\gamma^{-3}u^2(|\tilde{b}(0,0)|^2+|g(0)|^2).
\end{align*}
By Gr\"{o}nwall's inequality, there exists a constant $\mathbf{C}$ such that
\begin{align*}
\sup_{t\geq0}\mathbb{E}\Big[\gamma^{-2}u\bar{X}_t\cdot (K\bar{X}_t)+\frac{1}{2}|(1-2\tau)\bar{X}_t+\gamma^{-1}\bar{Y}_t|^2+\frac{1}{2}\gamma^{-2}|\bar{Y}_t|^2\Big]\leq \mathbf{C}<\infty.
\end{align*}
Thus, we obtain the result for $\mathcal{C}_2=\mathbf{C}/(\kappa u\gamma^{-2})$.
\end{proof}



\begin{proof}[Proof of \Cref{thm:propofchaos}]
To prove uniform in time propagation of chaos, we consider the coupling \\ $(\{(\bar{X}_t^i,\bar{Y}_t^i),(X_t^i,Y_t^i)\}_{i=1}^N)_{t\geq 0}$ given in \eqref{eq:KFP_coupling} and combine the results of \Cref{lem:contr_outside} and \Cref{lem:contr_inside}. The second moment control given in \Cref{lem:momentbound} will be essential to bound the terms involving the non-linearity. We write here $r_s^i(t)=r_s^i((\bar{X}_t^i,\bar{Y}_t^i),(X_t^i,Y_t^i))$,  $r_l^i(t)=r_l^i((\bar{X}_t^i,\bar{Y}_t^i),(X_t^i,Y_t^i))$, $\Delta^i(t)=r_s^i(t)-\epsilon r_l^i(t)$ and $\rho^i(t)=f((\Delta^i(t)\wedge D_{\mathcal{K}})+\epsilon r_l^i(t))$.
We distinguish two cases for all particles $i=1,...,N$: \\
\textit{Case 1:} Consider $\Delta^i(t) < D_{\mathcal{K}}$. Then 
$\rho^i(t)=f(r_s^i(t))$, and
by \Cref{lem:contr_inside} it holds for $\xi>0$
\begin{align} 
\rmd\rho^i(t)&=\rmd f(r_s^i(t))\leq -c_2 f(r_s^i(t))\rmd t + \gamma^{-1}u\Big(A_t^i+N^{-1}\sum_{j=1}^N\tilde{L}(|Z_t^j|+|Z_t^i|)\Big) \rmd t-\frac{
\alpha\gamma}{4}f'(R_1)|Z_t^i|\rmd t \nonumber
\\ & \quad+(1+\alpha)\gamma\xi \rmd t + \rmd M_t^i \nonumber
\\ & \leq -c_2 f(r_s^i(t))\rmd t + \gamma^{-1}u\Big(A_t^i+N^{-1}\sum_{j=1}^N\tilde{L}|Z_t^j|\Big) \rmd t-\frac{\alpha\gamma}{8}f'(R_1)|Z_t^i|\rmd t +(1+\alpha)\gamma\xi \rmd t + \rmd M_t^i, \label{eq:proof1_6}
\end{align}
where $A_t^{i}$ is given in \eqref{eq:bound_A_t^i} and  $c_2$ is given by \eqref{eq:c_2}. Note the last step holds by \eqref{eq:condition_tildeL}. \\
\textit{Case 2:} Consider $\Delta^i(t)\geq D_{\mathcal{K}}$. 
We obtain by \Cref{lem:contr_outside},
\begin{align*}
\rmd r_l^i(t)&\leq -c_1 r_l^i(t) \rmd t
 +  \frac{|(1-2\tau)Z_t^i+2\gamma^{-2} W_t^i|}{2\gamma r_l^i(t)}u\Big(A_t^i+N^{-1}\sum_{j=1}^N\tilde{L}(|Z_t^j|+|Z_t^i|)\Big) \rmd t
 \\ &  + \sqrt{2\gamma^{-1}u}\rc(Z_t^i,W_t^i)r_l^i(t)^{-1}((1-2\tau)Z_t^i+2\gamma^{-1}W_t^i)\cdot e_t^i{e_t^i}^T\rmd B_t^i
\end{align*}
with $c_1$ given in \Cref{lem:contr_outside}. 
Note that $\frac{\rmd}{\rmd x} f(D_{\mathcal{K}}+\epsilon x)=\epsilon f'(D_{\mathcal{K}}+\epsilon x)$. Further, since $f(D_{\mathcal{K}}+\epsilon x)$ is a concave function, $\frac{\rmd^2}{\rmd x^2} f(D_{\mathcal{K}}+\epsilon x)$ is negative. 
By Ito's formula, we obtain
\begin{align*}
\rmd \rho^i(t)&= \rmd f(D_{\mathcal{K}}+\epsilon r_l^i(t))
\\ & \leq \epsilon f'(D_{\mathcal{K}}+\epsilon r_l^i(t))\Big(-c_1 r_l^i(t)^2
 + \frac{|(1-2\tau)Z_t^i+2\gamma^{-1} W_t^i|}{2\gamma r_l^i(t)}u\Big(A_t^i+N^{-1}\sum_{j=1}^N\tilde{L}(|Z_t^j|+|Z_t^i|)\Big) \Big)\rmd t
 \\ & +  \frac{\epsilon f'(D_{\mathcal{K}}+\epsilon r_l^i(t))}{r_l^i(t)}\sqrt{2\gamma^{-1}u}\rc(Z_t^i,W_t^i)((1-2\tau)Z_t^i+2\gamma^{-1}W_t^i)\cdot e_t^i{e_t^i}^T\rmd B_t^i.
\end{align*} 
By \eqref{eq:proof1_1} and \eqref{eq:proof1_3}, which holds in the same line as in the proof of \Cref{thm:contraction_nonlinear_conf}, it holds 

\begin{equation} \label{eq:proof1_8}
\begin{aligned}
\rmd \rho^i(t)&\leq -f'(R_1)\frac{c_1\epsilon}{2}\min\Big(\frac{\sqrt{\kappa u} \gamma^{-1}}{\sqrt{8}\alpha},\frac{1}{2}\Big) \rho^i(t)\rmd t-f'(R_1)\frac{c_1\epsilon}{2}\sqrt{\kappa u\gamma^{-2}}|Z_t^i|\rmd t
\\ &  +2\epsilon\gamma^{-1}u\Big(A_t^i+N^{-1}\sum_{j=1}^N\tilde{L}(|Z_t^j|+|Z_t^i|)\Big)\rmd t +\rmd M_t^i,
\end{aligned}
\end{equation}
where $(\{M_t^i\}_{i=1}^N)_{t\ge 0}$ is some martingale.

Combining \eqref{eq:proof1_6} and \eqref{eq:proof1_8}, taking expectations and summing over $i=1,\ldots,N$ yields
\begin{equation} \label{eq:proof1_5}
\begin{aligned}
\frac{\rmd }{\rmd t} \mathbb{E}\Big[N^{-1}\sum_{i=1}^N \rho^i(t)\Big]&\leq -\min\Big(c_2,
f'(R_1)\frac{c_1\epsilon}{2}\min\Big(\frac{\sqrt{\kappa u} \gamma^{-1}}{\sqrt{8}\alpha},\frac{1}{2}\Big)\Big)\mathbb{E}\Big[N^{-1}\sum_{i=1}^N \rho^i(t)\Big] +\gamma^{-1}u\mathbb{E}\Big[N^{-1}\sum_{i=1}^N A_t^i\Big]
\\ & -\min\Big(f'(R_1)\frac{\gamma\alpha}{8},f'(R_1)\frac{c_1\epsilon}{2}\sqrt{\kappa u\gamma^{-2}}\Big)\mathbb{E}\Big[N^{-1}\sum_{i=1}^N |Z_t^i|\Big]+\tilde{L}u\gamma^{-1}\mathbb{E}\Big[N^{-1}\sum_{i=1}^N |Z_t^i|\Big]
\\ & \leq  -\min\Big(c_2,
f'(R_1)\frac{c_1\epsilon}{4}\min\Big(\frac{\sqrt{\kappa u} \gamma^{-1}}{\sqrt{8}\alpha},\frac{1}{2}\Big)\Big)\mathbb{E}\Big[N^{-1}\sum_{i=1}^N \rho^i(t)\Big] +\gamma^{-1} u\mathbb{E}\Big[N^{-1}\sum_{i=1}^N A_t^i\Big],
\end{aligned}
\end{equation}
where we used $2\epsilon\leq 1$ for the last term and \eqref{eq:condition_tildeL}.

To bound $\mathbb{E}[A_t^i]$, 
we note that given $\bar{X}_t^i$, $\bar{X}_t^j$, $j\neq i$ are identically and independent distributed with law $\bar{\mu}_t^x$ and
\begin{align}\label{eq:boundA_t^i1}
\mathbb{E}[\tilde{b}(\bar{X}_t^i,\bar{X}_t^j)|\bar{X}_t^i]=\int_{\mathbb{R}^d}\tilde{b}(\bar{X}_t^i,z)\bar{\mu}_t^x(\rmd z).
\end{align}
Hence, 
\begin{equation*} 
\begin{aligned}
\mathbb{E}\Big[|&\int_{\mathbb{R}^d}\tilde{b}(\bar{X}_t^i,z)\bar{\mu}_t^x(\rmd z)-\frac{1}{N}\sum_{j=1}^N\tilde{b}(\bar{X}_t^i,\bar{X}_t^j)|^2 \Big|\bar{X}_t^i\Big]
\\ & =\frac{N-1}{N^2} \Var_{\bar{\mu}_t^x}(\tilde{b}(\bar{X}_t^i, \cdot))
+\frac{1}{N^2}\mathbb{E}\Big[|\int_{\mathbb{R}^d}\tilde{b}(\bar{X}_t^i,z)\bar{\mu}_t^x(\rmd z)-\tilde{b}(\bar{X}_t^i,\bar{X}_t^i)|^2 \Big|\bar{X}_t^i\Big]
\\ & +\frac{2}{N^2} \sum_{j=1,j\neq i}^N\mathbb{E}\Big[|\int_{\mathbb{R}^d}\tilde{b}(\bar{X}_t^i, z)\bar{\mu}_t^x(\rmd z)-\tilde{b}(\bar{X}_t^i,\bar{X}_t^j)|\cdot|\int_{\mathbb{R}^d}\tilde{b}(\bar{X}_t^i,z)\bar{\mu}_t^x(\rmd z)-\tilde{b}(\bar{X}_t^i,\bar{X}_t^i)| \Big|\bar{X}_t^i\Big]
%
\end{aligned}
\end{equation*}
By \Cref{ass:tildeb}, Cauchy inequality and Young's inequality
\begin{equation} \label{eq:boundA_t^i2}
\begin{aligned}
\mathbb{E}\Big[\Big|\int_{\mathbb{R}^d}\tilde{b}(\bar{X}_t^i,z) \bar{\mu}_t^x(\rmd z)-\frac{1}{N}\sum_{j=1}^N\tilde{b}(\bar{X}_t^i,\bar{X}_t^j)\Big|^2\Big]&\leq \frac{4\tilde{L}^2}{N}\int_{\mathbb{R}^d} |x|^2\bar{\mu}_t^x(\rmd x)+ \frac{4\tilde{L}^2}{N^2}\int_{\mathbb{R}^d} |x|^2\bar{\mu}_t^x(\rmd x)
\\ & + \frac{8 \tilde{L}^2}{N}\int_{\mathbb{R}^d} |x|^2\bar{\mu}_t^x(\rmd x).
\end{aligned}
\end{equation}
Then, by Jensen's inequality 
\begin{align*}
\mathbb{E}[A_t^i]\leq \frac{4\tilde{L}}{N^{1/2}}\Big(\int_{\mathbb{R}^d} |x|^2\bar{\mu}_t^x(\rmd x)\Big)^{1/2}.
\end{align*}
By \Cref{lem:momentbound}, there exists a finite constant $\mathcal{C}_1$ such that for $N\geq 2$ and all $i=1,...,N$,
\begin{align} \label{eq:boundA_t^i}
\sup_{t\geq 0}\mathbb{E}[A_t^i]\leq \gamma u^{-1}\mathcal{C}_1N^{-1/2}.
\end{align}
Note that $\mathcal{C}_1$ depends on $\gamma$, $\mathbb{E}[|\bar{X}_0|^2+|\bar{Y}_0|^2]$, $d$, $u$, $R$, $\kappa$, $L_g$ and $\tilde{L}$. Inserting the bound for $\mathbb{E}[A_t^i]$ in \eqref{eq:proof1_5} yields
\begin{align*}
\frac{\rmd }{\rmd t} \mathbb{E}\Big[N^{-1}\sum_{i=1}^N \rho^i(t)\Big]&\leq -\min\Big(c_2,
f'(R_1)\frac{c_1\epsilon}{2}\min\Big(\frac{\sqrt{\kappa u} \gamma^{-1}}{\sqrt{8}\alpha},\frac{1}{2}\Big)\Big)\mathbb{E}\Big[N^{-1}\sum_{i=1}^N \rho^i(t)\Big]
 +\frac{\mathcal{C}_1}{N^{1/2}}.
\end{align*}
Applying Gr\"{o}nwall's inequality and \eqref{eq:c_simplify1} and \eqref{eq:c_simplify2} yields
\begin{align*}
\mathcal{W}_{1,\rho_N}(\bar{\mu}_t^{\otimes N},\mu_t^N )\leq \mathbb{E}\Big[N^{-1}\sum_{i=1}^N \rho^i(t)\Big]&\leq e^{-\tilde{c}t}\mathbb{E}\Big[N^{-1}\sum_{i=1}^N \rho^i(0)\Big] +\mathcal{C}_1N^{-1/2}\tilde{c}^{-1}.
\end{align*}
with $\tilde{c}$ given in \eqref{eq:c_thm_nonl}.
Taking the infimum over all couplings $\omega\in\Pi(\bar{\mu}_0^{\otimes},\mu_0^N)$ concludes the proof of the first result.

The second bound holds by \eqref{eq:rhoN_equivalence} with $M_1$ given in \eqref{eq:M1} and $M_2=\sqrt{2}/\mathbf{C}_1$ given in \eqref{eq:M2}.

\end{proof}

\appendix{
\section{Unconfined nonlinear Langevin dynamics} \label{app:unconf}

\subsection{Contraction for unconfined nonlinear Langevin dynamics} \label{sec:contr_unconf}

Consider the unconfined nonlinear Langevin dynamics given by 
\begin{equation} \label{eq:KFP_unconf}
\begin{cases}
& \rmd \bar{X}_t=\bar{Y}_t \rmd t
\\ & \rmd \bar{Y}_t=(-\gamma \bar{Y}_t +u \int_{\mathbb{R}^d} \tilde{b}(\bar{X}_t,z) \bar{\mu}_t^x(\rmd z))\rmd t +\sqrt{2\gamma u}\rmd B_t, \qquad (\bar{X}_0,\bar{Y}_0)\sim \bar{\mu}_0,
\end{cases}
\end{equation}
where $\gamma, u >0$, $\bar{\mu}_0$ is a probability measure on $\mathbb{R}^{2d}$, $\bar{\mu}_t^x=\Law(\bar{X}_t)$ and $(B_t)_{t\ge 0}$ is a $d$-dimensional standard Brownian motion.
We impose for the function $\tilde{b}$ and for the initial distribution:

\begin{assumption}\label{ass:W} 
The function $\tilde{b}:\mathbb{R}^{2d}\to \mathbb{R}^d$ is Lipschitz continuous, and there exist a function $\tilde{g}:\mathbb{R}^d\to\mathbb{R}^d$ and a positive definite matrix $\tilde{K}\in\mathbb{R}^{d\times d}$ with smallest eigenvalue $\tilde{\kappa}\in(0,\infty)$ and largest eigenvalue $L_{\tilde{K}}\in(0,\infty)$ such that
\begin{align*}
\tilde{b} (x,y)=-\tilde{K}(x-y)+\tilde{g}(x-y) \qquad \text{ for all } x, y\in\mathbb{R}^d,
\end{align*}
and $\tilde{g}$ is Lipschitz continuous with Lipschitz constant $L_{\tilde{g}}\in(0,\infty)$ and anti-symmetric, i.e., $\tilde{g}(-z)=-\tilde{g}(z)$ for all $z\in\mathbb{R}^d$.
\end{assumption}

\begin{assumption}\label{ass:init_distrib}
Let $\bar{\mu}_0\in\mathcal{P}(\mathbb{R}^{2d})$ satisfy $\int_{\mathbb{R}^{2d}} |(x,y)|^2\bar{\mu}_0(\rmd x \rmd y)<\infty$ and  $\int_{\mathbb{R}^{2d}} (x, y) \bar{\mu}_0(\rmd x \rmd y)=0$.
\end{assumption}

By \Cref{ass:W}, it holds $\frac{\rmd}{\rmd t} \mathbb{E}[(X_t,Y_t)]=\mathbb{E}[(Y_t,-\gamma Y_t)]$ and hence by \Cref{ass:init_distrib} $\mathbb{E}[(X_t,Y_t)]=0$ for all $t \ge 0$. Note that this observation is crucial in our analysis, since in general convergence to equilibrium can not be guaranteed for the unconfined dynamics unless the solution is centered or a recentering of the center of mass is considered.


We establish contraction in Wasserstein distance with respect to the distance function $\tilde{r}:\mathbb{R}^{2d}\times\mathbb{R}^{2d}\to[0,\infty)$ 
given by 
\begin{equation}\label{eq:tilder}
\begin{aligned}
\tilde{r}((x,y),(\bar{x},\bar{y}))^2
=\gamma^{-2}u(x-\bar{x})\cdot(\tilde{K}(x-\bar{x}))+\frac{1}{2}|(1-2\sigma)(x-\bar{x})+\gamma^{-1}(y-\bar{y})|^2+\frac{1}{2}\gamma^{-2}|y-\bar{y}|^2, 
\end{aligned}
\end{equation}
for $(x,y),(\bar{x},\bar{y})\in\mathbb{R}^{2d}$ where $\sigma$ is given by
\begin{align} \label{eq:sigma}
\sigma=\min(1/8,\tilde{\kappa}u\gamma^{-2}/2).
\end{align}

\begin{theorem}[Contraction for nonlinear unconfined Langevin dynamics in $L^2$ and $L^1$ Wasserstein distance]\label{thm:contr_unconfined}
Suppose \Cref{ass:W} holds. Let $\bar{\mu}_0$ and $\bar{\nu}_0$ be two probability distributions on $\mathbb{R}^{2d}$ satisfying \Cref{ass:init_distrib}. 
For $t\geq 0$, let $\bar{\mu}_t$ and $\bar{\nu}_t$ be the law of the processes $(\bar{X}_t,\bar{Y}_t)$ and $(\bar{X}_t',\bar{Y}_t')$, respectively, where $(\bar{X}_s,\bar{Y}_s)_{s\geq 0}$ and $(\bar{X}_s',\bar{Y}_s')_{s\geq 0}$ are solutions to \eqref{eq:KFP_unconf} with initial distribution $\bar{\mu}_0$ and $\bar{\nu}_0$, respectively. 
If
\begin{align} \label{eq:condition_LGamma1}
L_{\tilde{g}}\leq \sqrt{\tilde{\kappa}/u}(\gamma/2)\min(1/8,\tilde{\kappa}u\gamma^{-2}/2), 
\end{align}
then
\begin{align} \label{eq:contra_unconf_L2}
\mathcal{W}_{2,\tilde{r}}(\bar{\mu}_t,\bar{\nu}_t)\leq e^{-\hat{c}t}\mathcal{W}_{2,\tilde{r}}(\bar{\mu}_0,\bar{\nu}_0) \qquad \text{and} \qquad
\mathcal{W}_{2}(\bar{\mu}_t,\bar{\nu}_t)\leq M_3 e^{-\hat{c}t}\mathcal{W}_{2}(\bar{\mu}_0,\bar{\nu}_0),
\end{align}
where $\tilde{r}$ is defined in \eqref{eq:tilder} and where the contraction rate $\hat{c}$ is given by
\begin{align} \label{eq:tildec}
\hat{c}=\min(\gamma/16, \tilde{\kappa}\gamma^{-1}/4).
\end{align} 
The constant  $M_3$ is given by
\begin{align} \label{eq:M_2}
M_3=\max(\sqrt{L_{\tilde{K}}u+\gamma^2},\sqrt{3/2})\max(\sqrt{(\tilde{\kappa}u)^{-1}},\sqrt{2}).
\end{align}
Moreover, there exists a unique invariant probability measure $\bar{\mu}_\infty$ for \eqref{eq:KFP_unconf} and convergence in $L^2$ Wasserstein distance to $\bar{\mu}_\infty$ holds.

If
\begin{align} \label{eq:condition_LGamma2}
L_{\tilde{g}}\leq \sqrt{\tilde{\kappa}/u}(\gamma/4)\min(1/8,\tilde{\kappa}\gamma^{-2}/2),
\end{align}
then
\begin{align} \label{eq:contra_unconf_L1}
\mathcal{W}_{1,\tilde{r}}(\bar{\mu}_t,\bar{\nu}_t)\leq e^{-\hat{c}t}\mathcal{W}_{1,\tilde{r}}(\bar{\mu}_0,\bar{\nu}_0) \qquad \text{and} \qquad
\mathcal{W}_{1}(\bar{\mu}_t,\bar{\nu}_t)\leq M_3 e^{-\hat{c}t}\mathcal{W}_{1}(\bar{\mu}_0,\bar{\nu}_0)
\end{align}
and convergence in $L^1$ Wasserstein distance to $\bar{\mu}_\infty$ holds.
\end{theorem}
\begin{proof}
The proof uses a synchronous coupling and is postponed to \Cref{sec:proof_unconf_contr}.
\end{proof}

\begin{remark}
Note that \eqref{eq:contra_unconf_L2} implies directly a bound in $L^p$ Wasserstein distance for $1 \le p < 2$, i.e., by Jensen's inequality it holds $\mathcal{W}_{p}(\bar{\mu}_t,\bar{\nu}_t) \le \mathcal{W}_{2}(\bar{\mu}_t,\bar{\nu}_t) \leq M_0 M_3 e^{-\hat{c}t}\mathcal{W}_{p}(\bar{\mu}_0,\bar{\nu}_0)$, where $M_0=\mathcal{W}_{2}(\bar{\mu}_0,\bar{\nu}_0)/\mathcal{W}_{p}(\bar{\mu}_0,\bar{\nu}_0)$. The additional constant $M_0$ is finite by \Cref{ass:init_distrib}, but it might be very large. Here, contraction in $L^1$ Wasserstein distance is stated separately and \eqref{eq:contra_unconf_L1} is proven directly.
\end{remark}

\begin{remark} By \eqref{eq:condition_LGamma1} and \eqref{eq:condition_LGamma2}, it holds $L_{\tilde{g}}\le \tilde{\kappa}/8$ and $L_{\tilde{g}}\le \tilde{\kappa}/16$, respectively. Hence, contraction is proven for $\tilde{b}$ being a small perturbation of a linear function. Further, the contraction rate is maximized for $\gamma =2 \sqrt{\tilde{\kappa}u}$.
\end{remark}


\begin{remark}
Note that the underlying distance $\tilde{r}$ is defined similarly as $r_l$ in \eqref{eq:r_l} and coincides with $\rho$ defined in \eqref{eq:rho} if $\tilde{K}=K$, $\sigma=\tau$ and $\mathcal{K}=\{(0,0)\}$. Moreover, $\tilde{r}$ is equivalent to the Euclidean distance on $\mathbb{R}^{2d}$, i.e.,
\begin{equation} \label{eq:norm_equiv}
\begin{aligned} 
\min(\tilde{\kappa}u/2,1/4)\gamma^{-2}(|x-\bar{x}|+|y-\bar{y}|)^2 &\leq\min(\tilde{\kappa}u,1/2)\gamma^{-2}|(x,y)-(\bar{x},\bar{y})|^2\leq \tilde{r}((x,y),(\bar{x},\bar{y}))^2
\\ & \leq \max(L_{\tilde{K}}u\gamma^{-2}+1,(3/2)\gamma^{-2})|(x,y)-(\bar{x},\bar{y})|^2 
\\ & \leq \max(L_{\tilde{K}}u\gamma^{-2}+1,(3/2)\gamma^{-2})(|x-\bar{x}|+|y-\bar{y}|)^2.
\end{aligned}
\end{equation}
\end{remark}


\subsection{Uniform in time propagation of chaos in the unconfined case}\label{sec:unconf_propofchaos}

Next, we establish uniform in time propagation of chaos bounds for the unconfined Langevin dynamics.
Fix $N\in\mathbb{N}$. 
We consider the functions $\hat{\rho}_N,\tilde{\rho}_N:\mathbb{R}^{2Nd}\times\mathbb{R}^{2Nd}\to[0,\infty)$ given by 
\begin{align} 
&\hat{\rho}_N((x,y),(\bar{x},\bar{y}))^2:=N^{-1}\sum_{i=1}^N\tilde{r}(\pi(x,y),\pi(\bar{x},\bar{y}))^2, && \text{and} \label{eq:hatrho_N}
\\ & \tilde{\rho}_N((x,y),(\bar{x},\bar{y})):=N^{-1}\sum_{i=1}^N\tilde{r}(\pi(x,y),\pi(\bar{x},\bar{y}))&& \text{for all } x,y,\bar{x},\bar{y}\in\mathbb{R}^{Nd},\label{eq:tilerho_N}
\end{align}
where $\tilde{r}$ is given in \eqref{eq:tilder} and $\pi:\mathbb{R}^{2Nd}\to\mathbb{R}^{2Nd}$ is given by
\begin{align} \label{eq:pi}
\pi(x,y)=\Big(x^i-N^{-1}\sum_{j=1}^N x^j, y^i-N^{-1}\sum_{j=1}^Ny^j\Big)_{i=1}^N \qquad \text{for } (x,y)\in\mathbb{R}^{2Nd}.
\end{align}
The function $\pi$ defines a projection from $\mathbb{R}^{2Nd}$ to the hyperplane $\msh^N=\{(x,y)\in\mathbb{R}^{2Nd}:(\sum_i x^i,\sum_i y^i)=0\}$. 
We note that distances $\hat{\rho}_N$ and $\tilde{\rho}_N$ 
are equivalent to $\tilde{\ell}_N^p$ given by
\begin{align} \label{eq:tildel}
\tilde{\ell}_N^p((x,y),(\bar{x},\bar{y}))=\ell_N^p(\pi(x,y),\pi(\bar{x},\bar{y})), \qquad \text{for all } x,y,\bar{x},\bar{y}\in\mathbb{R}^{Nd},
\end{align}
with $p=1$ and $p=2$, respectively.

\begin{theorem}[Propagation of chaos for unconfined Langevin dynamics in $L^2$ and $L^1$ Wasserstein distance]\label{thm:propofchaos_unconfined}
Suppose \Cref{ass:W} holds.
Let $\bar{\mu}_0$ and $\mu_0$ be two probability distributions on $\mathbb{R}^{2d}$ satisfying \Cref{ass:init_distrib}. 
For $t\geq 0$, let $\bar{\mu}_t$  be the law of the process $(\bar{X}_t,\bar{Y}_t)$, where $(\bar{X}_s,\bar{Y}_s)_{s\geq 0}$ is a solution to \eqref{eq:KFP_unconf} with initial distribution $\bar{\mu}_0$. Let $\mu_t^N$ be the law of $\{X_t^{i,N},Y_t^{i,N}\}_{i=1}^N$, where $(\{X_s^{i,N},Y_s^{i,N}\}_{i=1}^N)_{s\geq 0}$ is a solution to \eqref{eq:KFP_meanfield_nongradient} with $b=0$ and with initial distribution $\mu^N_0=\mu_0^{\otimes N}$.
If $L_{\tilde{g}}$ satisfies \eqref{eq:condition_LGamma1}, then
\begin{align*}
&\mathcal{W}_{2,\hat{\rho}_N}(\bar{\mu}_t^{\otimes N},\mu_t^N)\leq e^{-\hat{c}/2t}\mathcal{W}_{2,\hat{\rho}_N}(\bar{\mu}_0^{\otimes N},\mu_0^N)+\hat{c}^{-1/2}\mathcal{C}_3N^{-1/2} \qquad \text{and}
\\ &\mathcal{W}_{2,\tilde{\ell}_N^2}(\bar{\mu}_t^N,\mu_t^N)\leq \sqrt{2}M_3 e^{-\hat{c}/2t}\mathcal{W}_{2,\tilde{\ell}_N^2}(\bar{\mu}_0^N,\mu_0^N)+M_4\hat{c}^{-1/2}\mathcal{C}_3N^{-1/2},
\end{align*}
where $\hat{c}$, $\tilde{l}_N^2$  and $M_3$ are given in \eqref{eq:tildec}, \eqref{eq:tildel} and \eqref{eq:M_2}, respectively. The constant $M_4$ is given by
\begin{align} \label{eq:M5}
M_4=\gamma \max(\sqrt{2/\tilde{\kappa}},2).
\end{align}
and
$\mathcal{C}_3$ is a positive constant depending on $\gamma$,
$d$, $\tilde{\kappa}$, $L_{\tilde{K}}$, $L_{\tilde{g}}$, $u$ and on the second moment of $\bar{\mu}_0$.
If $L_{\tilde{g}}$ satisfies \eqref{eq:condition_LGamma2}, then
\begin{align*}
&\mathcal{W}_{1,\tilde{\rho}_N}(\bar{\mu}_t^{\otimes N},\mu_t^N)\leq e^{-\hat{c}t}\mathcal{W}_{1,\tilde{\rho}_N}(\bar{\mu}_0^{\otimes N},\mu_0^N)+\hat{c}^{-1}\mathcal{C}_4N^{-1/2} \qquad \text{ and}
\\ &\mathcal{W}_{1,\tilde{\ell}_N^1}(\bar{\mu}_t^{\otimes N},\mu_t^N)\leq \sqrt{2}M_3 e^{-\hat{c}t}\mathcal{W}_{1,\tilde{\ell}_N^1}(\bar{\mu}_0^{\otimes N},\mu_0^N)+M_4\hat{c}^{-1}\mathcal{C}_4N^{-1/2},
\end{align*}
where $\mathcal{C}_4$ is a positive constant depending on $\gamma$, $d$, $\tilde{\kappa}$, $L_{\tilde{K}}$, $L_{\tilde{g}}$, $u$ and on the second moment of $\bar{\mu}_0$.
\end{theorem}
\begin{proof}
The proof is postponed to \Cref{sec:proof_unconf_contr}.
\end{proof}

\begin{remark}
For $t\ge 0$, let $\mu_t^N$ and $\nu_t^N$ denote the law of $\{X_t^{i,N},Y_t^{i,N}\}_{i=1}^N$ and $\{{X_t'}^{i,N},{Y_t'}^{i,N}\}_{i=1}^N$, where the processes $(\{X_s^{i,N},Y_s^{i,N}\}_{i=1}^N)_{s\ge 0}$ and  $(\{{X_s'}^{i,N},{Y_s'}^{i,N}\}_{i=1}^N)_{s\ge 0}$ are solutions to \eqref{eq:KFP_meanfield_nongradient} with initial distributions $\mu_0^N$ and $\nu_0^N$, respectively, and for which \Cref{ass:W} is supposed. An easy adaptation of the proof of \Cref{thm:propofchaos} shows that if \eqref{eq:condition_LGamma1} holds, then 
\begin{align*}
\mathcal{W}_{2,\hat{\rho}_N}(\mu_t^N,\nu_t^N)\leq e^{-\hat{c}t}\mathcal{W}_{2,\hat{\rho}_N}(\mu_0^N, \nu_0^N)\qquad \text{and} \qquad
\mathcal{W}_{2,\tilde{\ell}_N^2}(\mu_t^N,\nu_t^N)\leq \sqrt{2}M_3 e^{-\hat{c}t}\mathcal{W}_{2,\tilde{\ell}_N^2}(\mu_0^N,\nu_0^N),
\end{align*}
and if \eqref{eq:condition_LGamma2} holds, then
\begin{align*}
&\mathcal{W}_{1,\tilde{\rho}_N}(\mu_t^N,\nu_t^N)\leq e^{-\hat{c}t}\mathcal{W}_{1,\tilde{\rho}_N}(\mu_0^N,\nu_0^N) \qquad \text{ and} \qquad \mathcal{W}_{1,\tilde{\ell}_N^1}(\mu_t^N,\nu_t^N)\leq \sqrt{2}M_3 e^{-\hat{c}t}\mathcal{W}_{1,\tilde{\ell}_N^1}(\mu_0^N,\nu_0^N),
\end{align*}
where $\hat{c}$ and  $M_3$ are given in \eqref{eq:tildec} and \eqref{eq:M_2}, respectively.
For the proof, a coupling of two copies of $N$ particle systems is constructed in the same line as \eqref{eq:KFP_coupling_unconfined}. As it will clarify by an inspection of the proof of \Cref{thm:propofchaos}, we can obtain a slightly better contraction rate in $L^2$ Wasserstein distance for the particle system compared to the rate in the propagation of chaos result.
\end{remark}

\subsection{Proof of Section~\ref{sec:contr_unconf} and Section~\ref{sec:unconf_propofchaos}}\label{sec:proof_unconf_contr}

\begin{proof}[Proof of \Cref{thm:contr_unconfined}]
Given two probability measures $\bar{\mu}_0,\bar{\nu}_0$ on $\mathbb{R}^{2d}$ and a $d$-dimensional Brownian motion $(B_t)_{t\ge 0}$,  we consider the synchronous coupling $((\bar{X}_t,\bar{Y}_t),(\bar{X}_t',\bar{Y}_t'))_{t\geq 0}$ of two copies of solutions to \eqref{eq:KFP_unconf} on $\mathbb{R}^{2d}\times\mathbb{R}^{2d}$ given by
\begin{equation} \label{eq:KFP_coupling_unconf_nonl}
\begin{aligned}
&\begin{cases} 
\rmd \bar{X}_t&=\bar{Y}_t \rmd t
\\  \rmd \bar{Y}_t&=(-\gamma\bar{Y}_t+u \int_{\mathbb{R}^d}\tilde{b}(\bar{X}_t,z)\bar{\mu}_t^x (\rmd z))\rmd t +\sqrt{2\gamma u}\rmd B_t, \qquad (\bar{X}_0,\bar{Y}_0)\sim \bar{\mu}_0,
\end{cases}
\\
&\begin{cases}
\rmd \bar{X}_t'&=\bar{Y}_t' \rmd t
\\  \rmd \bar{Y}_t'&=(-\gamma \bar{Y}_t'+u \int_{\mathbb{R}^d}\tilde{b}(\bar{X}_t',\tilde{z})\bar{\nu}_t^x (\rmd \tilde{z}))\rmd t+\sqrt{2\gamma u}\rmd B_t ,  \qquad (\bar{X}_0',\bar{Y}_0')\sim\bar{\nu}_0,
\end{cases}
\end{aligned}
\end{equation}
where $\bar{\mu}_t^x=\Law(\bar{X}_t)$, $\bar{\nu}_t^x=\Law(\bar{X}_t')$. We set $\tilde{Z}_t=\bar{X}_t-\bar{X}_t'$ and $\tilde{W}_t=\bar{Y}_t-\bar{Y}_t'$. By \Cref{ass:W} the process $(\tilde{Z}_t,\tilde{W}_t)_{t\geq 0}$ satisfies
\begin{equation} \label{eq:differenceproc_unconf_nonl}
\begin{aligned}
\begin{cases}
\rmd \tilde{Z}_t&=\tilde{W}_t \rmd t
\\ \rmd \tilde{W}_t&=(-\gamma \tilde{W}_t+u\int_{\mathbb{R}^d}\tilde{b}(\bar{X}_t,z)\bar{\mu}_t^x (\rmd z)-u\int_{\mathbb{R}^d}\tilde{b}(\bar{X}_t',\tilde{z})\bar{\nu}_t^x (\rmd \tilde{z}))\rmd t
\\ & =(-\gamma \tilde{W}_t-u\tilde{K}\tilde{Z}_t+ u \int_{\mathbb{R}^d}\tilde{g}(\bar{X}_t-z) \bar{\mu}_t(\rmd z)-u \int_{\mathbb{R}^d}\tilde{g}(\bar{X}_t'-\tilde{z}) \bar{\nu}_t(\rmd \tilde{z}))\rmd t,
\end{cases}
\end{aligned}
\end{equation}
%
%
where we used that $\mathbb{E}[\tilde{Z}_t]=0$, which holds by \Cref{ass:W} and \Cref{ass:init_distrib}.
Let $\tilde{A},\tilde{B},\tilde{C}\in\mathbb{R}^{d\times d}$ be positive definite matrices given by
\begin{align} \label{eq:matrix_dist_undef}
\tilde{A}=\gamma^{-2}u \tilde{K} +(1/2)(1-2\sigma)^2\mathrm{Id}, \qquad \tilde{B}=(1-2\sigma)\gamma^{-1}\mathrm{Id}, \qquad\text{and}\qquad \tilde{C}=\gamma^{-2}\mathrm{Id},
\end{align} 
where $\sigma$ is given by \eqref{eq:sigma}.
Then, by Ito's formula,
\begin{align*}
\frac{\rmd}{\rmd t}&(\tilde{Z}_t\cdot(\tilde{A}\tilde{Z}_t)+\tilde{Z}_t\cdot(\tilde{B}\tilde{W}_t)+\tilde{W}_t\cdot (\tilde{C}\tilde{W}_t) ) 
\\ & \le 2(\tilde{A}\tilde{Z}_t)\cdot\tilde{W}_t\rmd t +(\tilde{W}_t\cdot(\tilde{B}\tilde{W}_t)-(B\tilde{Z}_t)\cdot(\gamma\tilde{W}_t+u\tilde{K}\tilde{Z}_t))\rmd t-2(\tilde{C}\tilde{W}_t)\cdot(\gamma\tilde{W}_t+u\tilde{K}\tilde{Z}_t)\rmd t
\\ & \quad +L_{\tilde{g}}u|\tilde{B}\tilde{Z}_t+2\tilde{C}\tilde{W}_t|(|\tilde{Z}_t|+\mathbb{E}[|\tilde{Z}_t|])\rmd t
\\ &\leq \tilde{Z}_t\cdot((-u\tilde{K}\tilde{B})\tilde{Z}_t)+\tilde{Z}_t\cdot(2\tilde{A}-\gamma \tilde{B}-2u\tilde{K}\tilde{C})\tilde{W}_t+\tilde{W}_t\cdot((\tilde{B}-\gamma \tilde{C})\tilde{W}_t)
\\ & \quad+L_{\tilde{g}}u|\tilde{B}\tilde{Z}_t+2\tilde{C}\tilde{W}_t|(|\tilde{Z}_t|+\mathbb{E}[|\tilde{Z}_t|])
\\ & \le -2\sigma\gamma (\tilde{Z}_t\cdot(\tilde{A}\tilde{Z}_t)+\tilde{Z}_t\cdot(\tilde{B}\tilde{W}_t)+\tilde{W}_t\cdot (\tilde{C}\tilde{W}_t) ) 
+L_{\tilde{g}}u|\tilde{B}\tilde{Z}_t+2\tilde{C}\tilde{W}_t|(|\tilde{Z}_t|+\mathbb{E}[|\tilde{Z}_t|]),
\end{align*}
where we applied \eqref{eq:sigma} in the last step 
More precisely, it holds for all $z\in\mathbb{R}^d$
\begin{equation}\label{eq:estimate_sigma1}
\begin{aligned}
z\cdot((-u\tilde{K}(1-4\sigma)\gamma^{-1})z) \leq -(\tilde{\kappa}u/2)\gamma^{-1}|z|^2 \leq -\gamma\sigma|z|^2\leq -\gamma\sigma(1-2\sigma)^2|z|^2
\end{aligned}
\end{equation}
and therefore
$z\cdot((-u\tilde{K}(1-2\sigma)\gamma^{-1})z)\leq -2\gamma\sigma (\tilde{\kappa}u\gamma^{-2}+(1/2)(1-2\sigma)^2)|z|^2$.

Then for $\tilde{r}(t)=\tilde{r}((\bar{X}_t,\bar{Y}_t),(\bar{X}_t',\bar{Y}_t'))=(\tilde{Z}_t\cdot(\tilde{A}\tilde{Z}_t)+\tilde{Z}_t\cdot(\tilde{B}\tilde{W}_t)+\tilde{W}_t\cdot (\tilde{C}\tilde{W}_t))^{1/2} $ given in \eqref{eq:tilder},
\begin{align} \label{eq:unconf_contr_proof}
\rmd \tilde{r}(t)^2\leq -2\sigma\gamma \tilde{r}(t)^2\rmd t +L_{\tilde{g}}u\gamma^{-1}|(1-2\sigma)\tilde{Z}_t+2\gamma^{-1}\tilde{W}_t|(|\tilde{Z}_t|+\mathbb{E}[|\tilde{Z}_t|])\rmd t.
\end{align}
By taking expectation, it holds
\begin{align} \label{eq:unconf_W1_proof}
\frac{\rmd }{\rmd t} \mathbb{E}[ \tilde{r}(t)^2]\leq -2\sigma\gamma \mathbb{E}[\tilde{r}(t)^2]+L_{\tilde{g}}u\gamma^{-1}\mathbb{E}[|(1-2\sigma)\tilde{Z}_t+2\gamma^{-1}\tilde{W}_t|(|\tilde{Z}_t|+\mathbb{E}[|\tilde{Z}_t|])]. 
\end{align}
By \eqref{eq:condition_LGamma1}, \eqref{eq:sigma} and Young's inequality, we obtain for the last term
\begin{equation}\label{eq:apply_cond_Lg}
\begin{aligned}
L_{\tilde{g}}u\gamma^{-1}\mathbb{E}[|(1-2\sigma)\tilde{Z}_t&+2\gamma^{-1}\tilde{W}_t|(|\tilde{Z}_t|+\mathbb{E}[|\tilde{Z}_t|])]  \le \frac{\sigma\sqrt{\tilde{\kappa}u}}{2}\mathbb{E}[|(1-2\sigma)\tilde{Z}_t+2\gamma^{-1}\tilde{W}_t|(|\tilde{Z}_t|+\mathbb{E}[|\tilde{Z}_t|]]
\\ & \le \sigma \gamma \Big(\tilde{\kappa}u\gamma^{-2}\mathbb{E}[|\tilde{Z}_t|^2]+\frac{1}{4}\mathbb{E}[|(1-2\sigma)\tilde{Z}_t+2\gamma^{-1}\tilde{W}_t|^2]\Big)
\\ & \le \sigma \gamma \Big(\tilde{\kappa}u\gamma^{-2}\mathbb{E}[|\tilde{Z}_t|^2]+\frac{1}{2}\mathbb{E}[|(1-2\sigma)\tilde{Z}_t+\gamma^{-1}\tilde{W}_t|^2]+\frac{1}{2}\mathbb{E}[|\tilde{W}_t|^2]\Big)\le \sigma \gamma \mathbb{E}[\tilde{r}(t)^2].
\end{aligned}
\end{equation}
By inserting this bound in \eqref{eq:unconf_W1_proof}, we obtain by Gr\"{o}nwall's inequality, 
\begin{align*}
\mathcal{W}_{2,\tilde{r}}(\bar{\mu}_t,\bar{\nu_t})^2\le \mathbb{E}[\tilde{r}(t)^2]
\leq e^{-2\hat{c}t}\mathbb{E}[\tilde{r}(0)^2]
\end{align*}
with $\hat{c}$ given in \eqref{eq:tildec}.
By taking the square root and the infimum over all couplings $\omega\in\Pi(\bar{\mu}_0,\bar{\nu}_0)$, we obtain the first result in $L^2$ Wasserstein distance.
The second bound holds by \eqref{eq:norm_equiv} with $M_3$ given by \eqref{eq:M_2}.
To obtain contraction in $L^1$ Wasserstein distance, we take the square root in \eqref{eq:unconf_contr_proof},
\begin{align*}
\rmd \tilde{r}(t)&\le -\sigma\gamma\tilde{r}(t)\rmd t+L_{\tilde{g}}u\gamma^{-1}\frac{|(1-2\sigma)\tilde{Z}_t+2\gamma^{-1}\tilde{W}_t|}{2\tilde{r}(t)}(|\tilde{Z}_t|+\mathbb{E}[|\tilde{Z}_t|])\rmd t
\\ & \le -\sigma\gamma\tilde{r}(t)\rmd t+L_{\tilde{g}}u\gamma^{-1}(|\tilde{Z}_t|+\mathbb{E}[|\tilde{Z}_t|])\rmd t,
\end{align*}
where the last step holds by
\begin{equation} \label{eq:prefactor_interaction}
\begin{aligned} 
\frac{|(1-2\sigma)\tilde{Z}_t+2\gamma^{-1}\tilde{W}_t|}{2\tilde{r}(t)}&\leq \frac{1}{2}\Big(\frac{(1-2\sigma)^2|\tilde{Z}_t|^2+4(1-2\sigma)\gamma^{-1}\tilde{Z}_t\cdot\tilde{W}_t+4\gamma^{-2}|\tilde{W}_t|^2 }{(\tilde{\kappa}u\gamma^{-2}+(1/2)(1-2\sigma)^2)|\tilde{Z}_t|^2+(1-2\sigma)\gamma^{-1}\tilde{Z}_t\cdot\tilde{W}_t+\gamma^{-2}|\tilde{W}_t|^2}\Big)^{1/2}\leq 1.
\end{aligned}
\end{equation}
Taking expectation and applying \eqref{eq:condition_LGamma2} we obtain
\begin{align*}
\frac{\rmd}{\rmd t}\mathbb{E}[\tilde{r}(t)]&\le -\sigma\gamma \mathbb{E}[\tilde{r}(t)]+2L_{\tilde{g}}u\gamma^{-1}\mathbb{E}[|\tilde{Z}_t|] 
 \le -\sigma\gamma \mathbb{E}[\tilde{r}(t)]+\frac{\sigma\gamma}{2}\mathbb{E}[\sqrt{\tilde{\kappa}u\gamma^{-2}|\tilde{Z}_t|}]\le  -\frac{\sigma\gamma}{2} \mathbb{E}[\tilde{r}(t)].
\end{align*}
Hence by Gr\"{o}nwall's inequality,
\begin{align*}
\mathcal{W}_{1,\tilde{r}}(\bar{\mu}_t,\bar{\nu}_t)\leq e^{-\hat{c}t}\mathbb{E}[\tilde{r}(0)],
\end{align*}
where $\hat{c}$ is given in \eqref{eq:tildec}.
Taking the infimum over all couplings $\omega\in\Pi(\bar{\mu}_0,\bar{\nu}_0)$, we obtain the first bound in $L^1$ Wasserstein distance. The second bound follows by \eqref{eq:norm_equiv} with $M_3$ given in \eqref{eq:M_2}.
\end{proof}


To prove \Cref{thm:propofchaos_unconfined}, we establish a second moment bound of the solution to the nonlinear unconfined Langevin equation.

\begin{lemma}[Moment control for unconfined Langevin dynamics] \label{lem:momentbound_unconfined}
Suppose that \Cref{ass:W} and \eqref{eq:condition_LGamma1} hold. Let $(\bar{X}_t,\bar{Y}_t)_{t\geq 0}$ be a solution to \eqref{eq:KFP_unconf} with initial distribution satisfying \Cref{ass:init_distrib}. 
Then there exists a finite constant $\mathcal{C}_5>0$ such that 
\begin{align*}
\sup_{t\geq 0} \mathbb{E}[|\bar{X}_t|^2]\leq \mathcal{C}_5.
\end{align*}
The constant $\mathcal{C}_5$ depends on $\gamma$, 
$d$, $\tilde{\kappa}$, $L_{\tilde{g}}$, $u$ and on the second moment of the initial distribution.
\end{lemma}

\begin{proof}
As in the proof of \Cref{lem:momentbound}, we adapt the proof idea from \cite[Lemma 8]{DuEbGuZi20}. First, we note that by \Cref{ass:W} and \Cref{ass:init_distrib}, 
$\mathbb{E}[\bar{X}_t]=\mathbb{E}[\bar{Y}_t]=0$ for all $t\geq 0$, 
since by anti-symmetry of $\tilde{g}$
\begin{align*}
\frac{\rmd }{\rmd t}\mathbb{E}[\bar{X}_t]=\mathbb{E}[\bar{Y}_t] ,\quad \frac{\rmd}{\rmd t}\mathbb{E}[\bar{Y}_t]=-\gamma\mathbb{E}[\bar{Y}_t],
\end{align*} 
and $\mathbb{E}[\bar{X}_0]=\mathbb{E}[\bar{Y}_0]=0$.
Hence, $\bar{X}_t\cdot\mathbb{E}[\bar{X}_t] =\bar{Y}_t\cdot\mathbb{E}[\bar{X}_t]=0$. Further, we bound $|\mathbb{E}_{x\sim\bar{\mu}_t}[\tilde{g}(\bar{X}_t,x)]|\leq L_{\tilde{g}} (|\bar{X}_t|+\mathbb{E}[|\bar{X}_t|])$. 
By Ito's formula and \Cref{ass:W}, it holds for $\sigma\in(0,1/2)$,
\begin{align}
&\rmd (\gamma^{-2}u\bar{X}_t\cdot (\tilde{K}\bar{X}_t)+(1/2)|(1-2\sigma)\bar{X}_t+\gamma^{-1}\bar{Y}_t|^2+(1/2)\gamma^{-2}|\bar{Y}_t|^2) \nonumber
\\& \leq (2\gamma^{-2}u\bar{X}_t\cdot(\tilde{K}\bar{Y}_t)+(1-2\sigma)^2\bar{X}_t\cdot\bar{Y}_t) \rmd t  +(1-2\sigma)\gamma^{-1}(|\bar{Y}_t|^2-\bar{X}_t\cdot (u\tilde{K}\bar{X}_t)-\gamma\bar{X}_t\cdot\bar{Y}_t)\rmd t \nonumber
\\ & +\gamma^{-2}(-2\gamma|\bar{Y}_t|^2-2(u\tilde{K}\bar{Y})_t\cdot\bar{X}_t) \rmd t  +L_{\tilde{g}}u|(1-2\sigma)\gamma^{-1}\bar{X}_t+2\gamma^{-2}\bar{Y}_t|(|\bar{X}_t|+\mathbb{E}[|\bar{X}_t|]) \rmd t
\nonumber
\\ & +2\gamma^{-1} u d \rmd t+\sqrt{2\gamma^{-1}u }((1-2\sigma)\bar{X}_t+2\gamma^{-1}\bar{Y}_t) \rmd B_t \nonumber 
\\ & \leq -(1-2\sigma)\gamma^{-1}u\bar{X}_t\cdot (\tilde{K}\bar{X}_t)\rmd t  -2\sigma\gamma ((1-2\sigma)\gamma^{-1}\bar{X}_t\cdot\bar{Y}_t+\gamma^{-2}|\bar{Y}_t|^2)\rmd t  +2\gamma^{-1}ud \rmd t \nonumber
\\ & +L_{\tilde{g}}u|(1-2\sigma)\gamma^{-1}\bar{X}_t+2\gamma^{-2}\bar{Y}_t|(|\bar{X}_t|+\mathbb{E}[|\bar{X}_t|]) \rmd t+\sqrt{2\gamma^{-1}u}((1-2\sigma)\bar{X}_t+2\gamma^{-1}\bar{Y}_t) \rmd B_t. \nonumber
\end{align} 
Then by \eqref{eq:estimate_sigma1} we obtain after taking expectation
\begin{align*}
\frac{\rmd}{\rmd t}&\mathbb{E}[\gamma^{-2}u\bar{X}_t\cdot (\tilde{K}\bar{X}_t)+\frac{1}{2}|(1-2\sigma)\bar{X}_t+\gamma^{-1}\bar{Y}_t|^2+\frac{1}{2}\gamma^{-2}|\bar{Y}_t|^2]
\\ & \leq -2\sigma\gamma\mathbb{E}[\gamma^{-2}u\bar{X}_t\cdot (\tilde{K}\bar{X}_t)+\frac{1}{2}|(1-2\sigma)\bar{X}_t+\gamma^{-1}\bar{Y}_t|^2+\frac{1}{2}\gamma^{-2}|\bar{Y}_t|^2]+2\gamma^{-1} ud
\\ & \quad +L_{\tilde{g}}u\gamma^{-1}\mathbb{E}[|(1-2\sigma)\bar{X}_t+2\gamma^{-1}\bar{Y}_t|(|\bar{X}_t|+\mathbb{E}[|\bar{X}_t|])].
\end{align*}
By \eqref{eq:condition_LGamma1} and Young's inequality, we bound the last term similarly as \eqref{eq:apply_cond_Lg} by
\begin{align*}
L_{\tilde{g}}u\gamma^{-1}\mathbb{E}[|(1-2\sigma)\bar{X}_t+2\gamma^{-1}\bar{Y}_t|&(|\bar{X}_t|+\mathbb{E}[|\bar{X}_t|])] 
\\ & \le \sigma\gamma\Big(\tilde{\kappa}u\gamma^{-2}\mathbb{E}[|\bar{X}_t|^2]+\frac{1}{2}\mathbb{E}[|(1-2\sigma)\bar{X}_t+\gamma^{-1}\bar{Y}_t|^2]+\frac{1}{2}\mathbb{E}[\bar{Y}_t|^2]\Big).
\end{align*}
Hence,
\begin{align*}
\frac{\rmd}{\rmd t}&\mathbb{E}\Big[\gamma^{-2}u\bar{X}_t\cdot (\tilde{K}\bar{X}_t)+\frac{1}{2}|(1-2\sigma)\bar{X}_t+\gamma^{-1}\bar{Y}_t|^2+\frac{1}{2}\gamma^{-2}|\bar{Y}_t|^2\Big]
 \\ & \leq -\sigma\gamma\mathbb{E}\Big[\gamma^{-2}u\bar{X}_t\cdot (\tilde{K}\bar{X}_t)+\frac{1}{2}|(1-2\sigma)\bar{X}_t+\gamma^{-1}\bar{Y}_t|^2+\frac{1}{2}\gamma^{-2}|\bar{Y}_t|^2\Big]+2\gamma^{-1}ud.
\end{align*}
Then by Gr\"{o}nwall's inequality, there exists a constant $\mathbf{C}$ such that
\begin{align*}
\sup_{t\geq0}\mathbb{E}[\gamma^{-2}u\bar{X}_t\cdot (\tilde{K}\bar{X}_t)\frac{1}{2}|(1-2\sigma)\bar{X}_t+\gamma^{-1}\bar{Y}_t|^2+\frac{1}{2}\gamma^{-2}|\bar{Y}_t|^2]\leq \mathbf{C}<\infty
\end{align*}
and we obtain the result for $\mathcal{C}_5=\mathbf{C}/({\tilde{\kappa}}\gamma^{-2}u)$.
\end{proof}

\begin{proof}[Proof of \Cref{thm:propofchaos_unconfined}]
We consider a synchronous coupling approach of solutions to \eqref{eq:KFP_unconf} and \eqref{eq:KFP_meanfield_nongradient} with $b\equiv0$. Fix $N\in\mathbb{N}$. Let $\{(B^{i}_t)_{t\geq 0}\}_{i=1}^N$ be $N$ independent $d$-dimensional Brownian motions and let $\mu_0$ and $\bar{\mu}_0$ be two probability measrues on $\mathbb{R}^{2d}$.
The coupling $(\{(\bar{X}_t^i,\bar{Y}_t^i),(X_t^{i},Y_t^{i})\}_{i=1}^N)_{t\geq 0}$ of $N$ copies of a solution to \eqref{eq:KFP_unconf} and a solution to \eqref{eq:KFP_meanfield_nongradient} with $b\equiv 0$ is given on $\mathbb{R}^{2Nd}\times\mathbb{R}^{2Nd}$ by
\begin{equation} \label{eq:KFP_coupling_unconfined}
\begin{aligned}
&\begin{cases} 
\rmd \bar{X}_t^i&=\bar{Y}_t^i \rmd t
\\  \rmd \bar{Y}_t^i&=(-\gamma\bar{Y}_t^i+u\int_{\mathbb{R}^d}\tilde{b}(\bar{X}_t^i,z)\bar{\mu}_t^x (\rmd z))\rmd t +\sqrt{2\gamma u}\rmd B_t^{i}, \qquad (\bar{X}_0^i,\bar{Y}_0^i)\sim\bar{\mu}_0,
\end{cases}
\\
&\begin{cases}
\rmd X_t^{i}&=Y_t^{i} \rmd t
\\  \rmd Y_t^{i}&=(-\gamma Y_t^{i}+uN^{-1}\sum_{j=1}^N\tilde{b}(X_t^{i},X_t^{j}))\rmd t+\sqrt{2\gamma u}\rmd B_t^{i},  \qquad (X_0^i,Y_0^i)\sim \mu_0
\end{cases}
\end{aligned}
\end{equation}
for $i=1,...,N$, where $\bar{\mu}_t^x=\Law(\bar{X}_t^i)$ for all $i$.
For simplicity, we omitted the parameter $N$ in the index of $(X_t^i,Y_t^i)$ in the particle model. We set $\tilde{Z}_t^i=\bar{X}_t^i-X_t^{i}-N^{-1}\sum_{j=1}^N(\bar{X}_t^j-X_t^{j})$ and $\tilde{W}_t^i=\bar{Y}_t^i-Y_t^{i}-N^{-1}\sum_{j=1}^N(\bar{Y}_t^j-Y_t^{j})$. By \Cref{ass:W}, the process $(\{\tilde{Z}_t^i,\tilde{W}_t^i\}_{i=1}^N)_{t\geq 0}$ satisfies 
\begin{equation} \label{eq:differenceproc_unconfined}
\begin{aligned}
\begin{cases}
\rmd \tilde{Z}_t^i&=\tilde{W}_t^i \rmd t
\\ \rmd \tilde{W}_t^i&=-\gamma \tilde{W}_t^i\rmd t +u\Big(\int_{\mathbb{R}^d}\tilde{b}(\bar{X}_t^i,z)\bar{\mu}_t^x (\rmd z)-N^{-1}\sum_{j=1}^N\int_{\mathbb{R}^d}\tilde{b}(\bar{X}_t^j,\tilde{z})\bar{\mu}_t^x (\rmd \tilde{z})
\\ & \quad -N^{-1}\sum_{j=1}^N\tilde{b}(X_t^{i},X_t^{j})
 +N^{-2}\sum_{j,k=1}^N\tilde{b}(X_t^{j},X_t^{k})\Big)\rmd t
\\ & =-\gamma \tilde{W}_t^i\rmd t+u\Big(-\tilde{K}\tilde{Z}_t^i+N^{-1}\sum_{j=1}^N(\tilde{g}(\bar{X}_t^i-\bar{X}_t^j)-\tilde{g}(X_t^i-X_t^j))+\tilde{A}_t^i+N^{-1}\sum_{j=1}^N\tilde{A}_t^j\Big)\rmd t,
\end{cases}
\end{aligned}
\end{equation}
%
where $\tilde{A}_t^k=\int_{\mathbb{R}^d}\tilde{b}(\bar{X}_t^k,z)\bar{\mu}_t^x(\rmd z)-N^{-1}\sum_{j=1}^N\tilde{b} (\bar{X}_t^k,\bar{X}_t^j)$ for all $k=1,...,N$.
Hence, for the positive definite matrices $\tilde{A}, \tilde{B}, \tilde{C}$ given in \eqref{eq:matrix_dist_undef},
we obtain for $i=1,...,N$,
\begin{align*}
\rmd &(\tilde{Z}_t^i\cdot (\tilde{A}\tilde{Z}_t^i)+\tilde{Z}_t^i\cdot(\tilde{B}\tilde{W}_t^i)+\tilde{W}_t^i\cdot(\tilde{C}\tilde{W}_t^i)
\\ & \le  2\tilde{Z}_t^i\cdot(\tilde{A}\tilde{W}_t^i)\rmd t+( \tilde{W}_t^i\cdot(\tilde{B}\tilde{W}_t^i)-\gamma\tilde{Z}_t^i\cdot(\tilde{B}\tilde{W}_t^i)-(\tilde{B}\tilde{Z}_t^i)\cdot(u \tilde{K}\tilde{W}_t^i))\rmd t +(\tilde{C}W_t^i)\cdot(-2\gamma W_t^i-2u \tilde{K}\tilde{Z}_t^i)\rmd t 
\\ & + |\tilde{B}\tilde{Z}_t^i+2\tilde{C}\tilde{W}_t^i|u \Big(L_{\tilde{g}}N^{-1}\sum_{j=1}^N(|\tilde{Z}_t^j|+|\tilde{Z}_t^i|)+A_t^i+N^{-1}\sum_{i=j}^NA_t^j\Big)\rmd t
\\& \le \Big(-(u \tilde{K}\tilde{Z}_t^i)\cdot(\tilde{B}\tilde{Z}_t^i)+\tilde{Z}_t^i\cdot((2\tilde{A}-\gamma\tilde{B}-2u \tilde{K}\tilde{C})\tilde{W}_t^i)+\tilde{W}_t^i\cdot((\tilde{B}-2\gamma\tilde{C})\tilde{W}_t^i)\Big)\rmd t
\\ & + |\tilde{B}\tilde{Z}_t^i+2\tilde{C}\tilde{W}_t^i|u\Big(L_{\tilde{g}}N^{-1}\sum_{j=1}^N(|\tilde{Z}_t^j|+|\tilde{Z}_t^i|)+A_t^i+N^{-1}\sum_{i=j}^NA_t^j\Big)\rmd t,
\end{align*}
where $A_t^k=|\int_{\mathbb{R}^d}\tilde{b}(\bar{X}_t^k,z)\bar{\mu}_t^x(\rmd z)-N^{-1}\sum_{j=1}^N\tilde{b} (\bar{X}_t^k,\bar{X}_t^j)|$ for all $k=1,...,N$.
Then by \eqref{eq:unconf_contr_proof} for $\tilde{r}^i(t)=\tilde{r}((\bar{X}_t^i,\bar{Y}_t^i),(X_t^i,Y_t^i))$
\begin{equation}\label{eq:proof_propofchaos_unconf2}
\begin{aligned}
\rmd & \tilde{r}^i(t)^2  = \rmd (\tilde{Z}_t^i\cdot (\tilde{A}\tilde{Z}_t^i)+\tilde{Z}_t^i\cdot(\tilde{B}\tilde{W}_t^i)+\tilde{W}_t^i\cdot (\tilde{C}\tilde{W}_t^i))
\\ & \leq -2\sigma\gamma \tilde{r}^i(t)^2\rmd t+\gamma^{-1} |(1-2\sigma)\tilde{Z}_t^i+2\gamma^{-1}\tilde{W}_t^i|\Big(L_{\tilde{g}}N^{-1}\sum_{j=1}^N(|\tilde{Z}_t^j|+|\tilde{Z}_t^i|)+A_t^i+N^{-1}\sum_{i=j}^NA_t^j\Big)\rmd t
\end{aligned}
\end{equation}
and hence, for $\hat{\rho}_t:=\hat{\rho}_N((X_t,Y_t),(\bar{X}_t,\bar{Y}_t))$ given in \eqref{eq:hatrho_N},
\begin{align} \label{eq:proof_propofchaos_unconf}
\rmd \hat{\rho}_t\leq -2\sigma\gamma \hat{\rho}_t\rmd t+\frac{u}{\gamma}
N^{-1}\sum_{i=1}^N\Big( |(1-2\sigma)\tilde{Z}_t^i+2\gamma^{-1}\tilde{W}_t^i|\Big(L_{\tilde{g}}N^{-1}\sum_{j=1}^N(|\tilde{Z}_t^j|+|\tilde{Z}_t^i|)+A_t^i+N^{-1}\sum_{j=1}^NA_t^j\Big)\Big)\rmd t.
\end{align}
For the last term, we obtain by \eqref{eq:condition_LGamma1} and Young's inequality
\begin{align*}
L_{\tilde{g}}u\gamma^{-1}\frac{1}{N^2}\sum_{i,j=1}^N|(1-2\sigma)\tilde{Z}_t^i&+2\gamma^{-1}\tilde{W}_t^i|(|\tilde{Z}_t^j|+|\tilde{Z}_t^i|)
\le \sigma\gamma \hat{\rho}_t
\end{align*}
similarly as in \eqref{eq:apply_cond_Lg} and 
\begin{align*}
\frac{u}{\gamma}\frac{1}{N^2}\sum_{i,j=1}^N|(1-2\sigma)\tilde{Z}_t^i+2\gamma^{-1}\tilde{W}_t^i|(A_t^i+A_t^j)
&\le \frac{\sigma\gamma}{2}\frac{1}{N}\sum_{i=1}^N\Big( \frac{1}{4}|(1-2\sigma)\tilde{Z}_t^i+2\gamma^{-1}\tilde{W}_t^i|^2\Big)+\frac{8u^2}{\gamma^3\sigma}\frac{1}{N}\sum_{i=1}^N (A_t^i)^2
\\ & \le \frac{\sigma\gamma}{2}\hat{\rho}_t+\frac{8}{\gamma^3\sigma}\frac{1}{N}\sum_{i=1}^N (A_t^i)^2.
\end{align*}
Inserting these estimates in \eqref{eq:proof_propofchaos_unconf} and taking expectation yields
\begin{align*}
\frac{\rmd}{\rmd t} \mathbb{E}[\hat{\rho}_t]\leq -\frac{\sigma\gamma}{2} \mathbb{E}[\hat{\rho}_t]+\frac{8}{\gamma^3\sigma}\frac{1}{N}\sum_{i=1}^N \mathbb{E}[(A_t^i)^2].
\end{align*}
We bound $\mathbb{E}[{A_t^i}^2]$ similar as in the proof of \Cref{thm:propofchaos}. Note that by \Cref{ass:W}, $\tilde{b}$ is Lipschitz continuous with a Lipschitz constant which is bounded from above by $L_{\tilde{K}}+L_{\tilde{g}}$. Hence, \eqref{eq:boundA_t^i1} and \eqref{eq:boundA_t^i2} hold here with $L_{\tilde{K}}+L_{\tilde{g}}$ instead of $\tilde{L}$.
Then,
\begin{align*}
\mathbb{E}[{A_t^i}^2]&\leq \mathbb{E}\Big[\Big|\int_{\mathbb{R}^d}\tilde{b}(\bar{X}_t^i,z)\bar{\mu}_t^x(\rmd z)-\frac{1}{N}\sum_{j=1}\tilde{b}(\bar{X}_t^i,\bar{X}_t^j)\Big|^2\Big]
 \leq \frac{16(L_{\tilde{K}}+L_{\tilde{g}})^2}{N}\int_{\mathbb{R}^d}|x|^2\bar{\mu}_t(\rmd x).
\end{align*}
By \Cref{lem:momentbound_unconfined}, there exists a constant $\mathcal{C}_6$ depending on $\gamma$, $\mathbb{E}[|\bar{X}_0|^2+|\bar{Y}_0|^2]$, $d$, $\tilde{\kappa}$, $L_{\tilde{K}}$, $L_{\tilde{g}}$, $u$ such that for $N\geq 2$ and $i=1,...,N$,
\begin{align*}
\sup_{t\geq 0}\mathbb{E}[{A_t^i}^2]\leq \mathcal{C}_6N^{-1}.
\end{align*}
Hence,
\begin{align*}
\frac{\rmd}{\rmd t}\mathbb{E}[\hat{\rho}_t^2]\leq -2\sigma\gamma\mathbb{E}[\hat{\rho}_0^2]+ \mathcal{C}_3^2N^{-1}/2,
\end{align*}
where $\mathcal{C}_3^2=\frac{8 u^2}{\sigma\gamma^3}\mathcal{C}_6$.
By Gr\"{o}nwall's inequality, 
\begin{align*}
\mathcal{W}_{2,\hat{\rho}_N}(\Law(X_t^1,...,X_t^N),(\bar{\mu}_t)^{\otimes N})^2\le\mathbb{E}[\hat{\rho}_t^2]\leq e^{-\hat{c}t}\mathbb{E}[\hat{\rho}_0^2]+\hat{c}^{-1}\mathcal{C}_3^2N^{-1}
\end{align*}
with $\hat{c}$ given in \eqref{eq:tildec}.
By taking the infimum over all couplings $\omega\in\Pi(\mu_0^N,\bar{\mu}_0^{\otimes N})$, we obtain the first result in $L^2$ Wasserstein distance. The second bound holds by \eqref{eq:norm_equiv} with $M_3$ and $M_4$ given by \eqref{eq:M_2} and \eqref{eq:M5}, respectively.
To obtain the bound in $L^1$ Wasserstein distance, we note that by \eqref{eq:proof_propofchaos_unconf2}
\begin{align} 
\rmd \tilde{r}^i(t)&=\frac{1}{2r^i(t)}\rmd \tilde{r}^i(t)^2\leq -\sigma\gamma\tilde{r}^i(t)\rmd t +\frac{|(1-2\sigma)\tilde{Z}_t^i+2\gamma^{-1}\tilde{W}_t^i|}{2\gamma\tilde{r}^i(t)}u\Big(\frac{L_{\tilde{g}}}{N}\sum_j(|\tilde{Z}_t^j|+|\tilde{Z}_t^i|)+A_t^i+\frac{1}{N}\sum_{j=1}^NA_t^j\Big)\rmd t \nonumber
\\ & \leq  -\sigma\gamma\tilde{r}^i(t)\rmd t +\gamma^{-1}u\Big(\frac{L_{\tilde{g}}}{N}\sum_j(|\tilde{Z}_t^j|+|\tilde{Z}_t^i|)+A_t^i+\frac{1}{N}\sum_{j=1}^NA_t^j\Big)\rmd t, \nonumber
\end{align}
where the last step holds by \eqref{eq:prefactor_interaction}.
By summing over $i$ and taking expectation, we obtain by \eqref{eq:condition_LGamma2} for $\tilde{\rho}_t:=\tilde{\rho}_N((X_t,Y_t),(\bar{X}_t,\bar{Y}_t))$ given in \eqref{eq:tilerho_N},
\begin{align*}
\frac{\rmd}{\rmd t}\mathbb{E}[\tilde{\rho}_t]\leq -\sigma\gamma/2\mathbb{E}[\tilde{\rho}_t]+\gamma^{-1}uN^{-1}\sum_{i=1}^N\mathbb{E}[A_t^i].
\end{align*}
By \Cref{ass:W} and \Cref{lem:momentbound_unconfined}, there exists a constant $\mathcal{C}_4$ depending on $\gamma$, $\mathbb{E}[|\bar{X}_0|^2+|\bar{Y}_0|^2]$, $d$, $\tilde{\kappa}$, $L_{\tilde{K}}$, $u$ and $L_{\tilde{g}}$ such that
\begin{align*}
\sup_{t\ge 0} \mathbb{E}[A_t^i]\le \mathcal{C}_4\gamma N^{-1/2}
\end{align*}
similarly as in \eqref{eq:boundA_t^i}.  
Hence,
\begin{align*}
\frac{\rmd}{\rmd t}\mathbb{E}[\tilde{\rho}_t]\leq -\frac{\sigma\gamma}{2}\mathbb{E}[\tilde{\rho}_t]+\mathcal{C}_4N^{-1/2}.
\end{align*}
By Gr\"{o}nwall's  inequality,
\begin{align*}
\mathcal{W}_{1,\tilde{\rho}_N}(\bar{\mu}_t^{\otimes N}, \mu_t^N) \le\mathbb{E}[\tilde{\rho}_t]\leq e^{-\hat{c}t}\mathbb{E}[\tilde{\rho}_0]+\hat{c}^{-1}\mathcal{C}_4N^{-1/2}
\end{align*}
for $\hat{c}$ given in \eqref{eq:tildec}. 
Taking the infimum over all couplings $\omega\in\Pi(\bar{\mu}_0^{\otimes N},\mu_0^N)$, we obtain the  first result in $L^1$~Wasserstein distance. The second bound holds by \eqref{eq:norm_equiv} with $M_3$ and $M_4$ given in \eqref{eq:M_2} and \eqref{eq:M5}.

\end{proof}

}

 \paragraph*{Acknowledgments}
The author would like to thank her supervisor Andreas Eberle for bringing up the idea of glueing two metrics to combine two local contraction results and for his support and advice during the development of this work. \\
Support by the \textit{Hausdorff Center for Mathematics} has been gratefully acknowledged. 
Gef\"ordert durch die Deutsche Forschungsgemeinschaft (DFG) im Rahmen der Exzellenzstrategie des Bundes und der L\"ander - GZ 2047/1, Projekt-ID 390685813.

\bibliographystyle{plain}
\bibliography{bibliography2}

\end{document}